\documentclass[10pt,a4paper]{amsart}

\usepackage[margin=3cm]{geometry}

\usepackage[english]{babel}
\usepackage[T1]{fontenc}
\usepackage[utf8]{inputenc}

\usepackage{graphicx}
\usepackage{enumerate}
\usepackage{subfig}
\usepackage{array}
\usepackage{color}
\usepackage{stmaryrd}
\usepackage{amsmath}
\usepackage{amssymb}
\usepackage{latexsym}
\usepackage{mathrsfs}
\usepackage{amsthm}
\usepackage{mathtools}
\usepackage{hyperref}

\newtheorem{theorem}{Theorem}[section]
\newtheorem{corollary}[theorem]{Corollary}
\newtheorem{proposition}[theorem]{Proposition}
\newtheorem{lemma}[theorem]{Lemma}

\theoremstyle{definition}
\newtheorem{definition}[theorem]{Definition}

\theoremstyle{remark}
\newtheorem{remark}[theorem]{Remark}
\newtheorem{example}[theorem]{Example}

\newcommand{\e}{\mathrm e}

\newcommand{\R}{\mathbb R}
\renewcommand{\S}{\mathbb S}
\newcommand{\N}{\mathbb N}
\newcommand{\Z}{\mathbb Z}
\newcommand{\Span}{\mathrm{Span}}

\newcommand{\st}{\text{ s.t. }}

\newcommand{\K}{\mathcal{K}}
\newcommand{\KR}{\mathcal{K}}
\newcommand{\D}{\mathcal{D}}
\newcommand{\U}{\mathcal{U}}
\renewcommand{\O}{\mathcal{O}}
\newcommand{\C}{\mathcal C}

\newcommand{\eps}{\varepsilon}

\newcommand{\cZ}{\mathcal{Z}}
\newcommand{\SZ}{\mathfrak{Z}}

\newcommand{\Sigmax}{\bar{\Sigma}^x}

\DeclareMathOperator{\sre}{Exp}

\usepackage{accents}
\newcommand{\ubar}[1]{\underaccent{\bar}{#1}}

\newcommand{\lp}{\left(}
\newcommand{\rp}{\right)}
\newcommand{\lab}{\left|}
\newcommand{\rab}{\right|}
\newcommand{\lb}{\left[}
\newcommand{\rb}{\right]}
\newcommand{\lc}{\left\{}
\newcommand{\rc}{\right\}}
\newcommand{\ip}[2]{\left\langle#1,#2\right\rangle}

\newcommand{\dist}{d}

\newcommand*\diff{\mathop{}\!\mathrm{d}}

\newcommand{\Lap}{\Delta}

\newcommand{\E}{\mathbb{E}}
\newcommand{\Prob}{\mathbb{P}}

\DeclareMathOperator{\Iso}{Iso}
\DeclareMathOperator{\Var}{Var}

\newcommand{\bR}{\mathbb{R}}

\renewcommand{\mid}{:}

\usepackage[dvipsnames]{xcolor}

\DeclareMathOperator{\Cov}{Cov}
\DeclareMathOperator{\Cut}{Cut}
\newcommand{\Ind}{\boldsymbol 1}

\begin{document}

\title[sub-Riemannian heat kernel asymptotics]{Uniform, localized asymptotics for sub-Riemannian heat kernels, their logarithmic derivatives, and associated diffusion bridges}
\author{Robert W.\ Neel \and Ludovic Sacchelli}
\address{Department of Mathematics, Chandler-Ullmann Hall, Lehigh University, Bethlehem, Pennsylvania, USA}
\email{robert.neel@lehigh.edu}
\address{Inria, Universit\'e C\^ote d'Azur, CNRS, LJAD, MCTAO team, Sophia Antipolis, France}
\email{ludovic.sacchelli@inria.fr}

\begin{abstract} We show that the small-time asymptotics of the sub-Riemannian heat kernel, its derivatives, and its logarithmic derivatives can be localized, allowing them to be studied even on incomplete manifolds, under essentially optimal conditions on the distance to infinity. Continuing, away from abnormal minimizers, we show that the asymptotics are closely connected to the structure of the minimizing geodesics between the two relevant points (which is non-trivial on the cut locus). This gives uniform heat kernel bounds on compacts, and also allows a complete expansion of the heat kernel, and its derivatives, in a wide variety of cases.

The method extends naturally to logarithmic derivatives of the heat kernel, where we again get uniform bounds on compacts and a more precise expansion for any particular pair of points, in most cases. In particular, we determine the measure giving the law of large numbers for the corresponding diffusion bridge, and the leading terms of the logarithmic derivatives are given by the cumulants of geometrically natural random variables with respect to this measure. One consequence is that the non-abnormal cut locus is characterized by the behavior of the log-Hessian of the heat kernel.
\end{abstract}

\maketitle

\tableofcontents

\section{Introduction}

Small-time heat kernel asymptotics, and a variety of related matters, have a long history and a substantial literature, as we (partially) outline below. The object of this paper is to give a systematic development on general, possibly incomplete, sub-Riemannian manifolds of a method originally due to Molchanov \cite{Molchanov}, on compact Riemannian manifolds, to determine heat kernel asymptotics at points in the cut locus  by ``gluing together'' the asymptotics at non-cut points, and to apply this method to a broad spectrum of asymptotic questions. Most of our results are new in the sub-Riemannian context, and a couple are also, to the best of our knowledge, new in the Riemannian case as well. In this section, we state many of our main results, to give an indication of their nature and range. However, a number of interesting results are left to the body of the paper, since including all of them here would be too unwieldy.

\subsection{Sub-Laplacians and heat kernels}
Let $M$ be smooth (connected) manifold of dimension $d$, and let $\mu$ be a smooth volume on $M$. That is, $\mu$ is a measure on $M$ such that, in any (smooth) local coordinates $(u_1,\ldots,u_d)$ on a coordinate patch $U$, $\mu|_U$ has smooth, non-vanishing density with respect to Lebesgue measure $\diff u_1\cdots \diff u_d$ on $U$. The most efficient way to proceed is to introduce the sub-Laplacian and the sub-Riemannian metric together. We let $\Delta$ be a smooth, second-order differential operator on $M$ such that any point is contained in a coordinate patch $U$ on which
\begin{equation}\label{Eqn:Delta}
\Delta =\sum_{i=1}^k \cZ_i^2 + \cZ_0
\end{equation}
where $\cZ_0,\cZ_1,\ldots, \cZ_k$ are smooth vector fields and $\cZ_1,\ldots,\cZ_k$ are bracket-generating (this is the strong H\"ormander condition). In this situation, $\Delta$ induces a sub-Riemannian structure on $M$, which corresponds to the Carnot-Carath\'eodory distance in the older PDE literature. In the case where the distribution is of constant rank $k$ on $U$, with $2\leq k\leq d$, we can choose the $\cZ_1,\ldots,\cZ_k$ in Equation \eqref{Eqn:Delta} to be an orthonormal basis for the distribution at each point of $U$. That is, the span of $\cZ_1,\ldots,\cZ_k$ gives the distribution at each point, and the orthonormality induces an inner product on the distribution. The formalism to accommodate the rank-varying case is more elaborate, and we refer to Chapter 3 of \cite{ABB_2018} for a rigorous treatment of the construction of a sub-Riemannian structure from the principal part of $\Delta$. Nonetheless, for the purposes of this paper, the distribution and the inner product on it are almost never directly referenced; rather, the induced distance $d(\cdot,\cdot)$ and the structure of distance-minimizing curves are the central objects. In particular, the Chow-Rashevskii theorem shows that the distance between any two points of $M$ is finite and $(M,d)$ is a metric space such that the metric topology agrees with the topology of $M$. Moreover, one version of the Hopf-Rinow theorem for sub-Riemannian manifolds gives that if $(M,d)$ is complete as a metric space, there exists a minimizing curve between any two points of $M$ and $M$ is geodesically complete. (Again, we take \cite{ABB_2018} as the canonical reference for the basic results of sub-Riemannian geometry.)

So, going forward, $M$ is a $d$-dimensional (possibly incomplete) sub-Riemannian manifold equip\-ped with a smooth volume and a sub-Laplacian, in the above sense. Turning back to the operator $\Delta$, the corresponding diffusion $X_t$ is given by the Stratonovich SDE
\begin{equation}\label{Eqn:Xt}
\diff X_t = \sqrt{2} \sum_{i=1}^k \cZ_i(X_t)\circ \diff W_t^i + \cZ_0(X_t)\, \diff t
\end{equation}
where the $W^i_t$ are independent, standard Brownian motions and the process is killed upon (possible) explosion. We note explicitly that $X_t$ is a strong Markov process with continuous paths (we refer to Chapter 1 of \cite{Hsu} for the basics of SDEs on manifolds). Starting from $X_0=x$, this diffusion has a transition density $p_t(x,y)$ with respect to $\mu$, which is smooth for $(t,x,y)\in(0,\infty)\times M\times M$ by the celebrated H\"ormander theorem. From an analytic perspective, $p_t(x,y)$ is the (minimal) heat kernel associated to $\Delta$. In particular, $p_t(x,y)$ satisfies the heat equation $\partial_t p_t(x,y)=\Delta_x p_t(x,y)$, where the subscript indicates that the spatial derivatives are applied to the $x$-variable. (Note that we adopt the analyst's convention of using $\Delta$ rather than $(1/2)\Delta$ as the infinitesimal generator of $X_t$, and thus the SDE requires the $\sqrt{2}$-factor. The difference simply amounts to rescaling $t$ by a factor of 2.) We do not assume that $p_t(x,y)$ is symmetric, although that is an important special case, and we will give results specific to the symmetric case as appropriate. Notably, this includes the most important type of sub-Laplacian in sub-Riemannian geometry, namely the case when $\Delta$ is given as the $\mu$-divergence of the horizontal gradient (which generalizes the fact that on a Riemannian manifold, the Laplace-Beltrami operator can be written as the divergence of the gradient). The small-time asymptotics of $p_t(x,y)$ are the central topic for us.

We note that, in the Riemannian case, sub-Laplacians are exactly operators of the form $\Delta_{\mathrm{LB}}+\cZ_0$, where here $\Delta_{\mathrm{LB}}$ is the Laplace-Beltrami operator and $\cZ_0$ is a smooth vector field. Thus the heat kernels we consider in the Riemannian regime are more general than the standard heat kernel (corresponding to $\cZ_0\equiv 0$).

In sub-Riemannian geometry, extremal curves, by which we mean critical points of the length functional, can be normal or abnormal (or both). The Molchanov method is effective for pairs of points such that all minimizing geodesics are strongly normal (meaning no non-trivial subsegment is abnormal). In the properly sub-Riemannian situation, that is, not locally a Riemannian manifold, the trivial geodesic is always non-strictly abnormal, so the method does not apply on the diagonal. For broad classes of sub-Riemannian manifolds, such as contact manifolds, there are no non-trivial abnormals, in which case the diagonal is the only place where the method is ineffective. Of course, on-diagonal heat kernel asymptotics are a natural object of study (see \cite{DavideTrace}, for a sub-Riemannian example), but this requires other approaches, such as perturbation methods. Interpolating between the diagonal and off-diagonal asymptotics, say, to derive ``good'' uniform bounds on the heat kernel in small time, is in general a hard problem. For the case of the Heisenberg group and more general H-type groups, see \cite{Nate} and \cite{HQLi}. In the Riemannian case, there are no abnormals, so the method applies everywhere, including the diagonal.

There are situations where, for more specialized sub-Riemannian structures, one can find expressions for the heat kernel that allow the small-time asymptotics to be extracted in an explicit way. For example, for left-invariant structures on Lie groups, generalized Fourier transforms can be used, as developed in \cite{MashaMalva}. The sub-Riemannian model spaces are especially well studied and have a large literature, but we mention \cite{FabriceJing} and \cite{FabriceMichel} as two examples of explicit computation of the heat kernel and its small-time asymptotics on such spaces.

On the other hand, there are sub-Riemannian (and sub-Riemannian-adjacent) situations that go beyond the framework of this paper. For example, the Grushin plane is a sub-Riemannian structure, but the most natural Laplacian to put on it is only defined up to the singular set, and the first-order term blows up as the singular set is approached. Thus this is not a smooth sub-Laplacian, and indeed, the L\'eandre asymptotics fail dramatically. For recent work in this direction on Grushin and related structures, see \cite{UgoDario, BN-Grushin, Gallone1, Gallone2}.

\subsection{Localization and the Molchanov method}

In the first part of this paper, we rigorously establish the Molchanov method on general (not necessarily complete) sub-Riemannian manifolds, and show that it applies to derivatives of the heat kernel as well as the heat kernel itself. A central ingredient is to prove that the heat kernel can be restricted to appropriate compacts with only an exponentially negligible error, including for its derivatives. This requires several steps, and touches upon some related directions, which we now describe in more detail.

In what follows, $(Z^1,\dots,Z^m)$ denote an arbitrary family of smooth vector fields on $M$ such that at each point $x\in M$, $(Z^1(x),\dots,Z^m(x))$ spans the whole tangent space $T_xM$.
We call multi-index any finite (and possibly empty) sequence of integers $\alpha\in\{1,\dots, m\}^k$, with $k\in \N$. Then for any smooth function $f:M\to \R$, we denote
$$
Z^\alpha f=Z^{\alpha_k}\circ Z^{\alpha_{k-1}}\circ\cdots\circ Z^{\alpha_{1}} f.
$$
If $\alpha=\emptyset$, we intend that $Z^\alpha f=f$.
For $g:M^2\to \R$, $Z^\alpha_xg(x,y)$ and $Z^\alpha_yg(x,y)$  denote the derivatives with respect to the first and second space variable, respectively. The purpose of introducing such families of vector fields is to give coordinate-free statements about derivatives of the heat kernel. Of course, because we work on compacts, statements about derivatives with respect to a family of smooth vector fields can be reduced to statements about partial derivatives in finitely many local coordinate chart, and vice versa, and we will take advantage of this when convenient.

The Molchanov method has three ingredients. One is the Chapman-Kolmogorov equation (or the Markov property of the diffusion). The other two are a ``coarse'' estimate valid globally and a ``fine'' estimate valid away from the cut locus. In the complete sub-Riemannian case, the coarse estimate is essentially due to L\'eandre \cite{leandremaj,leandreminoration}. For a sub-Riemannian structure on $\bR^d$ given by smooth, bounded vector fields with bounded derivatives of all orders, he proved that
\[\begin{split}
\lim_{t\searrow 0} -4t \log p_t(x,y) &= d^2(x,y) \\
\text{and }\limsup_{t\searrow 0} 4t \log \lp \lab \frac{\partial^{\alpha}}{\partial y^{\alpha}} p_t(x,y)  \rab \rp &\leq - d^2(x,y)
\end{split}\]
for any multi-index $\alpha$, uniformly on compacts. (And note that these asymptotics hold without regard to abnormals or the cut locus.) The fine estimate is essentially due to Ben Arous \cite{BenArous}. For the same sub-Riemannian structures as L\'eandre, he proved that there are smooth functions $c_i(x,y)$ with $c_0(x,y)>0$ such that, for any $N$,
\[
p_t(x,y)= t^{-d/2} \e^{-\frac{d(x,y)^2}{4t}} \left(\sum_{k=0}^{N}c_k(x,y)t^k + t^{N+1}r_{N+1}(t,x,y) \right)
\]
where $r_{N+1}$ is an appropriate remainder term, uniformly on compact subsets of $M\times M$ that avoid the cut locus (and abnormals, including the diagonal). Further, this expansion can be differentiated as many times as desired in $t$, $x$, and $y$. Note that both L\'eandre and Ben Arous used the Euclidean volume to define their heat kernel, but it is an exercise in using the product rule to show that if either result holds for one smooth volume, then it holds for any smooth volume.

In reviewing the above results, not to mention those that follow, one might note that the first-order part (or sub-symbol) of the sub-Laplacian is relatively unimportant in the form of the expansion, as is the choice of smooth volume.  Indeed, the distance function, and thus the minimal geodesics, cut locus, etc.\ depend solely on the principal symbol of the sub-Laplacian. The first-order term and the choice of volume only affect the constants $c_k$ in the Ben Arous expansion, which are given by transport equations. This is unsurprising-- if the first order part lies in the distribution, one can think of it as contributing a Girsanov factor, and more generally, its effect is negligible at distant points in small time. Similarly, changing the smooth volume multiplies $p_t$ by a smooth, non-vanishing function. In the Riemannian case, where we write the operator as $\Delta_{\mathrm{LB}}+Z_0$ and use the Riemannian volume, the effect of $Z_0$ relative to the ``standard'' $Z_0=0$ case can be explicitly isolated as an action term, as can be found in the original paper of Molchanov \cite{Molchanov} (and continuing into some of the other references mentioned). Here we follow Ben Arous and allow the $c_i$ to account for matters.

Our first task is to establish L\'eandre asymptotics for $p_t$ and its derivatives on a general sub-Riemannian manifold, along with natural localization results. These results go hand-in-hand. Indeed, the principle of ``not feeling the boundary'' was invoked in \cite{Molchanov} (without proof). That the heat kernel on a complete sub-Riemannian manifold satisfies the L\'eandre and Ben Arous asymptotics has been something of a folk theorem, alluded to in the literature used without elaboration in \cite{BBN-BiHeis, BBCN-IMRN, BBN-JDG}, for example. A general localization result (under what we will call the strong localization condition below) was proven in \cite{HsuLocal}, showing that any diffusion on a manifold satisfying L\'eandre asymptotics for $p_t$ itself on compacts has the property that the asymptotics of $p_t$ are local. A quick (one sentence) reference is made to L\'eandre's result on sub-Riemannian manifolds, but one should not be too casual here. Proving that the L\'eandre asymptotics hold on a general manifold in the first place uses localization, so a careless approach ends up being circular. (The resolution is to build-up the result in stages, a version of which we carry out.) A similarly brief reference to localizing L\'eandre asymptotics (for $p_t$) on possibly incomplete sub-Riemannian manifolds, under the strong condition, is given in \cite{HsuIncomplete}, which explicitly treats the Riemannian case. In fact, the idea of adapting the Riemannian arguments to the sub-Riemannian case is already suggested by Azencott \cite{Azencott}, but this preceded the work of L\'eandre. Recently, Ballieul and Norris \cite{NorrisB} gave a rigorous proof of the L\'eandre asymptotics for $p_t$ itself on a possibly incomplete sub-Riemannian manifold. Their primary focus is working with incomplete manifolds, and especially establishing localization results related to what we will call below the weak localization condition. For this reason, they employ considerable analytic machinery (such as volume doubling estimates, a local Poincare inequality, a parabolic mean-value inequality, etc.), and it is not clear that these extend to derivatives of the heat kernel.

\begin{definition}\label{Def:Local}
We say that a compact subset $\KR$ of $M\times M$ is \emph{localizable} if it satisfies one of the following two conditions 
\begin{itemize}
\item \emph{Strong localization condition:}  For every $(x,y)\in \KR$, we have $d(x,y)< d(x,\infty) +d(y,\infty)$. (Here $d(\cdot,\infty)$ is the distance to infinity; see Section \ref{Sect:Local1}.) 
\item \emph{Weak localization condition:} There exists $\eps>0$ such that, for every $(x,y)\in \KR$, the set $\{z :d(x,z)+d(z,y)< d(x,y)+\eps\}$ has compact closure, and $\Delta$ satisfies the ``sector condition'' of Bailleul-Norris. This is a condition that limits the degree of asymmetry of $p_t(x,y)$ on all of $M$. We describe this in more detail in Section \ref{Sect:Local}, but note already that it includes the case when $p_t(x,y)$ is symmetric.
\end{itemize}
\end{definition}

Note that if $M$ is complete, $d(x,\infty)=\infty$ for all $x$, and thus any compact $\KR$ satisfies the strong localization condition. In particular, for complete $M$, our results hold for any compact. When we compute precise asymptotic expansions, we will consider $p_t(x,y)$ for a fixed pair of points $x$ and $y$. In this case, the associated $\KR$ is the singleton $\{(x,y)\}$, and we will say that $x$ and $y$ are localizable.

As indicated above, we show that the L\'eandre asymptotics for derivatives can be combined with localization results for heat kernel itself to allow for localization of derivatives. Namely, in Section \ref{Sect:Local3}, we prove

\begin{theorem}\label{THM:TrueLocal}
Let $M$ be a possibly-incomplete sub-Riemannian manifold with a smooth volume $\mu$ and a smooth sub-Laplacian $\Lap$, and let $p_t(x,y)$ be the corresponding heat kernel. Let $\KR\subset M\times M$ be compact and localizable. Then there exists an open set $U\subset M$ with compact closure and a $\delta>0$ such that, for any $(x,y)\in \KR$, both $x$ and $y$ are in $U$, and we have that
\begin{subequations}
\begin{align}
\lim_{t\searrow 0} 4t \log p_t(x,y) &= -d^2(x,y) \label{Eqn:1a} \\
\text{and}\quad\limsup_{t\searrow 0} 4t \log p_t\lp x,U^c ,y\rp & \leq -\lp d^2(x,y)+\delta\rp \label{Eqn:1b}
\end{align}
\label{eqn:all-lines1}
\end{subequations}
uniformly for $(x,y)\in \KR$, and, for any multi-index $\alpha$,
\begin{subequations}
\begin{align}
 \limsup_{t\searrow 0} 4t \log \lp \lab  Z_y^{\alpha}  p_t(x,y)  \rab \rp & \leq - d^2(x,y) \label{Eqn:2a} \\ 
\text{and}\quad
\limsup_{t\searrow 0} 4t \log \lp \lab  Z_y^{\alpha}  p_t\lp x,U^c ,y\rp  \rab \rp & \leq -\lp d^2(x,y)+\delta\rp \label{Eqn:2b}
\end{align}
\label{eqn:all-lines2}
\end{subequations}
uniformly for $(x,y)\in \KR$.
\end{theorem}

Here $p_t\lp x,U^c,y\rp$ denotes the contribution to $p_t(x,y)$ from paths that leave $U$; see Section \ref{Sect:Local1}. In \eqref{Eqn:2a} and \eqref{Eqn:2b}, nothing prevents $Z_y^{\alpha}  p_t$ from being zero, and more to the point, nothing prevents the left-hand side of either equation from being $-\infty$. But in this case, the inequality certainly holds.

Note that, in contrast to the work just mentioned, we establish both localization and the L\'eandre asymptotics not only for $p_t$ itself, but also for its derivatives (in $y$), thus extending L\'eandre's original result fully. This is interesting in its own right, but moreover, having bounds on the derivatives of the heat kernel is needed to apply Molchanov's method to heat kernel derivatives and also to study logarithmic derivatives of the heat kernel. In light of the Ben Arous expansion, one might wonder about taking derivatives in $t$ and $x$ as well. This is more complicated. Time derivatives are generally accessible by using the forward Kolmogorav (Fokker-Planck) equation to replace them with spatial derivatives. However, it turns out that, in the symmetric case, time derivatives can be controlled in a way that is compatible with the Molchanov method, and this allows for lower bounds on pure time derivatives, as we see in a moment. This is an interesting phenomenon for the most important special case, so we pursue it in what follows. Of course, in the symmetric case, one can also consider $x$-derivatives in place of $y$-derivatives. We discuss this in more detail below, in the context of the Ben Arous expansion.

\subsection{Uniform bounds and complete expansions for the heat kernel asymptotics}

In the second part of the paper, we consider more refined asymptotics than the log-scale asymptotics just discussed. In particular, with Theorem \ref{THM:TrueLocal} in hand, we show that the Ben Arous expansion holds on a general sub-Riemannian manifold, give uniform bounds for the heat kernel and its derivatives in small time, and show that Molchanov's method supports complete asymptotic expansions for both the heat kernel and its derivatives. 

Similar to the situation described above for the L\'eandre asymptotics, though less widely considered, the Ben Arous expansion, for the heat kernel itself, was assumed to generalize to complete manifolds in earlier works, but here we make this rigorous and extend it to incomplete manifolds. More precisely, we have the following.

\begin{definition}\label{Def:SN-CritSet}
A geodesic $\gamma:[0,T]\to M$ is said to be strongly normal if for every $[s,t]\in[0,T]$, $\gamma _{[s,t]}$ is not abnormal. Then the critical set $\C$ in $M^2$ is the set of pairs of points $(x,y)$ such that either
\begin{itemize}

\item There exists multiple length minimizing curves joining $x$ and $y$.

\item The unique geodesic joining $x$ and $y$ is conjugate.

\item The unique geodesic joining $x$ and $y$ is not strongly normal. Crucially, if $M$ is properly sub-Riemannian, this includes points on the diagonal $\mathcal{D}=\lc (x,x)\rc\subset M^2$, but not if $M$ is Riemannian.
\end{itemize}

\end{definition}

\begin{theorem}[Uniform Ben Arous expansion]\label{T:UnifBenArous}
For any multi-index $\alpha$, and $l$ a non-negative integer in the symmetric case and 0 otherwise,
there exist 
sequences of smooth functions $c_k:M^2\setminus\C\to \R$, $r_{k}: (0,+\infty)\times M^2\setminus\C\to \R$, $k\in \N$, such that for all $n\in \N$, for all $(x,y)\in M^2\setminus\C$, for all $t\in\R^+$
$$
\partial_t^l
Z_y^{\alpha}
p_t(x,y)
=
\frac{\e^{-\frac{d(x,y)^2}{4t}}}{t^{|\alpha|+2 l+d/2}} 
\left(
\sum_{k=0}^{n}c_k(x,y)t^k
+
t^{n+1}r_{n+1}(t,x,y)
\right).
$$
%The functions $c_k$, $r_k$ are smooth in $t$ and $y$, and, furthermore,
For any localizable compact $\K\subset M^2\setminus\C$, $l'$ which is any non-negative integer in the symmetric case and 0 otherwise, and any multi-index $\alpha'$, there exists $t_0$ such that
$$
\sup_{0<t<t_0}
\sup_{(x,y)\in \K}
\left|
	\partial_t^{l'}	Z_y^{\alpha'}r_{n+1}(t,x,y)
\right|<\infty.
$$

Additionally, if $\alpha=0$, then $c_0(x,y)>0$ on $M^2\setminus \C$.

\end{theorem}

\begin{remark}
Note that this is not the complete generalization of the original Ben Arous expansion; here we don't allow derivatives in the $x$ or $t$ variables in general, whereas that was allowed in \cite{BenArous}. Even in the symmetric case, where $y$-derivatives can be replaced by $x$-derivatives by symmetry, we don't allow the mixing of $x$- and $y$-derivatives. Indeed, it isn't obvious whether or not one should expect such results without some global control of the geometry. This occurs also in the Riemannian situation, as discussed in Remark 4 and in Section 6 of \cite{OurAIHP}.
\end{remark}

We next establish the general formula for the heat kernel asymptotics, valid at the non-abnormal cut locus, coming from Molchanov's method. The method is based on gluing together two copies of the Ben Arous expansion. For any two points $x,y\in M$, we denote by $\Gamma(x,y)$ the midpoint set of $(x,y)$, that is the set of points $z$ that lay at the midpoint of length minimizing curves between $x$ and $y$:
$$
\Gamma(x,y)=
\left\{
z\in M\mid  d(x,z)=d(z,y)=\frac{d(x,y)}{2}
\right\}.
$$
For any $\varepsilon>0$, we set
$$
\Gamma_\varepsilon(x,y)=
\left\{
z\in M
\mid  
d(x,z)\leq \frac{d(x,y)+\varepsilon}{2}
\text{ and }
d(y,z)\leq \frac{d(x,y)+\varepsilon}{2}
\right\}.
$$

When the context is clear, we typically write $\Gamma$ and $\Gamma_\varepsilon$ instead of $\Gamma(x,y)$ and $\Gamma_\varepsilon(x,y)$.
For any pair $(x,y)\in M^2$, we let the  \emph{hinged energy functional} be 
$$
h_{x,y}=\frac{d(x,\cdot)^2+d(\cdot,y)^2}{2} .
$$
Again, let $l$ be any non-negative integer in the symmetric case and 0 in the general case, and let $\alpha$ be any multi-index, Now let  $\Sigma^{l,\alpha}:\R^+\times M^2\setminus \C\to \R$ be the Taylor expansion type factor in the Ben Arous expansion of the heat kernel. That is, $\Sigma^{l,\alpha}$ is  the function such that 
$$
\Sigma_t^{l,\alpha}(x,y)=
t^{|\alpha|+2 l+d/2} \e^{\frac{d(x,y)^2}{4t}}
\partial_t^l
Z_y^{\alpha}
p_t(x,y).
$$

Naturally, as a consequence of Theorem~\ref{T:UnifBenArous}, $\Sigma^{l,\alpha}$ is smooth, and for any compact $\K\subset M^2\setminus\C$, $l'$ which is any non-negative integer in the symmetric case and 0 otherwise, and any multi-index $\alpha'$, there exists $t_0$ such that
$$
\sup_{0<t<t_0}
\sup_{(x,y)\in \K}
\left|
	\partial_t^{l'}	Z_y^{\alpha'}\Sigma_t^{l,\alpha}(x,y)
\right|<\infty.
$$

\begin{corollary}\label{C:MolchanovLaplace}
Let $\K$ be a localizable compact subset of $M^2\setminus \D$ such that all minimizers between pairs $(x,y)\in\K$ are strongly normal.
Then for any $\varepsilon>0$ small enough, we have uniformly on $\R^+\times \K$, for all $(t,x,y)\in\R^+\times \K$
$$
\partial_t^l
Z_y^{\alpha}
p_t(x,y)
=
\left(\frac{2}{t}\right)^{|\alpha|+2 l+d}
\int_{\Gamma_\varepsilon} 
\e^{-\frac{h_{x,y}(z)}{t}}
\Sigma_{t/2}^{0,0}(x,z)
\Sigma_{t/2}^{l,\alpha}(z,y)
\diff\mu( z)
+
O\left(\e^{-\frac{d(x,y)^2+\varepsilon^2/2}{4t}}\right).
$$
\end{corollary}

It turns out that, because the behavior of the exponential map away from the cut locus is qualitatively the same in sub-Riemannian as in Riemannian geometry, the resulting formula is structurally the same, giving the heat kernel via the Laplace asymptotics of a geometrically-motivated integral, namely a Laplace integral with phase $h_{x,y}$. The proof, however, requires working with sub-Riemannian formalism: accounting for the possibility of abnormal minimizers, defining the exponential map in terms of the Hamiltonian flow on the co-tangent space, and so on. This parallels the fact that the Ben Arous asymptotics on a sub-Riemannian manifold directly generalize the classical Minakshishundaram-Pleijel asmyptotics on a Riemannian manifold, but requires additional work to prove. Given this, the asymptotics of heat kernel are determined by the theory of the asymptotics of Laplace integrals, which is a well-developed subject in its own right.

The application of this theory to Riemannian heat kernel asymptotics began with \cite{Molchanov} and was systematically developed in \cite{CBell}, though only the leading term of the expansion for $p_t$ itself was considered. In \cite{BBN-BiHeis, BBCN-IMRN, BBN-JDG}, the analogous sub-Riemannian situation was considered, including the relationship of the leading term to the geodesic geometry. Here we show that one can always bound the leading term, leading to two-sided estimates on the heat kernel itself, and to upper bounds on its derivatives, uniformly on (localizable) compacts. These uniform bounds on the heat kernel, in the compact Riemannian case, are discussed in Chapter 5 of \cite{Hsu}. More recently, Ludewig \cite{Ludewig1} extended the upper bound to any number of simultaneous derivatives in $x$, $y$, and $t$ for formally self-adjoint Laplace-type operators acting on a vector bundle over a compact Riemannian manifold. This was based on wave parametrix techniques, and, as above, it isn't clear whether or not one should expect such results in a more general context. In the present situation, we have the following.

\begin{proposition}\label{P:universal_bounds}
Let $\K$ be a localizable compact subset of $M^2\setminus \D$ such that all minimizers between pairs $(x,y)\in \K$ are strongly normal. Then for $l$ any non-negative integer in the symmetric case and 0 otherwise, and any multi-index $\alpha$, there exists $C>0$ such that for all $(x,y)\in \K$,
$$
\partial_t^l
Z_y^{\alpha}
p_t(x,y)
\leq
\frac{C}{t^{|\alpha|+2 l}t^{d-1/2}}
 \e^{-\frac{d(x,y)^2}{4t}}.
$$
In the case $\alpha=0$ there also exists $C'>0$ such that for all $(x,y)\in \K$,
$$
\frac{C'}{t^{2 l}t^{d/2}}
 \e^{-\frac{d(x,y)^2}{4t}}
 \leq
\partial_t^l p_t(x,y).
$$
\end{proposition}

In the past several years, there has been interest in ``complete'' expansions, meaning expansions to arbitrary order in $t$, not simply the leading term. For example, \cite{Inahama} uses a distributional form of Malliavin calculus, due to Watanabe, to give (potentially) complete expansions of the heat kernel on the sub-Riemannian cut locus, under some stochastically-motived assumptions, while \cite{Ludewig2,Ludewig1} uses wave parametrix techniques to, among other things, give (potentially) complete expansions of the heat kernel of a self-adjoint Laplace-type operator acting on a vector bundle over a compact Riemannian manifold. Here, it is important to note that writing the heat kernel asymptotics in terms of a Laplace integral, which all of these approaches do, in one way or another, to obtain a complete expansion or even an exact leading term, is only possible if the asymptotics of the integral can be explicitly determined. If the hinged energy functional is real analytic in some coordinates, then the general theory of Laplace asymptotics, as developed by Arnold and collaborators \cite{Arnold_singularities_diff_maps1, Arnold_singularities_diff_maps2} guarantees an expansion in rational powers of $t$ and integer powers of $\log t$. We observe that Corollary \ref{C:MolchanovLaplace} supports complete expansions via this approach, and illustrate by giving them in two typical cases, when $h_{x,y}$ has $A$-type singularities and when it is Morse-Bott, in Sections
\ref{Sect:A_n} and \ref{Sect:MB}.

On the other hand, in the general smooth case, one can have situations in which this theory does not apply, and where even the leading order appears not to be known. In the Riemannian context, C.\ Bella\"iche \cite{CBell} discusses the possibility of non-analytic hinged energy functions and the resulting breakdown in computing the asymptotics of the Laplace integral. This paper, however, seems not to be widely known, and the work on complete expansions just mentioned only explicitly considers the Morse-Bott case (as does the earlier work of \cite{StroockKusuoka}). Constructing a Riemannian metric realizing such non-analytic hinged energy functions (and more generally, arbitrary normal forms), was considered by A.\ Bella\"iche \cite{ABell}, who stated the existence of such Riemannian metrics as a theorem and briefly sketched a construction. Here, we provide complete details of the proof and also consider the properly sub-Riemannian case. That is, in constructing a sub-Riemannian metric with prescribed singularities for the hinged energy function, one must take into account the constraint imposed by the distribution, and in principle one might wonder if that is an obstruction to constructing an arbitrary singularity. Given that there are an infinite number of possible growth vectors, addressing all of them is impractical, but we show that for contact manifolds (which is the most widely-studied class of sub-Riemannian manifolds), a similar construction is possible. (Other possible growth vectors are left to the interested and suitably intrepid reader.) These results are the content of Section \ref{Sect:Prescribed}. This indicates that properly sub-Riemannian manifolds exhibit the same diversity of possible singularities as Riemannian manifolds. (The generic situation in low dimensions is another story, for which one can see \cite{BBCN-IMRN}.) 
 
 \subsection{Asymptotics for log-derivatives and bridges}
In the third part of the paper, we turn to the asymptotics of logarithmic derivatives of the heat kernel and to the law of large numbers for the Brownian bridge, which are closely related and accessible to a natural modification of Molchanov's method. Probabilistically, this is corresponds to considering the bridge process, rather than the underlying diffusion itself. 

We express the leading term of the $n$th logarithmic derivative is given as an $n$th-order joint cumulant. In particular, we can define a family of probability measures $m_t$ on $\Gamma_{\eps}$ in terms of a ratio of Laplace integrals, see Equation \eqref{Eqn:M_tDef}, which are subsequentially compact, see Theorem \ref{THM:Tightness}. In terms of the $m_t$, we have the following expression for the log-derivatives of $p_t$.

\begin{theorem}\label{THM:Cumulants}
Let $x$ and $y$ be localizable and such that all minimal geodesics from $x$ to $y$ are strongly normal, and let $Z^1,\ldots, Z^N$ be smooth vector fields in a neighborhood of $y$ (so that we understand that they act as differential operators in the $y$-variable). Then
\[
Z^N\cdots Z^1 \log p_t(x,y)=
\lp -\frac{1}{t}\rp^N \lc \kappa^{m_t}\lp d(\cdot,y)Z^1d(\cdot,y),\ldots, d(\cdot,y)Z^N d(\cdot,y) \rp +O(t)\rc ,
\]
where $\kappa^{m_t}$ is the joint cumulant (of $N$ random variables) with respect to $m_t$.
\end{theorem}

The logarithmic gradient and logarithmic Hessian for compact Riemannian manifolds were treated in \cite{Neel}, but the higher-order derivatives are new even in the Riemannian case. Moreover,  we show that the (non-abnormal) cut locus is characterized by the blow up of the logarithmic Hessian, which was proven in the Riemannian case in \cite{Neel}; we also note that the proof presented here is much improved over that of \cite{Neel}. This is a differential analogue of the recent result of \cite{LucaDavide} showing that the cut locus is characterized by the square of the distance failing to be semi-convex. We have the following (see Section \ref{Sect:CutChar} for definitions and further details).

\begin{corollary}\label{Cor:CutLocus}
Let $x$ and $y$ be localizable and such that all minimal geodesics from $x$ to $y$ are strongly normal, and let $\SZ$ be a set of vector fields on a neighborhood of $y$ which is $C^1$-bounded and such that $\SZ|_{T_yM}$ contains a neighborhood of the origin. Then $y\not\in \Cut(x)$ if and only if
\[
\limsup_{t\searrow 0} \lb \sup_{Z\in \SZ} t\lab Z_y Z_y \log p_t(x,y)\rab \rb <\infty
\]
and $y\in \Cut(x)$ if and only if 
\[
\lim_{t\searrow 0} \lb \sup_{Z\in \SZ} t Z_y Z_y \log p_t(x,y) \rb =\infty
\]
\end{corollary}

Because of the uniformity of our approach (on localizable compacts), as a consequence of the above, we obtain bounds on the logarithmic derivatives on compacts (disjoint from any abnormals), which in turn imply bounds on derivatives of the heat kernel itself. These bounds were proven in the case of compact Riemannian manifolds by \cite{HsuLogDer} and \cite{StroockTuretsky}, via stochastic analysis. In the properly sub-Riemannian case, our compact must avoid the diagonal, so the distance function doesn't explicitly appear. Note that an extension of these bounds to complete (non-compact) Riemannian manifolds was given only recently in \cite{XM-GenComplete}, while the incomplete Riemannian case was established in \cite{OurAIHP}. (Note that the Riemannian results are stronger, somewhat easier to prove, and have a different context and tradition, making a separate treatment natural.)

Finally, we consider the small-time asymptotics of the bridge process, in particular, the law of large numbers.  The law of large numbers in the sub-Riemannian case when there is a single minimizer between $x$ and $y$ was established in \cite{NorrisB}; note that this includes the possibility that the minimizer is abnormal. (Such convergence to a point mass causes one to wonder about a central limit theorem result for the fluctuations, which they pursue in \cite{BMN} for the case of a non-conjugate geodesic. For an on-diagonal central limit theorem, see \cite{Habermann}.) The uniform version of this result serves as the basic ingredient in Molchanov's method, and we see that the small-time limit of the bridge process is also governed by the $m_t$. Let $\mu^{x,y,t}$ be the natural renormalized measure on pathspace of the bridge process from $x$ to $y$ in time $t$, and if $m_0$ is a probability measure on $\Gamma$, let $\tilde{m}_0$ denote its natural lift to pathspace (see Section \ref{Sect:LLN} for details).

\begin{theorem}\label{THM:LLN}
Let $x,y\in M$ be localizable and such that all minimizers from $x$ to $y$ are strongly normal. Then for any sequences of times $t_n\rightarrow 0$, $\mu^{x,y,t_n}$ converges if and only if $m_{t_n}$ does, and if so, letting $m_0$ denote the limit of $m_{t_n}$, we have $\mu^{x,y,t_n}\rightarrow \tilde{m}_0$.
\end{theorem}

The Riemannian version of this result (and the one below in Theorem \ref{THM:RA-LLN}) was established in \cite{HsuBridge} using a large deviation principle. Similar large deviation principles were recently given for some sub-Riemannian manifolds by Bailleul \cite{Bailleul-Bridge} and Inahama \cite{Inahama-LargeDeviations}, but a law of large numbers was not addressed (beyond this single minimizer case as discussed above). Our approach circumvents direct use of large deviations.

We observe that Theorems \ref{THM:Cumulants} and \ref{THM:LLN} are especially appealing when viewed together. The probability measures that govern the leading terms of the log-derivatives of $p_t$ are exactly the measures coming from the asymptotic behavior of the bridge process.

In the real-analytic case, the support of the limiting measure arising in the law of large numbers can be described. In particular, in this case, one can quantify the degree of degeneracy of $h_{x,y}$ at any point $z\in \Gamma$, which by extension we think of as the degree of degeneracy of the exponential map. There is a closed, non-empty subset $\Gamma^m$ of $\Gamma$ which corresponds to those points of ``maximum degeneracy.'' Then we have the following improvement to both the law of large numbers and the limit of the log-derivatives of $p_t$.

\begin{theorem}\label{THM:RA-LLN}
Let $x$ and $y$ be localizable and such that all minimizers between them are strongly normal, and suppose that around any point of $\Gamma$ there is a coordinate chart such that $h_{x,y}$ is real-analytic in these coordinates. (In particular, this holds if $M$ and $\Delta$ themselves are real-analytic.) Then $m_t$ converges to a limit $m_0$ as $t\searrow 0$, and the support of $m_0$ is exactly $\Gamma^m$. Further, the bridge measure $\mu^{x,y,t}$ converges to $\tilde{m}_0$ as $t\searrow 0$, and for $Z^i$ as in Theorem \ref{THM:Cumulants}, we have
\[
\lim_{t\searrow 0} t^N Z_y^N \cdots Z_y^1 \log p_{t}(x,y) = 
 \lp -\frac{d(x,y)}{2} \rp^N   \kappa^{m_0} \Big(Z_y^1d(\cdot,y),\ldots,  Z_y^N d(\cdot,y) \Big)
 \]
\end{theorem}

Finally, the limiting measure can be concretely determined in cases where the exponential map has a simple normal form, analogously to what we see for the asymptotic expansion of the heat kernel itself, and we treat the cases when $h_{x,y}$ has $A$-type singularities and when it is Morse-Bott in Sections \ref{Sect:A_n-LLN} and \ref{Sect:MB-LLN}.

\subsection{Acknowledgements}

The authors thank Ismael Bailleul and Karen Habermann for helpful discussions about the law of large numbers and Ugo Boscain for advice on sub-Riemannian technicalities (of which there are many). This work was partially supported by a grant from the Simons Foundation (\#524713 to Robert Neel).

\section{Localization of heat kernel derivatives}\label{Sect:Local}

Estimates on the small-time behavior of the heat kernel that can be used to localize its asymptotics have a long history, as already described. Localization of the $y$-derivatives of the heat kernel can be accomplished using a method described in Section 3 of \cite{OurAIHP}, which is based on combining earlier results of L\'eandre and Bailleul-Norris.  This section is devoted to describing these localization results.

\subsection{Background results}\label{Sect:Local1}
We begin by clarifying the basic definitions used in the localization conditions. Namely, for a closed set $A\subset M$, we have
\[\begin{split}
d(x,A) &= \inf\lc d(x,z): z\in A\rc \\
\text{and}\quad d(x,\infty) &= \sup\lc d(x, A) : \text{$A$ closed and $M \setminus A$ relatively compact}\rc .
\end{split}\]
We also let
\[
d(x, A, y) = \inf\{d(x, z) + d(z, y) : z \in A\},
\]
denote the distance from $x$ to $y$ via paths that hit $A$. Note that $d(x,\infty)$ is continuous in $x$, and, for any fixed $A$, $d(x,A)$ is continuous in $x$ and $d(x,A,y)$ is continuous in $(x,y)$.

Recall that the diffusion $X_t$ and its transition measure depend only on the operator $\Delta$, but that we choose some smooth reference measure $\mu$ in order to write the transition measure as $p_t(x,y)\, d\mu$. Observe that, in place of \eqref{Eqn:Delta}, we can instead write $\Delta$ in the form
\begin{equation}\label{Eqn:DivForm}
\Delta f=\mathrm{div}_{\mu}\lp \nabla f\rp +\hat{\cZ}_0(f)
\end{equation}
for some smooth vector field $\hat{\cZ}_0$, where $\nabla f$ denotes the horizontal derivative of a smooth function $f$ and $\mathrm{div}_\mu$ is the divergence (of a vector field) with respect to $\mu$. Here $\hat{\cZ}_0$ need not be the same as $\cZ_0$ in \eqref{Eqn:Delta}; indeed, the point is that the divergence of $\nabla$, for any smooth volume, gives an operator with the correct principle symbol, so that $\hat{\cZ}_0$ is whatever first-order term is necessary to reconcile \eqref{Eqn:DivForm} with \eqref{Eqn:Delta}. With this notation, the ``sector condition'' introduced by Bailleul-Norris \cite{NorrisB} is that
\begin{equation}\label{Eqn:SectorCon}
\sup_M \lab \hat{\cZ}_0\rab <\infty.
\end{equation}
Here the length $ \lab \hat{\cZ}_0\rab$ is understood with respect to the inner product on the horizontal distribution $\Span\lc \cZ_1\ldots, \cZ_k\rc$, which, in particular, means that $\hat{\cZ}_0$ must lie everywhere in this span. Note that $\mu$ is independent of $\Delta$, in the sense that it can be any smooth measure on $M$. Thus, the sector condition can be understood essentially as a condition on $\Delta$; namely, that there exists a smooth measure that ``almost symmetrizes'' $\Delta$, in the sense that it differs from a symmetric operator by a bounded vector field. If such a measure exists, one should presumably consider the heat kernel with respect to this measure, if one is interested in the best possible localization condition.

\begin{remark}
The weak localization condition of Definition \ref{Def:Local} is not a condition solely on $\KR$ but also on $\Delta$ on all of $M$, so neither localization condition implies the other in general. However, if we consider only the condition on the distance to infinity on $\KR$ (for example, if we restrict our attention to symmetric operators), then we see that the strong localization condition implies the weak one, which explains the choice of nomenclature.
\end{remark}

If $U$ is an open set, we let $p^U_t(x,y)$ be the heat kernel on $U$ (which should be understood with Dirichlet boundary conditions, corresponding to the associated diffusion being killed upon leaving $U$, which is the first hitting time of the closed set $U^c$). Since we consider possibly-incomplete sub-Riemannian manifolds, we could consider $U$ as a sub-Riemannian manifold in its own right, which is consistent with the above, but when we view $U$ as a subset of $M$, we extend $p_t^U(x,y)$ to be zero whenever $x$ or $y$ is in $U^c$. For a closed set $A$ (which, in light of the previous, can be thought of as $U^c$) we define $p_t(x,A,y) = p_t(x,y)-p^{A^c}_t(x,y)$, so that $p_t(x,A,y)$ gives the contribution to $p_t(x,y)$ from paths that hit $A$. Note that we have the fundamental decomposition $p_t(x,y)= p^U_t(x,y) + p_t(x,U^c,y)$ (which is non-trivial only for $x$ and $y$ both in $U$). 

We are now in a position to describe the conditions under which the heat kernel itself localizes. For the strong localization case, we have the following.
\begin{theorem}\label{THM:StrongCond}
Let $A\subset M$ be closed and suppose $M\setminus A$ has compact closure. Then for any compact subset $K$ of $M\setminus A$
\[
\limsup_{t\searrow 0} 4t \log p_t(x,A,y) \leq - \lp d(x,A)+d(y,A)\rp^2 
\]
uniformly for $x\in K$ and $y\in K$. Also, if $\KR$ is a compact subset of $\{(x,y)\in M\times M: d(x,y)< d(x,\infty) +d(y,\infty)\}$, then we have $4t\log p_t(x,y)\rightarrow -d^2(x,y)$ uniformly for $(x,y)\in \KR$.
\end{theorem}

This was essentially given in two papers of Hsu \cite{HsuIncomplete,HsuLocal} from the 90s, but the focus there was on the Riemannian versions, and complete details in the sub-Riemannian case were not given. A complete proof was given recently by Bailleul-Norris \cite{NorrisB}, but under the additional assumption that $Z_0$ lies in the span of $\cZ_1,\ldots,\cZ_k$. This is an artifact of their approach, which is designed to handle the weak localization condition, as indicated below. For completeness, we give a brief proof along the lines of Hsu in Appendix \ref{App:Local}.

In the weak localization case, we have the following variant of the previous, which combines Theorems 1.1 and 1.2 of \cite{NorrisB} in the case when the sector condition holds.
\begin{theorem}[Bailleul-Norris]\label{THM:WeakCond}
Suppose that $(M,\Delta,\mu)$ satisfies the sector condition \eqref{Eqn:SectorCon}. Let $A\subset M$ be closed such that $M\setminus A$ has compact closure. Then for any compact subset $K$ of $M\setminus A$,
\[
\limsup_{t\searrow 0} 4t \log p_t(x,A,y) \leq - d^2(x,A,y)
\]
uniformly for $x\in K$ and $y\in K$. Also, if $\KR$ is any compact subset of $M\times M$, then we have $4t\log p_t(x,y)\rightarrow -d^2(x,y)$ uniformly for $(x,y)\in \KR$.
\end{theorem}

Our main task now is to show that the $y$-derivatives of the heat kernel and its natural logarithm satisfy the analogous estimates under the same localization conditions. First, we recall a crucial result of L\'eandre \cite{leandremaj}. Let $\cZ_0, \cZ_1,\ldots, \cZ_k$ be smooth vector fields on $\bR^d$, such that $\cZ_1,\ldots, \cZ_k$ are bracket-generating. We also assume that these vector fields, and all of their derivatives (in standard Cartesian coordinates on $\bR^d$) are bounded, and that we have some smooth volume $\mu$.  In particular, we are in the situation described in \eqref{Eqn:Delta}, where the sub-Riemannian structure and hypo-elliptic operator $\Delta$ satisfy additional global conditions. In this situation, L\'eandre proved that, for any multi-index $\alpha$,
\begin{equation}\label{Eqn:LeandreOG}
\limsup_{t\searrow 0} 4t \log \lp \lab \partial_y^{\alpha}p_t(x,y)  \rab \rp \leq - d^2(x,y) \quad\text{uniformly on any compact subset of $\bR^d\times\bR^d$,}
\end{equation}
where $p_t$ and $d(\cdot,\cdot)$ are heat kernel and distance associated to the diffusion and induced sub-Riemannian structure on $\bR^d$, and where $\partial^\alpha_y$ denotes the $\alpha$ partial derivative (in standard Cartesian coordinates on $\bR^d$) acting on the $y$-variable. To localize this estimate, we first need a lemma showing that the process cannot move away from its starting point too quickly.

For any system of (smooth) coordinates $u_1 , \ldots , u_n$ on an open set $U \subset M$ with compact closure, we say that the coordinates are extendable if they can be extended to a neighborhood of the closure of $U$. While a general open, contractible set might not admit an extendable coordinate system on an incomplete manifold, it is clear from the ball-box theorem of sub-Riemannian geometry that any point of $M$ has a neighborhood that admits an extendable coordinate system.

A central feature of our approach is that we take precompact subsets of $M$ and include them in different ambient sub-Riemannian manifolds, for which better estimates are already known. In preparation for this, it is useful to observe how multiplying the vectors fields determining the sub-Riemannian structure on $M$ by a smooth function (eventually a bump function) affects the structure. Thus, for (smooth) vector fields $Z_1,\ldots, Z_k$ and a smooth function $\phi$, we observe that
\begin{equation}\label{Eqn:GluingCom}
\lb \phi Z_i, \phi Z_j\rb = \phi^2 \lb Z_i,  Z_j\rb +\phi \cdot\lp Z_i\phi\rp\cdot Z_j -\phi \cdot\lp Z_j \phi\rp\cdot Z_i .
\end{equation}
So the Lie bracket of $\phi Z_i$ and $\phi Z_j$ differs from (a multiple of) that of $Z_i$ and $Z_i$ only by a vector field in the span of $Z_i$ and $Z_j$ (assuming $\phi\neq 0$ at the point in question). It follows that, on the set where $\phi>c>0$, if the $Z_i$ are bracket-generating, so are the $\phi Z_i$. For completeness, we also see that
\begin{equation}\label{Eqn:GluingOp}
\lp \phi Z_i\rp^2 = \phi^2\cdot Z_i^2 + \phi\cdot \lp Z_i \phi\rp\cdot Z_i .
\end{equation}
Thus, applying this to our $\Delta$, on the set where $\phi>c>0$, the operators $\sum_{i=1}^k \lp \phi \cZ_i\rp^2 +\phi \cZ_0$ and $\sum_{i=1}^k \cZ^2_i + \cZ_0$ have principal symbols that differ only by scaling by $\phi^2$ and sub-symbols that differ by a horizontal vector field plus $(1-\phi)\cZ_0$.

Let $a>0$ be a positive constant. We let $\sigma_a$ be the first time the diffusion $X_t$, as in \eqref{Eqn:Xt}, moves a distance $a$ from its starting point, so that $\sigma_a$ is the first hitting time of the sub-Riemannian sphere of radius $a$ around $X_0$ (and which can be infinite on an incomplete $M$ if the process blows up before traveling distance $a$). Then we have the following result, showing that $X_t$ can't move too far from its starting point too quickly. Here and in what follows, we let
\[
B(x,r)=\lc z: d(x,z)<r\rc
\]
denote the open ball of radius $r$ centered at $x$.

\begin{lemma}\label{Lem:SigA}
Let $K\subset M$ be compact, and fix $a>0$. Then there exists $T>0$ such that
\[
\Prob\lp \sigma_a <T | X_0=x\rp <\frac{1}{2}
\]
for any $x\in K$.
\end{lemma}
\begin{proof}
First suppose that $K$ is contained in a single coordinate chart, so that $z_1,\ldots, z_d$ are coordinates on some neighborhood of $K$. By monotonicity, if the lemma holds for any $a$, it holds for any larger $a$, so we can assume that $a$ is small enough so that for any $x\in K$, $B(x,a)$ is contained in this coordinate patch. By smoothness and compactness, there exists some $c$ such that, for any $(z_1,\ldots,z_d)\in K$,
\[
\lb z_1-c,z_1+c\rb  \times \cdots\times \lb z_d-c,z_d+c\rb \subset B \lp z_1,\ldots,z_d; a\rp .
\]
Let $\tau_i$ be the event that the $z_1$ coordinate moves by $c$ from its starting value, and observe that if $\tau_i<T$ for all $i=1,\ldots, d$, then $\sigma_a<T$. Again by smoothness and compactness, if we write the SDE satisfied by $z_i$ under the diffusion as
\[
dz_i\lp X_t\rp = \alpha_i\lp X_t\rp \, dW_t + \beta_i\lp X_t\rp\, dt ,
\]
for some one-dimensional Brownian motion $W_t$, then there exists $\lambda>0$ such that $|\alpha_i|<\lambda$ and $|\beta_i|<\lambda$ everywhere on $\cup_{x\in K} B_a(x)$, for all $i=1,\ldots, d$. This uniform bound on the coefficients implies uniform bounds on how quickly the bounded variation part of $z_i\lp X_t\rp$ can grow in $t$ and also how quickly the quadratic variation of the martingale part can grow. Hence the probability of $\tau_i<T$ can be made as small as we want by taking $T$ small, uniformly over $X_0=x\in K$ and for all $i=1,\ldots, d$. In particular, by Boole's inequality, we can find $T$ such that $\sigma_a <T$ is less then $1/2$, uniformly for $x\in K$. Since $a$ was arbitrarily small, this proves the lemma when $K$ is contained in a single coordinate patch.

In general, every point of $K$ is contained in some coordinate patch, and by compactness, $K$ can be covered by finitely many coordinate patches for which the result holds, proving it in general.
\end{proof}

\subsection{Localization bounds}\label{Sect:Local2}

Next, we use the results of \cite{OurAIHP} to give the following localized version of \eqref{Eqn:LeandreOG}.

\begin{lemma}\label{Lem:Our7}
For a smooth sub-Riemannian structure $(M,\Delta,\mu)$, suppose that, for some $\eta>0$, $B$, $B'$ and $B''$ are concentric open balls of radii $\eta/2$, $(3/2)\eta$, and $(7/2)\eta$, respectively, and that $B''$ has compact closure. Suppose further that we have an extendable coordinate system on $B''$. Then, for any multi-index $\alpha$,
\begin{equation}\label{Eqn:DerBound}
\limsup_{t\searrow 0} 4t \log \lp \lab \partial_y^{\alpha} p^{B''}_t(x,y)  \rab \rp \leq - d^2(x,y)
\end{equation}
uniformly over $(x,y)\in \overline{B'}\times B$, where the partial derivatives are understood with respect to this extendable coordinate system.
\end{lemma}

\begin{proof}
Our general assumptions on $M$ plus Lemma \ref{Lem:SigA} means that we are in the situation of Section 3 of \cite{OurAIHP}. Then Lemma 7 of that paper says that the conclusion of the lemma holds if, for any multi-index $\alpha$, we have
\begin{equation}\label{Eqn:LeadndreOG2}
\limsup_{t\searrow 0} 4t \log \lp \lab \partial_y^{\alpha}p_t(x,y)  \rab \rp \leq - d^2(x,y) ,
\end{equation}
uniformly for $x$ and $y$ in $\overline{B''}$.

First, suppose that $M$ is diffeomorphic to $\bR^d$ and has a sub-Riemmanian structure of the kind under which L\'endre proved \eqref{Eqn:LeandreOG}. Then \eqref{Eqn:LeadndreOG2} immediately follows, proving the result.

In general, $M$ might not be of this form, but because $p^{B''}_t(x,y)$ depends only on the restriction of the structure to $B''$, we can get around this with a straightforward gluing argument. Let $U$ be a contractible open neighborhood of $\overline{B''}$ such that the extendable coordinate system, which we write as $(z_1,\ldots,z_d)$, extends to $U$. Then $U$ can be included in $\bR^d$ using the coordinates (indeed, by definition $U$ is diffeomorphic to some open subset of $\bR^d$ via the coordinate system). Now let $\phi$ be a smooth bump function, so that $0\leq \phi\leq 1$, $\phi\equiv 1$ on a neighborhood of $\overline{B''}$ and $\phi$ is smooth with support contained in $U$. Then we can use $\phi$ to extend the sub-Riemannian structure to all of $\bR^d$ by taking
\[
\hat{\Delta} = \sum_{i=1}^k \lp \phi \cZ_i\rp ^2 + \phi \cZ_0 + \sum_{i=1}^{d} \lp  (1-\phi) \partial_{z_i} \rp ^2
\quad\text{and}\quad \hat{\mu}= \phi \cdot \mu +(1-\phi)\cdot \mu_{\textrm{Euc}} ,
\]
where the $\cZ_i$ give the original sub-Riemannian structure (via $\Delta$) on $U$ and $\mu_{\textrm{Euc}}$ is the standard Euclidean volume on $\bR^d$. Then it is clear, by \eqref{Eqn:GluingCom} and \eqref{Eqn:GluingOp} (and the surrounding discussion), that $\hat{\Delta}$ and $\hat{\mu}$ determine a smooth sub-Riemannian structure on $\bR^d$, which we call $\hat{M}$, which agrees with that of $M$ on $B''$ and which agrees with the standard Euclidean structure (including the standard Euclidean volume) on $U^c$. By the smoothness of all the objects involved, we see that $\hat{M}$ satisfies the assumptions under which L\'eandre proved \eqref{Eqn:LeandreOG}. Thus, as before, if we let $p_t^{\hat{M}}$ denote the heat kernel on $\hat{M}$, we have \eqref{Eqn:LeadndreOG2} for $p_t^{\hat{M}}$, with $\overline{B''}$ understood as a subset of $\hat{M}$. We can then apply Lemma 7 of  \cite{OurAIHP} to see that, for any multi-index $\alpha$,
\[
\limsup_{t\searrow 0} 4t \log \lp \lab \partial_y^{\alpha} p^{B''}_t(x,y)  \rab \rp \leq - d_{\hat{M}}^2(x,y)
\]
uniformly over $(x,y)\in \overline{B'}\times B$. Here $p^{B''}_t$ is the heat kernel on $\hat{M}$ killed upon leaving $B''$, but the point is that that is the same as the heat kernel on $M$ killed upon leaving $B''$, since it only depends on the restriction of the sub-Riemannian structure to $B''$. Further, the radii of $B$, $B'$, and $B''$ are such that $d_{\hat{M}}(x,y)=d(x,y)$ for all $(x,y)\in \overline{B'}\times B$ (where $d(x,y)$ is understood with respect to $M$). To see this, note that, for $(x,y)\in \overline{B'}\times B$ and any $u$ and $v$ in $\partial B''$, we have $d(x,y)\leq 2\eta$ while $d(x,u)\geq 2\eta$ and $d(v,y)\geq 3\eta$, using the triangle inequality. It follows (on $M$) that all length-minimizing curves from $x$ to $y$ lie within $B''$ (and there is at least one), and any curve from $x$ to $y$ that leaves $B''$, no matter how the metric is extended beyond $B''$, cannot minimize the distance. This verifies the claim. In other words, for the conclusion of the lemma, it doesn't matter whether we consider $B''$ as included in $M$ or $\hat{M}$, and thus we have proven the lemma in general.
\end{proof}

We recall and slightly reformulate one more result from \cite{OurAIHP}, which will be a main tool in what follows.

\begin{lemma}\label{Lem:Our8}
For a sub-Riemannian manifold $M$, suppose that we have sets $K_0\subset U_0\subset K_1\subset U_1\subset M$ where $K_0$ and $K_1$ are compact and $U_0$ and $U_1$ are open with compact closure, and suppose that, for some $\eps>0$, the heat kernel on $M$ satisfies the estimate
\[
\limsup_{t\searrow 0} 4t \log p_t\lp x,U_1^c,y\rp \leq -\lp d(x,y)+\eps\rp ^2
\]
uniformly for $x$ and $y$ in $K_1$. Suppose further that for some $\eta>0$ and $y_0\in K_0$, the concentric balls $B, B', B''$ (which depend on $\eta$) centered at $y_0$ are as in Lemma \ref{Lem:Our7} (including the existence of an extendable coordinate system), and the closure of $B''$ is contained in $U_0$. Then for any multi-index $\alpha$,
\[
\limsup_{t\searrow 0}  4t \log\lp  \lab\partial_y^{\alpha} p_t(x,U_1^c,y)\rab \rp \leq -\lp d(x,y) +\eps-3\eta \rp^2
\]
uniformly over $x\in K_0$ and $y\in B$.
\end{lemma}

\begin{proof}
Our assumptions on $M$ and the results of Lemma \ref{Lem:Our7} mean that we can apply Lemma 8 of \cite{OurAIHP} to the situation described. But this exactly gives the conclusion of the lemma.
\end{proof}

We will need the following basic result about the limsup of the log of a linear combination.
\begin{lemma}\label{Lem:limsups}
Let $f_1(t),\ldots, f_n(t)$ be non-negative functions for $t\in (0,T]$, for some $T>0$, and $a_1,\ldots,a_n$ be positive constants. Then
\[
\limsup_{t\searrow 0} t\log\lp a_1f_1(t)+\cdots +a_nf_n(t)\rp = \max_{i=1,\ldots,n}\lc\limsup_{t\searrow 0}  t\log\lp f_i(t)\rp \rc 
\]
\end{lemma}
\begin{proof}
We sketch the proof. Note that, for any $t$,
\[
\log\lp a_1f_1(t)+\cdots +a_nf_n(t)\rp - \max_{i=1,\ldots,n}\lc  \log\lp a_i f_i(t)\rp \rc \leq \log n.
\]
This implicitly assumes that $a_1f_1(t)+\cdots +a_nf_n(t)$ is not equal to zero. But if it is, we have that $\log\lp a_1f_1(t)+\cdots +a_nf_n(t)\rp$ and  $\max_{i=1,\ldots,n}\lc  \log\lp a_i f_i(t)\rp \rc$ are both $-\infty$, and we understand their difference to be zero and the above to be true. Then multiplying this inequality through by $t$, it follows that
\begin{equation}\label{Eqn:LD-1}
\limsup_{t\searrow 0} t\log\lp a_1f_1(t)+\cdots +a_nf_n(t)\rp = \max_{i=1,\ldots,n}\lc\limsup_{t\searrow 0}  t\log\lp a_if_i(t)\rp \rc .
\end{equation}
In addition, for each $i$, we see
\begin{equation}\label{Eqn:LD-2}
 t\log\lp a_if_i(t) \rp- t\log\lp f_i(t)\rp = t\log a_i \rightarrow 0 \quad\text{as $t\searrow 0$,}
\end{equation}
so that $\limsup_{t\searrow 0} t\log\lp a_if_i(t) \rp=\limsup_{t\searrow 0}  t\log\lp f_i(t)\rp$, and the lemma follows.
\end{proof}
We will also use this in the case when the functions $f_i(t)$ also depend uniformly on some parameters. In particular, suppose that we have a set of smooth vector fields $Z^1,\ldots,Z^k$ and also a set of (smooth) coordinates on a neighborhood of the closure of some open set $U$, then we can write
\[
Z^1\cdots Z^k p_t(x,y) =  \sum_{\alpha:|\alpha|\leq k} c_{\alpha}(y) \partial^{\alpha} p_t(x,y)
\]
where the $c_{\alpha}(y)$ are smooth functions depending on the $Z^i$ and the choice of coordinates. It follows that
\[
4t \log\lp\lab Z^1\cdots Z^k p_t(x,y)\rab\rp \leq 4t \log \sum_{\alpha:|\alpha|\leq k} \lab c_{\alpha}(y)\rab \lab\partial^{\alpha} p_t(x,y)\rab.
\]
Then if we have that 
$\limsup_{t\searrow 0}  4t \log\lp  \lab\partial^{\alpha} p_t(x,y)\rab \rp \leq -d(x,y)^2$
uniformly (for $(x,y)$ in some subset of $M\times M$ such that the projection into the second component is contained in $U$) for all $\alpha$ with $|\alpha|\leq k$ (which will be the case below), we can conclude that
\[
\limsup_{t\searrow 0}  4t \log\lp  \lab Z^1\cdots Z^k p_t(x,y)\rab \rp \leq - d(x,y)^2
\]
uniformly as well. To see this, we note that the previous lemma, with $f_{\alpha}= \lab\partial^{\alpha} p_t(x,y)\rab$ for $|\alpha|\leq k$ and $a_{\alpha}=\lab c_{\alpha}(y)\rab$, applies pointwise in $(x,y)$. Then the uniformity of the limsup over $\alpha$ means that \eqref{Eqn:LD-1} holds uniformly. Moreover, the $c_{\alpha}(y)$ are smooth on a compact set containing $U$, and are therefore bounded on $U$. It follows that the convergence in \eqref{Eqn:LD-2} is also uniform, from which we see that the conclusion of the lemma holds uniformly, and the claim follows.

\subsection{Localized asymptotics}\label{Sect:Local3}

The basic logic of our approach is to establish localization estimates and then control the heat kernel on compacts by localizing estimates for compact manifolds. One minor difficulty is that the derivative estimates we wish to localize, namely \eqref{Eqn:LeandreOG}, were proven for certain structures on $\bR^d$, not for compact manifolds. (In particular, while a precompact subset of some sub-Riemannian manifold $M$ can be included in a compact manifold of the same dimension via a standard smooth doubling argument, as we will do below, in general there are topological obstructions to including it in the Euclidean space, viewed as a smooth manifold, of the same dimension.) Thus we take a slight detour to establish these estimates for compact manifolds.

\begin{theorem}\label{THM:LeandreCompact}
Let $M$ be a compact sub-Riemannian manifold. Then for any multi-index $\alpha$,
\[
 \limsup_{t\searrow 0} 4t \log \lp \lab  Z_y^{\alpha} p_t(x,y)  \rab \rp  \leq - d^2(x,y)
\]
uniformly for $(x,y)\in M\times M$.
\end{theorem}

\begin{proof}
By Whitney's embedding theorem, $M$ can be smoothly embedded in $\bR^{d+n}$ for some positive integer $n$. So identify $M$ as a compact $d$-dimensional submanifold of $\bR^{d+n}$, and let $\Sigma_{s}$ and $\Sigma_{2s}$ be tubular neighborhoods of $M$ of radii $s$ and $2s$ respectively, where here the radius is understood with respect to the Euclidean metric on $\bR^{d+n}$, and where $s$ is sufficiently small so that these tubular neighborhoods exist (in the sense of the tubular neighborhood theorem). It follows that $\Sigma_s$ can be realized as a fixed-radius subset of the normal bundle over $M$, that is, locally $\Sigma_s$ can be written as $V\times B(0,s)$ where $V$ is an open subset of $M$ and $B(0,s)$ is the open ball of Euclidean radius $s$ around the origin in $\bR^{n}$, and similarly for $\Sigma_{2s}$. We start by putting the corresponding product metric on $\Sigma_s$. More precisely, we give $M$ its sub-Riemannian structure determined by $\Delta_M$, we give $B(0,s)$ the standard Euclidean metric determined (in the formalism of this paper) by the usual Laplace operator $\Delta_{\bR^n}$, and we give $\Sigma_s$ the sub-Riemannian structure determined by the product operator $\Delta_M\times\Delta_{\bR^n}$. Next, we rescale the metric on the $B(0,s)$ factor by some positive constant $\lambda$ large enough so that, if we denote the rescaled ball by $B^{\lambda}(0,s)$, the radius of $B^{\lambda}(0,s)$ is greater than the diameter of $M$ (which is finite by compactness). Finally, let $\phi$ be a smooth bump function with $\phi\equiv 1$ on $\Sigma_s$, $0\leq \phi\leq 1$ on $\Sigma_{2s}\setminus\Sigma_s$, and $\phi\equiv 0$ on $\bR^{d+n}\setminus\Sigma_{2s}$. Then just as in the proof of Lemma \ref{Lem:Our7}, we can use $\phi$ to smoothly transition the product sub-Riemannian structure on $\Sigma_s$ to the usual Euclidean structure on $\bR^{d+n}\setminus\Sigma_{2s}$. Using the same $\phi$, we extend the product measure on $\Sigma_s$ (that is, the smooth measure given as a product of the measure on $M$ and the usual Euclidean volume measure on $\bR^n$) to a smooth measure on all of $\bR^{d+n}$.

We now observe several basic properties of the resulting sub-Riemannian structure on $\bR^{d+n}$. Since $M$ is compact and every point as a neighborhood where $\Delta_M$ can be written in the form \eqref{Eqn:Delta}, using a finite partition of unity, we can write $\Delta_M$ using finitely many globally defined vector fields. The operator on $B^{\lambda}(0,s)$ can be written using rescaled versions of the usual coordinate vector fields, which implies that the product operator $\Delta_M\times\Delta_{B^{\lambda}(0,s)}$ can be written using finitely many vector fields defined on $\Sigma_s$. The structure transitions to the usual Euclidean structure outside of $\Sigma_{2s}$, so that we have a sub-Riemannian structure on $\bR^{d+n}$ determined by an operator which can be written globally in the form \eqref{Eqn:Delta} using finitely many vector fields. Moreover, since all of these vector fields except for the standard coordinate vector fields on $\bR^{d+n}$ are compactly supported, they are bounded along with their derivatives of all orders. The strong H\"ormander condition follows immediately from the fact that it holds on $M$, on $B^{\lambda}(0,s)$, and on $\bR^{d+n}$, and thus we have a sub-Riemannian structure of the kind considered by L\'eandre, so that \eqref{Eqn:LeandreOG} holds.

If we write $d'$ for the resulting distance on $\bR^{d+n}$ (which is not the usual Euclidean distance), then we claim that, for any $x,y\in M$, we have
\begin{equation}\label{Eqn:Distances}
d'\lp (x,0),(y,0)\rp = d_M(x,y) ,
\end{equation}
where we identify $M$ with its image in $\bR^{d+n}$ using the natural product coordinates on $M\times B(0,s)=\Sigma_s\subset \bR^{d+n}$. Indeed, because of the product structure on $\Sigma_s$, any curve $\gamma$ in $M$ from $x$ to $y$ has the same length as its image under the inclusion, which we can write as $(\gamma,0)$ going from $(x,0)$ to $(y,0)$. Any curve from $(x,0)$ to $(y,0)$ in $\Sigma_s$ that isn't contained in the image of $M$ is strictly longer than its projection onto $M$; that is, the length of $(\gamma_1,\gamma_2)$ is strictly longer than that of $(\gamma_1,0)$ if $\gamma_2$ is not identically 0. Also, the rescaled metric on $B^{\lambda}(0,s)$ was chosen so that any path from $(x,0)$ to $(y,0)$ that leaves $\Sigma_s$ has length greater than twice the diameter of $M$. These facts establish the claim.

In reference to the notation of Lemma \ref{Lem:Our8}, let $K_0=M$, let $U_0$ be a small enough neighborhood of $M$ (in a sense to be indicated in a moment), let $K_1=\overline{U}_0$, and let $U_1=\Sigma$. Again using that the metric on $B^{\lambda}(0,s)$ was scaled so that $d'\lp(x,0),\Sigma_s^c \rp$ is more than the diameter of $M$, we see that for small enough $U_0$, we can find $\eps>0$ such that
\[
d'\lp (x,0), \Sigma_s^c\rp +d'\lp (y,0),\Sigma_s^c\rp > d'(x,y)+\eps 
\]
for all $x,y\in K_1$. Then by Theorem \ref{THM:StrongCond} with $A=\Sigma^c$ and $K=K_1$, we have that 
\[
\limsup_{t\searrow 0} 4t \log p_t\lp (x,0),\Sigma_s^c,(y,0)\rp \leq -\lp d' \lp (x,0),(y,0)\rp+\eps\rp ^2
\]
uniformly for $x$ and $y$ in $K_1$. Next, for any $y\in M$, we can find an $\eta>0$ and a ball $B$ around $y$ such that we can apply Lemma \ref{Lem:Our8} to see that, for any multi-index $\alpha$,
\[
\limsup_{t\searrow 0}  4t \log\lp  \lab\partial_y^{\alpha} p^{\bR^{d+n}}_t((x,0),\Sigma_s^c,(y,y'))\rab \rp \leq -\lp d' \lp (x,0),(y,y')\rp +\eps-3\eta \rp^2
\]
uniformly over $x\in M$ and $(y,y')\in B$,  where $p^{\bR^{d+n}}_t$ is the heat kernel for the sub-Riemannian structure we've put on $\bR^{d+n}$, and the partial derivatives $\partial_y^{\alpha}$ are understood with respect to the standard coordinate vector fields on $\bR^{d+n}$. Because $M$ is compact, we can find $\delta>0$ such that
\begin{equation}\label{Eqn:SigmaC}
\limsup_{t\searrow 0} 4t \log \lp  \lab\partial_y^{\alpha} p^{\bR^{d+n}}_t \lp(x,0), \Sigma_s^c, (y,0)\rp\rab\rp \leq - \lp d_M(x,y)+\delta \rp^2
\end{equation}
uniformly for $x,y\in M$, where we've freely used \eqref{Eqn:Distances}. 

Since $\Sigma_s$ has compact closure and $p_t^{\bR^{d+n}}$ satisfies  \eqref{Eqn:LeandreOG} uniformly on compacts, we have
\begin{equation}\label{Eqn:Rdn}
\limsup_{t\searrow 0} 4t \log \lp  \lab\partial_y^{\alpha} p_t^{\bR^{d+n}} \lp(x,0),(y,0)\rp\rab\rp \leq - d_M^2(x,y)
\end{equation}
uniformly for $x,y\in M$, again using \eqref{Eqn:Distances}. Applying Lemma \ref{Lem:limsups} to the decomposition
\[
p_t^{\Sigma_s}\lp (x,u),(y,v)\rp = p^{\bR^{d+n}}_t\lp(x,u),(y,v)\rp - p^{\bR^{d+n}}_t\lp (x,u), \Sigma_s^c, (y,v)\rp
\]
(valid for any $x,y\in M$ and $u,v\in B(0,s)$, and thus on all of $\Sigma_s$) and using \eqref{Eqn:SigmaC} and \eqref{Eqn:Rdn}, we conclude that
\begin{equation}\label{Eqn:CoordEstimate}
\limsup_{t\searrow 0} 4t \log \lp  \lab\partial_y^{\alpha} p_t^{\Sigma_s} \lp(x,0),(y,0)\rp\rab\rp \leq - \lp d_M(x,y)+\delta \rp^2
\end{equation}
uniformly for $x,y\in M$.

By the product structure on $\Sigma_s$, we have
\[
p_t^{\Sigma_s}\lp (x,u),(y,v)\rp = p^M_t (x,y) \cdot p^{B^{\lambda}(0,s)}_t(u,v) \\
\]
where $ p^{B^{\lambda}(0,s)}_t$ is the heat kernel on the Euclidean disk in $\bR^n$ of the correct radius with Dirichlet boundary conditions (and where we recall that our heat kernels are written with respect to the product measure on $\Sigma_s$). If $Z^1,\ldots, Z^k$ are smooth vector fields in some neighborhood of $y\in M$, then (choosing an arbitrary extension of these vector fields to a neighborhood of $(y,0)\in\Sigma_s$),
\begin{equation}\label{Eqn:SigmaS}
\log \lp  \lab Z^1\cdots Z^k p_t^{\Sigma_s} \lp(x,0),(y,0)\rp\rab\rp = \log \lp  \lab Z^1\cdots Z^k  p^M_t (x,y) \rab \rp + \log p^{B^{\lambda}(0,s)}_t(0,0)
\end{equation}
since $p^{B^{\lambda}(0,s)}_t(u,v)$ is constant on $M$ (because $u\equiv v\equiv 0$ on $M$) and the $Z^i$ are all tangent to $M$. Note also that the left-hand side is independent of the extension of the $Z^i$ to a neighborhood of $M$ and the first term on the right-hand side is understood in the intrinsic structure on $M$. Observe that
\[
\log p^{B^{\lambda}(0,s)}_t(0,0) \leq \log p^{\bR^n}_t(0,0) =\log \frac{1}{\lp 4\pi t\rp^{n/2}} =-\frac{n}{2}\log \lp 4\pi t\rp,
\]
so that $\limsup_{t\searrow 0} 4t \log  p^{B^{\lambda}(0,s)}_t(0,0) =0$. Combining this with \eqref{Eqn:SigmaS}, Lemma \ref{Lem:limsups} (and the remarks immediately following it), and \eqref{Eqn:CoordEstimate}, we conclude that
\[
\limsup_{t\searrow 0} 4t \log \lp  \lab Z^1\cdots Z^k p_t^M \lp x,y \rp\rab\rp \leq - \lp d_M(x,y)+\delta \rp^2 
\]
uniformly for $x,y\in M$. Since $\delta>0$ was arbitrary, this proves the desired result.
\end{proof}

%%%%%%%%%%%%%%%%%%%%%%%%%%%%%%%%%%%%%%%%%%%%%%%%

With these preliminaries taken care of, we can prove the main result of this section.
\begin{proof}[Proof of Theorem \ref{THM:TrueLocal}]
We give the proof in three steps. \\
\emph{Step 1}: We first establish the basic localization estimates when $\KR$ satisfies the strong localization condition, so assume this. For ease of exposition, we temporarily assume also that $M$ is incomplete, so that $d(x,\infty)$ is finite for all $x\in M$. Then $d(x,\infty)+d(y,\infty)-d(x,y)$ is continuous in $(x,y)\in M\times M$ and $\KR$ is compact, so we can find $\delta>0$ such that, for any $(x,y)\in \KR$,
\[
d(x,y)+4\delta < d(x,\infty)+d(y,\infty) .
\]
Let $\pi_1(\KR)$ and $\pi_2(\KR)$ be the projections onto the first and second components, respectively. Note that the triangle inequality implies, under this condition, that $d(x,\infty)$ and $d(x,\infty)$ themselves are each greater than $2\delta$, for any $x\in \pi_1(\KR)$ and $y\in \pi_2(\KR)$. Now we determine $U_1$ as follows
\[
U_1=\lp \bigcup_{x\in \pi_1(\KR)} B\lp x,d(x,\infty)-\frac{\delta}{2}\rp \rp \cup \lp \bigcup_{y\in \pi_2(\KR)} B\lp y,d(y,\infty)-\frac{\delta}{2}\rp \rp .
\]

We claim that $U_1$ is open with compact closure.  Indeed, the openness is clear, because it is written as a union of open sets. Next, for any $x_0\in \pi_1(\KR)$, we take $x\in B\lp x_0, \frac{\delta}{8}\rp$. Then the triangle inequality implies that
\[
B\lp x, d(x,\infty)-\frac{\delta}{2}\rp \subset  B\lp x_0, d(x_0,\infty)-\frac{\delta}{4}\rp ,
\]
where the ball on the right has compact closure. It follows that
\[
\bigcup_{x\in B\lp x_0,\frac{\delta}{8}\rp} B\lp x,d(x,\infty)-\frac{\delta}{2}\rp \subset B\lp x_0,d(x_0,\infty)-\frac{\delta}{4}\rp ,
\]
and thus the union on the left has compact closure (because its closure is contained in a compact). Since $\pi_1(\KR)$ is compact, it can be covered by finitely many open balls $B\lp x_i,\frac{\delta}{8}\rp$ for $i=1,\ldots, n$. Thus we have
\[
\bigcup_{x\in \pi_1(\KR)} B\lp x,d(x,\infty)-\frac{\delta}{2}\rp \subset 
\bigcup_{i= 1,\ldots,n} \bigcup_{x\in B\lp x_i,\frac{\delta}{8}\rp} B\lp x,d(x,\infty)-\frac{\delta}{2}\rp .
\]
Because the closure of a finite union is equal to the union of the closures, the set on the right has compact closure, and thus the set on the left has compact closure. An identical argument shows that
\[
\bigcup_{y\in \pi_2(\KR)} B\lp y, d(y,\infty)-\frac{\delta}{2}\rp
\]
has compact closure, and since the union of two sets with compact closure has compact closure, we have verified our claim that $U_1$ has compact closure.

Continuing, if we consider the open set $\mathcal{U}\subset M\times M$, given by
\[
\mathcal{U}= \lc (x,y): \text{there exists $(x_0,y_0)\in\KR$ such that } x \in B\lp x_0, \frac{\delta}{2}\rp \text{ and } y \in B\lp y_0,\frac{\delta}{2}\rp\rc ,
\]
we have, by repeated use of the triangle inequality,
\begin{equation}\label{Eqn:Tri}
d(x,y)+ 2\delta < d\lp x,U_1^c \rp+d\lp y,U_1^c\rp 
\end{equation}
for any $(x,y)\in\mathcal{U}$. At this point, we observe that we will let $U=U_1$ in the theorem. We have that $U$ is open with compact closure, and for any $(x,y)\in \KR$, both $x$ and $y$ are in $U$ by construction. Then $\KR$ can be taken to be the same in Theorem \ref{THM:StrongCond}, so \eqref{Eqn:1a} is immediate.

Further, in reference to the notation of Theorem \ref{THM:StrongCond}, we can let $A=U_1^c$ and, for any compact subset $D$ of $\mathcal{U}$ let $K=\pi_1(D)\cup \pi_2(D)$. (Also note that $K\subset U^1_c$, by the triangle inequality.) Then by Theorem \ref{THM:StrongCond} and \eqref{Eqn:Tri}, it follows that
\[
\limsup_{t\searrow 0} 4t \log p_t\lp x,U^c,y\rp \leq -\lp d\lp x,U^c \rp+d\lp y,U^c\rp\rp^2 \leq - \lp d(x,y)+2\delta \rp^2 
\]
uniformly for $(x,y)$ in $D$. Then  \eqref{Eqn:1b} follows by taking $D=\KR$.

Recall that the above assumes that $M$ is incomplete. In the case when $M$ is complete, we can find an open set $W$ with compact closure and a $\delta>0$ such that $\KR\subset W\times W$ and for any $(x,y)\in\KR$, 
\[
d(x,y)+4\delta < d\lp x,W^c\rp +d\lp y,W^c\rp .
\]
Then we can define $U_1$ as above, with $d(x,\infty)$ and $d(y,\infty)$ replaced by $d\lp x,W^c\rp$ and $d\lp y,W^c\rp$, and the rest of the preceding remains the same. (From one point of view, we consider $W$ as an incomplete sub-Riemannian manifold in its own right, for which the strong localization holds for $\KR$ as a subset of $W\times W$.)

It remains to consider \eqref{Eqn:2a} and \eqref{Eqn:2b}. We start with \eqref{Eqn:2b}.

Continuing from the above, choose any $\lp x_0,y_0\rp\in \KR$ and some $\delta' \leq \delta$ (which we will put an additional constraint on in a moment). Then we let
\[\begin{split}
K_0 &= \overline{B\lp x_0,\frac{\delta'}{4}\rp} \cup \overline{B\lp y_0,\frac{\delta'}{4}\rp}, \\
U_0 &=  B\lp x_0,\frac{\delta'}{3}\rp \cup B\lp y_0,\frac{\delta'}{3}\rp, \\
\text{and}\quad K_1&= \overline{U_0} .
\end{split}\]
Thus, we have $K_0\subset U_0\subset K_1\subset U_1\subset M$ where $K_0$ and $K_1$ are compact and $U_0$ and $U_1$ are open with compact closure. Further, we find that (for small enough $\delta'$)
\[
\limsup_{t\searrow 0} 4t p_t\lp x,U_1^c,y\rp \leq - \lp d(x,y)+2\delta \rp^2 
\]
uniformly for $x$ and $y$ in $K_1$. In particular, if $x\in \overline{B\lp x_0,\frac{\delta'}{3}\rp}$ and $y\in \overline{B\lp y_0, \frac{\delta'}{3}\rp}$ or vice versa (which are the interesting cases), then this follows by the choice of $\mathcal{U}$ and $\delta$ and the fact that $\delta'\leq \delta$. The other cases are when $x$ and $y$ are both in $\overline{B\lp x_0,\frac{\delta'}{3}\rp}$ or both in $\overline{B\lp y_0,\frac{\delta'}{3}\rp}$. However, by shrinking $\delta'$ if necessary, we can make $d(x,y)$ uniformly close to 0 on these balls, so that the estimate holds in these cases as well (the only need for these cases is to coordinate statements that are written for subsets of $M\times M$ and those written for subsets of $M$).

We have now verified the assumptions of Lemma \ref{Lem:Our8}, so that, for any sufficiently small $\eta>0$, we can find an open ball $B_{x_0}$ around $x_0$ and an open ball $B_{y_0}$ around $y_0$, such that there exists a coordinate system on the closure of $B_{y_0}$, and for any multi-index $\alpha$,
\[
\limsup_{t\searrow 0}  4t \log\lp  \lab\partial_y^{\alpha} p_t(x,U_1^c,y)\rab \rp \leq \lp d(x,y) +2\delta-3\eta \rp^2
\]
uniformly over $x\in \overline{B_{x_0}}$ (and thus for $x\in B_{x_0}$) and $y\in B_{y_0}$, where the partial derivatives act on $y$ and are taken with respect to the aforementioned system of coordinates. Because $\lp x_0,y_0\rp\in \KR$ was arbitrary and we have established this estimate on an open neighborhood of $\lp x_0,y_0\rp$, by compactness we can cover $\KR$ by finitely many such opens. We recall that $U_1$ was fixed above and doesn't depend on $\lp x_0,y_0\rp$ (unlike $K_0$, $U_0$ and $K_1$). By taking $\eta$ small enough so that $3\eta<\delta$, it follows that, for any multi-index $\alpha$,
\[
\limsup_{t\searrow 0}  4t \log\lp  \lab\partial_y^{\alpha} p_t(x,U_1^c,y)\rab \rp \leq \lp d(x,y) +\delta \rp^2
\]
uniformly for $(x,y)\in\KR$, where the partial derivatives are understood with respect to one of the finitely many coordinate patches which cover $\pi_2(\KR)$ in the obvious way. Since we can re-write the differential operator $Z_y^{\alpha}$ in terms of local coordinates on each of these coordinate patches, as described in Lemma \ref{Lem:limsups} and the comments that follow, \eqref{Eqn:2b} follows (again recalling that $U=U_1$).

 \emph{Step 2}: We continue with the situation, and notation, from Step 1, and note that, for the strong localization case, it remains only to establish \eqref{Eqn:2a}. We do this by including $U$ into a compact sub-Riemannian manifold and using \eqref{Eqn:2b} and Theorem \ref{THM:LeandreCompact}. (These gluing constructions are basic tools of smooth manifold geometry, and we refer to \cite{LeeBook} for the details.)

Let $V$, $V'$, $V''$, and $V'''$ be open neighborhoods of $\overline{U}$ in $M$ with compact closure, such that
\[
\overline{V}\subset V' \subset \overline{V'}\subset V''\subset \overline{V''}\subset V'''  .
\]
Then we can find a smooth bump function $\phi$ such that $0\leq \phi\leq 1$, $\phi\equiv 1$ on $\overline{V''}$, and the support of $\phi$ is contained in $V'''$. By Sard's theorem, there is some $a\in (1/4,3/4)$ such that $\phi^{-1}\lp [a,\infty)\rp$ is a smooth, compact  submanifold-with-boundary $S$ of $M$. Thus we can take the smooth double of S set to get a compact smooth manifold $\tilde{M}$, in which $V''$ is naturally included.

Note that $S$ and thus also $\tilde{M}$ need not be connected (indeed, the set $U=U_1$ from Step 1 need not be connected). However, because $S$ is a smooth, compact submanifold-with-boundary, it has only finitely many connected components, and thus the same is true of $\tilde{M}$. The argument that follows doesn't require that $\tilde{M}$ be connected. However, if one prefers, one can of course consider the connected component one at a time, and then use the fact that there are only finitely many to conclude that all uniform bounds on components can be chosen to hold uniformly over all components, and thus over all of $\tilde{M}$.

Continuing, we can again use another smooth bump function, supported in a neighborhood of $\overline{V'}\subset \tilde{M}$, to extend the sub-Riemannian structure from $V'$ to all of $\tilde{M}$. To give more detail, recall that the sub-Riemanniann structure on $V''$ (which is the restriction of that on $M$) is given by the vector fields $\cZ_i$ and the volume $\mu$. Then we can determine a (preliminary) sub-Riemannian structure on $\tilde{M}$ by vector fields $\tilde{\cZ}_i$ for $i=0,\ldots,\tilde{k}$ and smooth volume $\tilde{\mu}$ (indeed, we make no assumption about the rank being constant, we could make this a Riemannian structure).  Then, as in the proof of Lemma \ref{Lem:Our7}, we can let $\phi$ be a smooth bump function (with slight abuse of notation-- this $\phi$ is not the same as the previous $\phi$) with $0\leq \phi\leq 1$, $\phi\equiv 1$ on a neighborhood of $\overline{V'}$, and the support of $\phi$ contained in $V''$. Then, recalling \eqref{Eqn:GluingCom} and \eqref{Eqn:GluingOp}, we see that
\[
\hat{\Delta} = \sum_{i=1}^k \lp \phi \cZ_i\rp ^2 + \phi \cZ_0 + \sum_{i=1}^{\tilde{k}} \lp  (1-\phi) \tilde{\cZ}_i\rp ^2 + ( 1-\phi) \tilde{\cZ}_0
\quad\text{and}\quad \hat{\mu}= \phi \cdot \mu +(1-\phi)\cdot \tilde{\mu} ,
\]
gives a sub-Riemannian structure on $\tilde{M}$ which agrees with that of $M$ on $\overline{V'}$.

We need one further condition on the structure on $\tilde{M}$, namely that the distance between any two points in $U$ is the same for both the original $M$-distance and the $\tilde{M}$-distance (note that this is not automatic from the fact that the restriction of the sub-Riemannian structure to $U$ is the same, because the distance can, in principle, depend on the lengths of curves that exit $V'$). To ensure this, we can take another bump function $\psi$ satisfying $0\leq \psi\leq 1$, $\psi\equiv 0$ on $\overline{U}$, and $\phi\equiv 1$ on $V^c$. Then we further rescale $\phi \cZ_1,\ldots, \phi \cZ_k$ to be 
\[
\frac{\phi}{1+C\psi} \cZ_1,\ldots, \frac{\phi}{1+C\psi} \cZ_k
\]
 for some large enough $C>0$, which we now describe. Because $\tilde{M}$ is compact, we can make the generating vectors on $V'\setminus \overline{V}$ as short as we wish, uniformly, by making $C$ large. In particular, we can choose $C$ so that distance from $\overline{V}$ to $(V')^c$ is at least two times the $M$-distance between any two points of $U$, and we assume we have done so. It is clear that the $\frac{\phi}{1+C\psi} \cZ_i$ give a sub-Riemannian structure on $\tilde{M}$, and this is the structure we now take.
 
Consider any $(x,y)\in \KR$. Every curve contained in $U$ has the same length in either metric, so the $M$-distance minimizing curves (of which there is at least one, and all of which were contained in $U$) stay the same. Thus $d_{\tilde{M}}(x,y)\leq d_M(x,y)$. Now consider any admissible curve $\gamma$ from $x$ to $y$ contained in $V'$. Such a curve is admissible in either the original $M$-structure or in the $\tilde{M}$-structure just defined. By construction, the generating vectors for $\tilde{M}$, in $V'$, are not any longer than they were in $M$, and thus $\ell_{\tilde{M}}(\gamma)\geq \ell_M(\gamma)$, where $\ell_M$ and $\ell_{\tilde{M}}$ denote the length functionals on curves with respect to the two structures. Since $\ell_M(\gamma)\geq d_M(x,y)$, it follows that $\ell_{\tilde{M}}(\gamma)\geq d_M(x,y)$. On the other hand, suppose $\gamma$ is an admissible curve in $\tilde{M}$ from $x$ to $y$ that is not contained in $V'$. Then by our choice of $C$, $\ell_{\tilde{M}}(\gamma)> 2d_M(x,y)$. Since these two cases cover all curves from $x$ to $y$, we conclude that $d_{\tilde{M}}(x,y)\geq d_M(x,y)$.

We have established that
\begin{equation}\label{Eqn:dDef}
d_M(x,y) =d_{\tilde{M}}(x,y)
\end{equation}
for all $(x,y)\in \KR$. Another immediate consequence of the construction of the sub-Riemannian structure on $\tilde{M}$ is that the same argument as in the previous step can be applied to $U^c$ as a subset of $\tilde{M}$, so that
\[
\limsup_{t\searrow 0} 4t \log \lp \lab  Z_y^{\alpha}  p^{\tilde{M}}_t\lp x,U^c ,y\rp  \rab \rp  \leq -\lp d^2(x,y)+\delta\rp 
\]
uniformly for $(x,y)\in\mathcal{U}$, with the same $\delta$, where $d(x,y)$ can be thought of as either $d_M(x,y)$ or $d_{\tilde{M}}(x,y)$. This is because \eqref{Eqn:Tri} still holds, for either distance, because $d(x,U^c)$ depends only on the lengths of curves contained in $U$.

Recalling also that Theorem \ref{THM:LeandreCompact} applies to $\tilde{M}$, we now know that, for any multi-index $\alpha$,
\[\begin{split}
\limsup_{t\searrow 0} 4t \log \lp \lab  Z_y^{\alpha}  p^{\tilde{M}}_t\lp x ,y\rp  \rab \rp & \leq - d^2(x,y) \\
\limsup_{t\searrow 0} 4t \log \lp \lab  Z_y^{\alpha}  p^{M}_t\lp x,U^c ,y\rp  \rab \rp & \leq -\lp d^2(x,y)+\delta\rp \\ 
\limsup_{t\searrow 0} 4t \log \lp \lab  Z_y^{\alpha}  p^{\tilde{M}}_t\lp x,U^c ,y\rp  \rab \rp & \leq -\lp d^2(x,y)+\delta\rp 
\end{split}\]
uniformly for $(x,y)\in \KR$. Here we use the fact that $\KR$ and $U$ can be viewed as subsets of either $M$ or $\tilde{M}$, and that, by \eqref{Eqn:dDef}, $d(x,y)$ is unambiguous, being understood as either $d_M$ or $d_{\tilde{M}}$. From the decompositions
\[
p^{\tilde{M}}_t(x,y) = p^U_t(x,y)+ p^{\tilde{M}}_t\lp x,U^c ,y\rp\quad\text{and}\quad 
p^{M}_t(x,y) = p^U_t(x,y)+ p^{M}_t\lp x,U^c ,y\rp ,
\]
we see that
\[
p^{M}_t(x,y) = p^{\tilde{M}}_t\lp x ,y\rp - p^{\tilde{M}}_t\lp x,U^c ,y\rp + p^{M}_t\lp x,U^c ,y\rp .
\]
Then we can write
\[\begin{split}
\limsup_{t\searrow 0} 4t \log \lp \lab  Z_y^{\alpha}p^{M}_t(x,y)\rab\rp &\leq 
\limsup_{t\searrow 0} 4t \log \lp \lab  Z_y^{\alpha} p^{\tilde{M}}_t\lp x ,y\rp \rab\rp \\
 + &\limsup_{t\searrow 0} 4t \log \lp \lab  Z_y^{\alpha} p^{\tilde{M}}_t\lp x,U^c ,y\rp\rab\rp 
+ \limsup_{t\searrow 0} 4t \log \lp \lab  Z_y^{\alpha} p^{M}_t\lp x,U^c ,y\rp\rab\rp ,
\end{split}\]
and applying Lemma \ref{Lem:limsups} (and the comments following it about uniformity) along with the above estimates for the quantities on the rigth-hand side, we establish \eqref{Eqn:2a}.

This completes the proof under the strong localization condition.

\emph{Step 3}: We move on to the case when $\KR$ satisfies the weak localization condition. The approach is the same as above, except that we need a different choice of $U_1$.

We can take $\delta>0$ such that, for all $(x,y)\in \KR$, the set $\{z:d(x,z)+d(z,y) < d(x,y)+5 \delta\}$ has compact closure. Then we now take
\[
U_1=\bigcup_{(x,y)\in \KR} \lc z:d(x,z)+d(z,y) < d(x,y)+3\delta \rc .
\]
We claim that, as before, $U_1$ is open with compact closure. Again, openness is immediate. Next, for any $\lp x_0,y_0\rp \in \KR$, if we take $(x,y)\in B\lp x_0,\frac{\delta}{4}\rp\times B\lp y_0,\frac{\delta}{4}\rp$, then 
\[
\{ z: d(x,z)+d(z,y) < d(x,y)+4 \delta\} \subset \{z:d(\lp x_0,z\rp +d\lp z,y_0\rp < d\lp x_0,y_0\rp+5 \delta\}
\]
by the triangle inequality, and note that the set on the right has compact closure, by assumption. It follows that
\[
S\lp x_0,y_0\rp = \bigcup_{(x,y)\in B\lp x_0,\frac{\delta}{4}\rp\times B\lp y_0,\frac{\delta}{4}\rp} \lc z:d(x,z)+d(z,y) < d(x,y)+4\delta \rc
\]
is open, and has compact closure because it is contained in a set with compact closure. Because $\KR$ is compact, we can find a finite number of points $\lp x_1,y_1\rp,\ldots,\lp x_n,y_n\rp$ such that the sets $B\lp x_n,\frac{\delta}{4}\rp\times B\lp y_n, \frac{\delta}{4}\rp$ cover $\KR$, and thus $U_1\subset \cup_{i=1}^n S\lp x_i,y_i\rp$. Moreover, since the closure of a finite union is equal to the union of the closures, we have
\[
\overline{U_1}\subset \cup_{i=1}^n \overline{S\lp x_i,y_i\rp} .
\]
Finally, the union on the right-hand side is compact, because it is a finite union of compacts, and thus the closure of $U_1$ is compact, as claimed.

From here, we again define the the open set $\mathcal{U}\subset M\times M$ by
\[
\mathcal{U}= \lc (x,y): \text{there exists $(x_0,y_0)\in\KR$ such that } x \in B\lp x_0,\frac{\delta}{2}\rp \text{ and } y \in B\lp y_0, \frac{\delta}{2}\rp\rc ,
\]
Then the important point is that, by the triangle inequality,
\[
d(x,y)+ 2\delta < d\lp x,U_1^c,y \rp
\]
for any $(x,y)\in\mathcal{U}$, which is the analogue of \eqref{Eqn:Tri} under weak localization. From here, we again take $U=U_1$, and \eqref{Eqn:1a} and  \eqref{Eqn:1b} follow just as before, except that Theorem \ref{THM:WeakCond} should be used in place of Theorem \ref{THM:StrongCond}.

Continuing, \eqref{Eqn:2b} can then be proved just as in Step 1, with the same choices of $K_0$, $U_0$ and $K_1$ and the same use of Lemma \ref{Lem:Our8}.

Finally, to establish \eqref{Eqn:2a}, we include $\overline{U}$ in a compact $\tilde{M}$, exactly as in Step 2. We again have \eqref{Eqn:dDef}. Moreover, we see that
\[
d_{M} \lp x,U^c_1,y\rp  \leq d_{\tilde{M}} \lp x,U^c_1,y\rp 
\]
for any $x,y\in U$, because no curves that come close to realizing the infimum that defines $d_{\tilde{M}} \lp x,U^c_1,y\rp$
can leave $V'$ (because of how we rescaled the metric using $\psi$), and all curves contained in $V'$ are at least as long under the $\tilde{M}$-structure as under the $M$-structure.

The only remaining possible issue in applying the same reasoning as in Step 2 is that in order to conclude that
\begin{equation}\label{Eqn:Needed}
\limsup_{t\searrow 0} 4t \log \lp \lab  Z_y^{\alpha}  p^{\tilde{M}}_t\lp x,U^c ,y\rp  \rab \rp  \leq -\lp d^2(x,y)+\delta\rp ,
\end{equation}
we must know that $\KR$ satisfies the weak localization condition as a subset of $\tilde{M}$ (equipped with $\hat{\Delta}$ and $\hat{\mu}$, of course). More precisely, while we have already discussed the distance to $U^c$, the sector condition is a global assumption. However, the sector condition is easy to arrange. Because the original $\Delta$ on $M$ satisfies the sector condition, we know that $\cZ_0$ is in the span of the $\cZ_i$. Moreover, we're free to assume that $\tilde{\cZ}_0$ lies in the span of the $\tilde{\cZ}_i$, or even to take $\tilde{\cZ}_0\equiv 0$, since having a valid sub-Riemannian structure on $\tilde{M}$ depends only on $\tilde{\cZ}_1,\ldots,\tilde{\cZ}_{\tilde{i}}$, so assume that we do so. In reference to \eqref{Eqn:DivForm}, by the smoothness of all objects involved, we see that for any smooth $f$,
\[
\hat{\Delta} f - \mathrm{div}_{\hat{\mu}}\lp \tilde{\nabla} f\rp = \hat{\cZ}_0(f)
\]
for some smooth $\hat{\cZ}_0$ that can be written as a linear combination of 
\[
\frac{\phi}{1+C\psi} \cZ_1,\ldots,\frac{\phi}{1+C\psi} \cZ_k, (1-\phi)\tilde{\cZ}_1,\ldots, (1-\phi)\tilde{\cZ}_{\tilde{k}}
\]
with smooth coefficients. Then by smoothness and compactness of $\tilde{M}$, the ``sector condition'' \eqref{Eqn:SectorCon} is satisfied (in particular, the coefficient functions in writing $\hat{\cZ}_0$ as a linear combination of the above generating vector fields  are bounded). Thus we can apply Theorem \ref{THM:WeakCond} to $U^c$ as a subset of $\tilde{M}$ in order to get \eqref{Eqn:Needed}.

Having arranged for \eqref{Eqn:Needed} to hold, the rest of the argument is identical to that of Step 2. This completes the proof.
\end{proof}

\begin{remark}
As noted, part of the logic behind our proof of Theorem \ref{THM:TrueLocal} was to localize the heat kernel asymptotics to certain compact sets, via \eqref{Eqn:1b} and \eqref{Eqn:2b}. This goes somewhat beyond the claim in the theorem, which asserts only that there is some open $A$ with compact closure for which the results of the theorem hold. (Indeed, for a compact manifold that's trivial, but in Step 2 of the proof, we needed $A$ to be a particular set $U^c$, not the entire compact manifold.)

Motivated in part by this, one could ask if particular sets $A$ can be given. Fortunately in the course of the proof, we determined such sets. If $\KR$ satisfies the strong localization condition, we can find an open set $W$ with compact closure and a $\delta>0$ such that $\KR\subset W\times W$ and for any $(x,y)\in\KR$, 
\[
d(x,y)+4\delta < d\lp x,W^c\rp +d\lp y,W^c\rp .
\]
(and any such open set works) and then determine $A$ by
\[
A^c=\lp \bigcup_{x\in \pi_1(\KR)} B\lp x,d(x,W^c)-\frac{\delta}{2}\rp \rp \cup \lp \bigcup_{y\in \pi_2(\KR)} B\lp y,d(y,W^c)-\frac{\delta}{2}\rp \rp .
\]
On the other hand, if $\KR$ satisfies the weak localization condition, we can find a $\delta>0$ such that we can determine $A$ by
\[
A^c= \bigcup_{(x,y)\in \KR} \lc z:d(x,z)+d(z,y) < d(x,y)+3\delta \rc .
\]
\end{remark}

At this point, we explain how time derivatives can be incorporated into the bounds of \eqref{Eqn:2a} and \eqref{Eqn:2b} in the case when $\Delta$ is a symmetric operator. If so, the heat kernel is also symmetric, in the sense that $p_t(x,y)=p_t(y,x)$ for any $x,y \in M$. Then we can use the symmetry to move the spatial derivatives in the heat equation onto the $y$-variable, so that
\[
\partial_t p_t(x,y)= \Delta_y p_t(x,y) = \sum_{i=1}^k \cZ_{i,y}^2 p_t(x,y) +\cZ_{0,y} p_t(x,y) .
\]
It follows that, for any positive integer $l$, $\partial^l_t p_t(x,y)$ can be written as a finite linear combination of terms of the form $Z_y^{\alpha_j} p_t(x,y)$, where $|\alpha_j |\leq 2l$ for each $j$ (and the $Z^i$ happen to be drawn from the $\cZ_i$). Then in light of Lemma \ref{Lem:limsups}, under the assumptions of Theorem \ref{THM:TrueLocal} plus the additional assumption that $\Delta$ is symmetric (in which case one can always consider the weak localization condition) \eqref{Eqn:2a} and \eqref{Eqn:2b} can be improved to
\[\begin{split}
\limsup_{t\searrow 0} 4t \log \lp \lab  \partial^l_t Z_y^{\alpha}  p_t(x,y)  \rab \rp  &\leq - d^2(x,y)  \\
\text{and}\quad
\limsup_{t\searrow 0} 4t \log \lp \lab  \partial^l_t Z_y^{\alpha}  p_t\lp x,U^c ,y\rp  \rab \rp  &\leq -\lp d^2(x,y)+\delta\rp 
\end{split}\]
for any non-negative integer $l$. We will use this as necessary in what follows.

Recall
$$
\Gamma_\varepsilon(x,y)=
\left\{
z\in M
\mid  
d(x,z)\leq \frac{d(x,y)+\varepsilon}{2}
\text{ and }
d(y,z)\leq \frac{d(x,y)+\varepsilon}{2}
\right\}.
$$
Now that  Léandre asymptotics are proved, we have enough to show that the heat responsible for $p_t(x,y)$ is located near the midpoint set $\Gamma$ at $t/2$.

\begin{corollary}\label{C:MolchanovV1}
Let $\K$ be a localizable compact subset of $M^2$. 
Let $l$ be any non-negative integer in the symmetric case and 0 otherwise, and $\alpha$ any multi-index.
For any $\varepsilon>0$ small enough, we have uniformly on $\R^+\times \K$, for all $(t,x,y)\in\R^+\times \K$
$$
\partial_t^l
Z_y^{\alpha}
p_t(x,y)
=
\int_{\Gamma_\varepsilon} 
	p_{t/2}(x,z)
	\partial_t^l
Z_y^{\alpha}p_{t/2}(z,y)
\diff\mu( z)
+
O\left(\e^{-\frac{d(x,y)^2+{\varepsilon^2/2}}{4t}}\right).
$$
{(Here and in the rest of the paper, $\partial_t p_{t/2}$ should be understood as $(\partial_\tau p_{\tau})_{|\tau=t/2}$.)}
\end{corollary}

\begin{proof}

Let $\varepsilon>0$ and let $K\subset M$ be the closure of the set of points $z$ for which there exists $(x,y)\in M$ such that either $(x,z)\in \K$, $(z,y)\in \K$ or $z\in \Gamma_\varepsilon(x,y)$ for $(x,y)\in \K$. $K$ is naturally bounded and  compact. Since $\K$ is localizable, so is $K^2$, for $\varepsilon$ small enough. Then we can apply Theorem~\ref{THM:TrueLocal} over $K^2$, with $U$ and $\delta>0$ defined as there.

Let $p^U_t$ be the heat kernel on $U$ with Dirichlet boundary conditions. By Theorem~\ref{THM:TrueLocal}, for any $x,y\in K^2$, and any multi-index $\alpha$,  non-negative integer $l$, $\partial_t^{l}Z_y^{\alpha} p_t(x,y)=\partial_t^{l}Z_y^{\alpha} p_t^U(x,y)+\partial_t^{l}Z_y^{\alpha} p_t(x,U^c,y)$. We get
\begin{equation}\label{E:difference_dirichlet}
\lab \partial_t^{l}Z_y^{\alpha} p_t(x,y) - \partial_t^{l}Z_y^{\alpha} p^{U}_t(x,y)  \rab \leq C  \exp\lp  -\frac{d(x,y)^2+\delta}{4t} \rp.
\end{equation}
(Here $C$ is uniform over pairs in $K^2$.)

Using the fact that $\partial_t p_t(x,y)=\left.\partial_\tau\right|_{\tau=0} p_{t+\tau}(x,y)$ and dividing $t+\tau$ as $t/2+(t/2+\tau)$:
$$
\partial_t^{l}Z_y^{\alpha} p^{U}_t(x,y)
=
\left.\partial_\tau^{l}\right|_{\tau=0}Z_y^{\alpha}\int_{U}\lp p^{U}_{t/2}(x,z)p^{U}_{t/2+\tau}(z,y)\rp \diff \mu(z)
=
\int_{U}
p^{U}_{t/2}(x,z) \partial_t^{l} Z_y^{\alpha}  p^{U}_{t/2}(z,y)  
\diff \mu(z).
$$
We divide the domain of this last integral in order to estimate each part:
\begin{equation}\label{E:integral_division}
\begin{aligned}
\int_{U}
p^{U}_{t/2}(x,z) \partial_t^{l} Z_y^{\alpha}  p^{U}_{t/2}(z,y)  
  \diff \mu(z)=&
\int_{\Gamma_\varepsilon} 
p^{U}_{t/2}(x,z) \partial_t^{l} Z_y^{\alpha}  p^{U}_{t/2}(z,y)  
\diff \mu(z)
\\&+
\int_{U\setminus \Gamma_\varepsilon} 
p^{U}_{t/2}(x,z) \partial_t^{l} Z_y^{\alpha}  p^{U}_{t/2}(z,y)  
\diff \mu(z)
\end{aligned}
\end{equation}
Theorem~\ref{THM:TrueLocal} implies by uniformity over the spatial domain that for any $\eta>0$ to be fixed later, there exists $C>0$ such that for all $(x,y)\in \K$, $z\in K$,
$$
\left| p^{U}_{t/2}(x,z)\right|\leq C \exp\lp-\frac{d(x,z)^2{-\eta}}{2t}\rp
$$
and
$$
\left|\partial_t^{l} Z_y^{\alpha}  p^{U}_{t/2}(z,y)\right|\leq C \exp\lp-\frac{d(z,y)^2{-\eta}}{2t}\rp.
$$
Hence
$$
\left|p^{U}_{t/2}(x,z) \partial_t^{l} Z_y^{\alpha}  p^{U}_{t/2}(z,y)  \right|\leq C^2 \exp\lp-\frac{d(x,z)^2+d(z,y)^2{-2\eta}}{2t}\rp.
$$

Let $\varepsilon\geq 0$. If $d(x,z)\geq\frac{d(x,y)+\varepsilon}{2}$ then by triangular inequality, $d(z,y)\geq \frac{d(x,y)-\varepsilon}{2}$, so that
\begin{equation}\label{E:triangular_inequality}
\frac{1}{2}\lp d(x,z)^2+d(z,y)^2\rp
\geq 
\frac{1}{2}
\lp 
	\lp \frac{d(x,y)+\varepsilon}{2}\rp^2
	+
	\lp\frac{d(x,y)-\varepsilon}{2}\rp^2 
\rp
=\frac{d(x,y)^2+\varepsilon^2}{4}.
\end{equation}
Hence, for the integration over $U\setminus \Gamma_\varepsilon$ in \eqref{E:integral_division} we get the bound
$$
\int_{U\setminus \Gamma_\varepsilon}  p^{U}_{t/2}(x,z)\partial_t^{l}Z_y^{\alpha}p^{U}_{t/2}(z,y)  \diff \mu(z)
\leq 
\mu(U)C^2 \exp\lp-\frac{d(x,y)^2+\varepsilon^2{-4\eta}}{4t}\rp.
$$

The integration over $\Gamma_\varepsilon$ in \eqref{E:integral_division} should instead be compared with the same integral for the true kernel $p_t$.
In order to compare the integral of $p^{U}_{t/2}(x,z) \partial_t^{l} Z_y^{\alpha}  p^{U}_{t/2}(z,y)  $ and $p_{t/2}(x,z) \partial_t^{l} Z_y^{\alpha}  p_{t/2}(z,y) $, we use \eqref{E:difference_dirichlet}.
We are considering  pairs of points $(x,y)\in \K$, so that pairs $(x,z)$ and $(z,y)$, with $z\in \Gamma_\varepsilon(x,y)$, all belong to $K^2$. Then
we have both
$$
\left|(p_{t/2}(x,z) -p^{U}_{t/2}(x,z))\partial_t^{l} Z_y^{\alpha}  p_{t/2}(z,y) \right| \leq  C\exp\left(-\frac{d(x,z)^2+\delta{-\eta}}{2t}\right)\exp\left(-\frac{d(z,y)^2{-\eta}}{2t}\right)
$$
and 
$$
\left|
p^{U}_{t/2}(x,z) (\partial_t^{l} Z_y^{\alpha}  p^{U}_{t/2}(z,y) -\partial_t^{l} Z_y^{\alpha}  p_{t/2}(z,y)  )\right|\leq
C\exp\left(-\frac{d(x,z)^2{-\eta}}{2t}\right)\exp\left(-\frac{d(z,y)^2+\delta{-\eta}}{2t}\right)
$$
Taking \eqref{E:triangular_inequality} with $\varepsilon=0$ yields
$$
\left|p_{t/2}(x,z) \partial_t^{l} Z_y^{\alpha}  p_{t/2}(z,y)  
-
p^{U}_{t/2}(x,z) \partial_t^{l} Z_y^{\alpha}  p^{U}_{t/2}(z,y) \right|
\leq C^2\exp\lp-\frac{d(x,y)^2+4\delta{-8\eta}}{4t} \rp
$$
Then  there exists $C>0$ such that for all $(x,y)\in \K$,
\[\begin{split}
\int_{\Gamma_\varepsilon}
\lp
p_{t/2}(x,z) \partial_t^{l} Z_y^{\alpha}  p_{t/2}(z,y)  
-
p^{U}_{t/2}(x,z) \partial_t^{l} Z_y^{\alpha}  p^{U}_{t/2}(z,y)  
\rp
\diff \mu(z)
& \leq \\
C^2\mu(U) &
\exp\lp-\frac{d(x,y)^2+4\delta{-8\eta}}{4t} \rp
\end{split}\]
(since $\mu(\Gamma_\varepsilon)\leq \mu(U)$).

In conclusion, for all $(x,y)\in \K$,
\begin{multline*}
\left| 
\partial_t^{l}Z_y^{\alpha} p_t(x,y) 
-
\int_{\Gamma_\varepsilon}
p_{t/2}(x,z)\partial_t^{l}Z_y^{\alpha}p_{t/2}(z,y) \diff \mu(z)\right|
\leq \\
\begin{aligned}
&
 \lab \partial_t^{l}Z_y^{\alpha} p_t(x,y) - \partial_t^{l}Z_y^{\alpha} p^{U}_t(x,y)  \rab
\\
&+\int_{U\setminus\Gamma_\varepsilon}   p^{U}_{t/2}(x,z)\partial_t^{l}Z_y^{\alpha}p^{U}_{t/2}(z,y) \diff \mu(z)
\\
&
+
\int_{\Gamma_\varepsilon} \lp p^{U}_{t/2}(x,z)\partial_t^{l}Z_y^{\alpha}p^{U}_{t/2}(z,y)-p_{t/2}(x,z)\partial_t^{l}Z_y^{\alpha}p_{t/2}(z,y)\rp \diff \mu(z),
\end{aligned}
\end{multline*}
so that 
\begin{multline*}
\left| 
\partial_t^{l}Z_y^{\alpha} p_t(x,y) 
-
\int_{\Gamma_\varepsilon}
p_{t/2}(x,z)\partial_t^{l}Z_y^{\alpha}p_{t/2}(z,y)  \diff \mu(z)\right|
\leq
\\ C\e^{-\frac{d(x,y)^2+4\delta-8\eta}{4t}}
+C'\e^{-\frac{d(x,y)^2+\varepsilon^2-4\eta}{4t}}
+C''\e^{-\frac{d(x,y)^2+4\delta-8\eta}{4t}}
\leq C'''\e^{-\frac{d(x,y)^2+\varepsilon^2-4\eta}{4t}},
\end{multline*}
for $\varepsilon,\eta$ small enough.
Taking $\eta=\varepsilon^2/8$ proves the corollary.
\end{proof}
    
\section{Ben Arous expansion theorem}

{

In \cite{BenArous}, Ben Arous gives a full asymptotic expansion of the heat kernel in small time, for pairs of points away from the diagonal, the cut locus, or joined by abnormal minimizers.
In the rest of the paper, for any $x\in M$, $\Cut(x)$ denotes the cut locus of $x$ in $M$. Furthermore,  we recall the critical set $\C\subset M\times M$ from Definition \ref{Def:SN-CritSet}: the set of pairs of points $(x,y)$ such that either $y\in \Cut(x)$,  $x\in \Cut(y)$, $x=y$ in the non-Riemannian case (that is, $\C$ contains the diagonal in the properly sub-Riemannian case), or such that a length-minimizing curve from $x$ to $y$ is not  strongly normal.
In \cite{BenArous}, Ben Arous definition of cut locus includes points connected by an abnormal geodesic, which is not the convention we follow, hence the introduction of $\C$.
We can describe Ben Arous results with the following definition, which will supply convenient terminology for this section and allow us to treat the symmetric and general cases in parallel.

\begin{definition}\label{D:Ben Arous}
We say that the Ben Arous expansion holds uniformly on the compact subset $\K\subset M^2\setminus \C$ if 
for $l$ any non-negative integer in the symmetric case and  $0$ otherwise, and any multi-index $\alpha$, we have the following.

There exists an open neighborhood $\O$ of $\K$ in $M^2\setminus\C$,
there exist sequences of smooth functions $c_k:\O\to \R$, $k\in \N$, $r_{k}: (0,\infty)\times \O\to \R$, such that for all $N\in \N$, for all $(x,y)\in \O$, for all $t$ small enough
$$
\partial_t^l
Z_y^{\alpha}
p_t(x,y)
=
t^{-(|\alpha|+2 l+n/2)} \e^{-\frac{d(x,y)^2}{4t}}
\left(
\sum_{k=0}^{N}c_k(x,y)t^k
+
t^{N+1}r_{N+1}(t,x,y)
\right),
$$
and, for $l'$ any non-negative integer in the symmetric case and 0 otherwise, and any multi-index $\alpha'$, there exists $t_0>0$ such that
$$
\sup_{0<t<t_0}
\sup_{(x,y)\in \K}
\left|
	\partial_t^{l'}		Z_y^{\alpha'}r_{N+1}(t,x,y)
\right|<\infty.
$$

Additionally, if $\alpha=
0$, then $c_0(x,y)>0$ on $\O$.
\end{definition}

In particular, Theorems 3.1-3.4 in \cite{BenArous} imply that the heat kernel satisfies Definition~\ref{D:Ben Arous} with $M=\R^d$ and $\K$  any compact set in $\R^d\times \R^d\setminus \C$.
It has been widely accepted that Ben Arous expansions should hold uniformly on compact sets where no abnormal minimizers exist between two distinct points.  Thanks to localization, we are able to prove this fact using Molchanov's method, and we then apply this result to give uniform universal bounds of the heat kernel on compact sets without abnormal minimizers.

Techniques in this section rely heavily on properties of midpoint sets, the set of points equidistant from two points. Hence we preface our work on Ben Arous expansions with some preliminary remarks and definitions that will appear in proofs throughout the section.

}
\subsection{From localization to compactness near geodesics}
Recall that for any two points $x,y\in M$, we denote by $\Gamma(x,y)$ the midpoint set of $(x,y)$, that is the set of points $z$ that lay at the midpoint of length-minimizing curves between $x$ and $y$:
$$
\Gamma(x,y)=
\left\{
z\in M\mid  d(x,z)=d(z,y)=\frac{d(x,y)}{2}
\right\}
$$
and  for any $\varepsilon\geq 0$, we set
$$
\Gamma_\varepsilon(x,y)=
\left\{
z\in M
\mid  
d(x,z)\leq \frac{d(x,y)+\varepsilon}{2}
\text{ and }
d(y,z)\leq \frac{d(x,y)+\varepsilon}{2}
\right\}.
$$
For any $\eta\geq 0$, we denote by $\D(\eta)$ the subset of $M^2$ 
$$
\D(\eta)=\left\{ (x,y)\in M^2\mid d(x,y)\leq \eta \right\}.
$$
In particular, $\D(0)=\D\subset M\times M$ denotes the diagonal of $M\times M$.

Finally, regarding the localization condition, if $\K$ is a localizable compact subset $M\times M$, it implies in both strong and weak cases that for some $\varepsilon_0>0$, the set  
$$
\Lambda=\left\{z\mid d(x,z)+d(z,y)\leq d(x,y)+\varepsilon_0\right\}
$$
is compact for any $(x,y)\in \K$.

First, let us prove the following lemma.

\begin{lemma}\label{L:GeodesicsExist}
Let $\K$ be a compact subset of $M\times M$. If $\K$ is localizable then for any $(x,y)\in \K$, there exists a length-minimizing curve joining $x$ to $y$.
\end{lemma}

\begin{proof}
Let us consider a sequence of constant speed admissible curves $\gamma_n:[0,1]\to M$ such that $\gamma_n(0)=x$, $\gamma_n(1)=y$ and $\ell(\gamma_n)\to d(x,y)$, where 
$\ell(\gamma)$ denotes the length of an admissible curve $\gamma:[0,1]\to M$.
Recall $\Lambda=\{z\mid d(x,z)+d(z,y)\leq d(x,y)+\varepsilon_0\}$, for some $\varepsilon_0>0$. If for any given $n$, there exists $t_n\in [0,1]$ such that $\gamma_n(t_n)\in \partial \Lambda$, then $\ell(\gamma_n)\geq d(\gamma_n(0),\gamma_n(t_n))+d(\gamma_n(t_n),\gamma_n(1))\geq d(x,y)+\varepsilon_0$. The assumption that $\ell(\gamma_n)\to d(x,y)$ implies that there exists $n_0>0$ such that for all $n\geq n_0$, $\gamma_n([0,1])\subset \Lambda$.
Furthermore, since $\Lambda$ is compact under the localization condition (if  $\varepsilon_0>0$ is small enough), there exists $\eta>0$ such that for any $(x_0,y_0)\in \Lambda$ with $d(x_0,y_0)\leq \eta$, there exists a length-minimizing curve joining $x_0$ to $y_0$. Let $N=\lceil d(x,y)/\eta\rceil+1$, and consider the $N-1$ sequences $(\gamma_n(k/N))_{n\in \N}$, $k\in \{1,\dots,N-1\}$. They can all be assumed to converge, up to extraction, with limits $\lim\gamma_n(k/N)= z_k$, $k\in \{0,\dots,N\}$, so that $x=z_0$ and $y=z_N$. Now since $\gamma_n$ are constant speed admissible curves,
$d(\gamma_n(k/N),\gamma_n((k+1)/N))\to d(x,y)/N=d(z_k,z_{k+1})<\varepsilon_0$, for all $k\in \{0,\dots,N-1\}$. As a consequence for all $k\in \{0,\dots,N-1\}$, there exists a length-minimizing curve of length $d(x,y)/N$ joining $z_k$ to $z_{k+1}$, hence the existence of an admissible curve between $x=z_0$ and $y=z_N$ of length $d(x,y)$.
\end{proof}

The above argument also allows to show that under these assumptions, any point in $\Gamma(x,y)$ indeed belongs to the midpoint of a length-minimizing curve between $x$ and $y$.
If $\K$ avoids the diagonal $\D$ and no abnormal minimizers exist between any pair $(x,y)\in \K$, it follows that no point along a length-minimizing curve, except its endpoints, can be cut or conjugate (see, e.g. \cite[Theorem 8.52]{ABB_2018}) and thus the pairs $(x,z)$ and $(z,y)$, $z\in \Gamma(x,y)$ avoid $\C$ entirely.
This is the heart of
Molchanov's method, which separates heat kernels evaluated at pairs $(x,y)$ in $\K$ into products of heat kernels evaluated at pairs $(x,z)$ and $(z,y)$ with $z$ in the neighborhood of $\Gamma(x,y)$. In particular, even when $(x,y)\in \K\cap \C$, the length-minimizing curves joining them cannot be cut or conjugate at their midpoint and the heat kernel may still be described.
Localization hypotheses allows to say more on the compactness properties of midpoint sets.

\begin{lemma}\label{L:fat_compacts}
Let $\K$  be a localizable compact subset of $M^2$ such that $\K\cap \D=\emptyset$ and no strictly abnormal minimizers exist between any pair $(x,y)\in \K$. Let 
$$
\Gamma_{\eps}^l(\K)=
\left\{
(x,z)\mid \exists y\in M \st (x,y)\in \K, z\in \Gamma_\varepsilon(x,y)
\right\}, 
$$
$$
\Gamma_{\eps}^r(\K)=
\left\{
(z,y)\mid \exists x\in M \st (x,y)\in \K, z\in \Gamma_\varepsilon(x,y)
\right\}.
$$
(As with $\Gamma(x,y)$, we may denote $\Gamma_0^l(\K)$ by $\Gamma^l(\K)$ and $\Gamma_0^r(\K)$ by $\Gamma^r(\K)$.)
There exists $\varepsilon_1>0$ such that $\Gamma^l_\varepsilon(\K)$ and $\Gamma^r_\varepsilon(\K)$ are non-empty compact subsets of $M^2\setminus \C$ for all $\varepsilon\leq \varepsilon_1$.
\end{lemma}

\begin{proof}
By symmetry of the definitions, we prove the statement for $\Gamma^l_\varepsilon(\K)$.

By Lemma~\ref{L:GeodesicsExist} geodesics between pairs of points $(x,y)\in \K$ exist and all remain in a compact set in $M$. As such, for any pair $(x,y)\in \K$, the set $\Gamma(x,y)$ is non-empty. Since we naturally have $\Gamma^l_\varepsilon(\K)\subset \Gamma^l_{\varepsilon'}(\K)$ as soon as $0\leq\varepsilon\leq \varepsilon'$, all sets are non-empty.

Assume $\varepsilon\leq \varepsilon_0$, then we can prove compactness of $\Gamma^l_\varepsilon(\K)$. Indeed, denoting by $\pi:M^2\to M$ the continuous map such that $\pi_1(x,y)=x$, $\Gamma^l(\K)\subset\pi_1(\K)\times \Lambda$. Hence we only need to show that $\Gamma^l_\varepsilon(\K)$ is a closed subset of the compact set $\pi_1(\K)\times \Lambda$. Let $(x_n,z_n)_{n\in \N}\in\Gamma^l_\varepsilon(\K)$, converging towards $(x^*,z^*)\in \pi(\K)\times \Lambda$. For all $n$ there exists $y_n\in M$ such that $(x_n,y_n)\in \K$ and  $z_n\in \Gamma_\varepsilon(x_n,y_n)$. Since $\K$ is compact, $y_n$ can be assumed to converge (up to extraction) towards $y^*$.
By continuity of the sub-Riemannian distance over $M^2$, passing to the limit in $\max(d(x_n,z_n),d(z_n,y_n))\leq d(x_n,y_n)/2+\varepsilon$ implies $\max(d(x^*,z^*),d(z^*,y^*))\leq d(x^*,y^*)/2+\varepsilon$. Hence $(x^*y^*)\in \Gamma^l_\varepsilon(\K)$ and the set is thus compact.

Finally regarding the intersection with $\C$, the nesting property of the sets $\Gamma^l_{\varepsilon}(\K)$ means we can assume by contradiction that for all positive $n\in\N$, there exists $(x_n,z_n)\in \Gamma^l_{\varepsilon_0/n}(\K)\cap \C$. The sequence $(x_n,z_n)$ belongs to the compact $\Gamma^l_{\varepsilon_0}(\K)$, and can be assumed to converge up to extraction to $(x^*,z^*)\in \Gamma^l_{\varepsilon_0}(\K)$. Furthermore, for all $n$, there exists $y_n$ such that $(x_n,y_n)\in \K$ and $z_n\in \Gamma_{\varepsilon_0/n}(x_n,y_n)$. Since $\K$ is compact, $y_n$ can also be assumed to converge (up to extraction) towards $y^*$ such that $(x^*,y^*)\in \K$. Then, passing to the limit in $\max(d(x_n,z_n),d(z_n,y_n))\leq d(x_n,y_n)/2+\varepsilon_0/n$ yields $\max(d(x^*,z^*),d(z^*,y^*))\leq d(x^*,y^*)/2$, hence $z^*\in \Gamma(x^*,y^*)$. On the other hand, $\C\cap \Gamma^l_{\varepsilon_0}(\K)$ is also closed, meaning that $(x^*,z^*)\in \C$, which is in contradiction with the hypothesis that $\K\cap \D=\emptyset$ and there exist no strictly abnormal minimizers exist between any pair $(x,y)\in \K$. Hence the statement.
\end{proof}

Let us introduce a final family by of sets of $M^2\setminus \C$ that will prove useful for the extension of Ben Arous expansion theorem and prove some of their properties relying on the same idea as Lemma~\ref{L:fat_compacts}. For any compact set $\K\in M^2\setminus \C$, for any $\varepsilon>0,\eta>0$, the set $\U_{\varepsilon,\eta}$ is the set of points $(x,y)\in M^2$ such that $d(x,y)\geq \eta$ and such that $x$ and $y$ are both $\varepsilon$-close to length-minimizing curves linking a pair $(x',y')\in \K$ (in a sense to be made precise below).

These sets have the nice property of being compact supersets of $\K$, and, more importantly, that the midpoint sets generated with the sets $\U_{\varepsilon,\eta}$ themselves are contained within another set of the same family (assuming $\varepsilon$ is small enough). This is crucial to be able to apply Molchanov's method repeatedly.

\begin{lemma}\label{L:propU}
Let $\K$ be a localizable compact subset of $M^2\setminus \C$.
For any $\varepsilon\geq 0,\eta>0$, let 
$\U_{\varepsilon,\eta}\subset M^2$ be defined as follows:  $(x,y)\in \U_{\varepsilon,\eta}$ if $d(x,y)\geq \eta$ and  there exists $(x',y')\in \K$, such that 
$d(x',x)+d(x,y')\leq d(x',y')+ \varepsilon$ and  $d(x',y)+d(y,y')\leq d(x',y')+ \varepsilon$.

For every $\eta>0$ and  $0\leq \varepsilon\leq \varepsilon_0$, with $\varepsilon_0$ such that $\Lambda$ is compact, $\U_{\varepsilon,\eta}$ is compact. For every $\eta>0$, there exists $\varepsilon_1$ such that if $\varepsilon\leq \varepsilon_1$, $\U_{\varepsilon,\eta}\cap  \C=\emptyset$. Furthermore, for any $\varepsilon_2>0$, $\eta>0$,  there exists $\varepsilon>0$ such that
\begin{equation}\label{E:GammaU}
\Gamma_\varepsilon^l(\U_{\varepsilon,\eta})\subset  \U_{\varepsilon_2,\eta/3}
\quad\text{ and }\quad
\Gamma_\varepsilon^r(\U_{\varepsilon,\eta})\subset  \U_{\varepsilon_2,\eta/3}.
\end{equation}

\end{lemma}

\begin{proof}
Without loss of generality, we can assume $\eta $ small enough to not have to account for the case  $\U_{\varepsilon,\eta}=\emptyset$.

We prove compactness and non-intersection with $\C$ in a manner similar to Lemma~\ref{L:fat_compacts}.
If $\varepsilon\leq \varepsilon_0$ then $\U_{\varepsilon,\eta}\subset \Lambda^2$, and we only need to check closure of $\U_{\varepsilon,\eta}$ to get compactness.
Let $(x_n,y_n)$ be a sequence in $\U_{\varepsilon,\eta}$ converging to $(x^*,y^*)\in \Lambda^2$. By assumption, there exists $(x_n',y_n')\in \K$ such that 
$\max(d(x'_n,x_n)+d(x_n,y'_n),d(x'_n,y_n)+d(y_n,y'_n))\leq d(x'_n,y'_n)+ \varepsilon$. Since $\K$ is compact, up to extraction we have $(x'_n,y'_n)\to(x^\dagger,y^\dagger)\in \K$. Then by continuity of the sub-Riemannian distance, $\max(d(x^\dagger,x^*)+d(x^*,y^\dagger),d(x^\dagger,y^*)+d(y^*,y^\dagger))\leq d(x^\dagger,y^\dagger)+ \varepsilon$, proving that $(x^*,y^*)\in  \U_{\varepsilon,\eta}$.

Regarding the existence of the stated $\varepsilon_1$, it is sufficient to remark that for fixed $\eta$, the sets $\U_{\varepsilon,\eta}$ are nested with respect to $\varepsilon$ and thus assuming that  $\U_{\varepsilon,\eta}\cap  \C\neq\emptyset$ for all $\varepsilon$ would imply that $\U_{0,\eta}\cap  \bar{\C}\neq\emptyset$. However all points in $\U_{0,\eta}$ are pairs of points belonging to geodesics between pairs $(x,y)\in\K$. The fact that $\K\subset M^2\setminus \C$ implies that it is also also true for pairs of points along geodesics.

Let us prove \eqref{E:GammaU} by considering the existence for all $n\in \N$ of a pair $(x_n,z_n)\in \Gamma_{1/n}^l(\U_{1/n,\eta})\setminus  \U_{\varepsilon_2,\eta/3}$. (Once again, for fixed $\eta$, the sets are nested). It implies the existence of $(y_n)$ in $M$ such that $(x_n,y_n)\in  \U_{1/n,\eta}$, $z_n\in \Gamma_{1/n}(x_n,y_n)$. If $N$ is large enough, the first part of the statement applies to show that $\U_{1/N,\varepsilon}$ is a compact subset of $M^2\setminus \C$, and by definition of the sets $\U_{\varepsilon,\eta}$, it inherits the localizability from $\K$ if $N$ is large enough. Hence Lemma~\ref{L:fat_compacts} applies to show that  $\Gamma_{1/N}^l(\U_{1/N,\eta})$ is compact. Then for all $n\geq N$, $(x_n,z_n)\in \Gamma_{1/N}^l(\U_{1/N,\eta})$. By compactness, $(x_n,y_n,z_n)$ can be assumed to converge towards $(x^*,y^*,z^*)$, where $(x^*,y^*,z^*)$ all belong to a geodesic between a pair $x$ and $y$ such that $(x,y)\in \K$. Distance-wise, $d(x_n,y_n)\geq \eta$, hence $d(x_n,z_n)\geq \eta/2-1/n$ by triangular inequality. On the other hand, $d(x_n,x^*) \to 0$ and $d(z_n,z^*) \to 0$, hence both $x_n$ and $z_n$ become arbitrarily close to a geodesic between a pair in $\K$. By triangular identity, this means that for $n$ large enough, $(x_n,z_n)\in \U_{\varepsilon_2,\eta/3}$, contradicting the existence of the sequence and proving that \eqref{E:GammaU} holds for $\varepsilon$ small enough.

\end{proof}

\subsection{Uniform Ben Arous expansions}

{ Expansions of the heat kernel, as given by Ben Arous in \cite{BenArous}, hold for sub-Riemannian distribution over $\R^d$. This means that we are able to write these expansions for pairs of points  in a manifold as long as the points are close enough to a appear in the domain of the same chart. For pairs of points that are further apart,  Molchanov's method and localization naturally shows that almost all information can be gathered from integration of the heat kernel on small neighborhoods of the midpoints. Using Laplace integral asymptotics to derive Ben Arous expansions from the integral, we are effectively increasing the possible distance between pairs points by a fixed rate (slightly smaller than 2). Repeating this argument, we are able to prove that Ben Arous expansion holds for points arbitrarily far apart. Using compactness arguments, this allows to prove that the expansion holds uniformly on compacts sets of the manifold.

 The announced Theorem~\ref{T:UnifBenArous} can then be expressed as follows.

}

\begin{proposition}\label{P:BenArous}
If $\K$ is a localizable compact set in $M^2\setminus \C$, 
then  Ben Arous expansion holds uniformly on $\K$ (in the sense of Definition~\ref{D:Ben Arous}).
As a consequence, Theorem~\ref{T:UnifBenArous} holds.
\end{proposition}

This statement is obtained as a consequence of three lemmas. First, the Ben Arous expansions hold uniformly for points that are close enough to each other as a consequence of the original Ben Arous expansion theorem. Second, we use Molchanov's technique to show that if the statement holds for pairs of points sufficiently close in a compact, then we can increase this maximal distance by shaving off an arbitrarily small neighborhood of the border.
Finally, we tie things up by completing the statement on the derivatives of the remainders.
The proof of Proposition~\ref{P:BenArous} comes at the end of the section, as a conclusion of this sequence of lemmas.

\begin{lemma}\label{L:BenArousSmall}
Let $\K\subset M^2$ be a compact. There exists $\delta_0>0$ such that Ben Arous expansions hold uniformly on any compact set in {$\K\cap\D(\delta_0)\setminus \C$}.
\end{lemma}

\begin{proof}
Let $K=\pi_1(\K)\cup\pi_2(\K)\subset M$ be the compact set of points such that $x\in K$ if there exists $y\in M$ such that either $(x,y)\in \K$ or $(y,x)\in \K$.
For all $x\in K$, there exist $R_x>0$ and an isometry $\zeta_x:B(x,R_x)\to \R^d$ that maps $B(x,R_x)$ to a neighborhood of 0 in $\R^d$.
The family $\left( B(x,R_x/4)\right)_{x\in K}$ is an open cover of $K$, so we can extract a finite collection $(x_i)_{1\leq i\leq n}$, $R_i=R_{x_i}$, such that 
$K\subset \cup_{i=1}^n B(x_i,R_i/4)$.

Let $\delta_0=\min_{i}R_i/4$. For any $x\in K$ there exists an integer $i$, $1\leq i\leq n$, such that $x\in   B(x_i,R_i/4)$. For any $y\in M$ such that  $d(x,y)\leq \delta_0$, $y\in   B(x_i,R_i/2)$.
Hence for all pairs $(x,y)\in K^2\cap \{d\leq \delta_0\}$, there exists $i$, $1\leq i\leq n$, such that $x,y\in \bar B(x_i,R_i/2)\subset B(x_i,R_i)$.

For all $1\leq i\leq m$, let $\tilde{p}^i_t$ denote the heat kernel on $\zeta_{i}(B(x_i,R_i))$. As a consequence of Theorem~\ref{C:MolchanovV1},
there exists $\varepsilon_i$ such that uniformly
for all $(x,y)\in  B(x_i,R_i/2)\cap \D(\delta_0)$, 
$$
p_t(x,y)=\tilde{p}^i_t(\zeta_{i}(x),\zeta_{i}(y))+O\left(\e^{-\frac{d(x,y)^2+\varepsilon_i}{4t}}\right).
$$
The same holds for all the time and spatial derivatives of $p_t$. Then for all $1\leq i\leq m$, the Ben~Arous expansion  holds uniformly
on
${\bar B(x_i,R_i/2)}^2\cap \D(\delta_0)$ since they classically hold on the compact  ${\zeta_i(\bar B(x_i,R_i/2))}^2$. By taking the maximum of the uniform bounds on each ball, and the shortest time intervals,  we get that the Ben~Arous expansion hold uniformly on any compact subset of 
$\left[\cup_{i=1}^m \bar B(x_i,R_i/2)^2\right]\cap \D(\delta_0)$ that excludes $\C$.
\end{proof}

We now prove that we can expand the domain on which Ben Arous expansion holds. However we only partially prove that fact at first; the bounds on the derivatives of the remainder will be proved in the next lemma. In the remainder of this section, we give the proofs assuming $l$ is any non-negative integer. In the non-symmetric case, the results are given by taking $l=0$. Indeed, this reflects the fact that there is no problem in taking derivatives of the Ben Arous expansion per se, the difficulties only arise due to the lack of a sufficient version of the localization results and L\'eandre asymptotics, as manifested in the proof of Corollary \ref{C:MolchanovV1}.

With $\K$ be a localizable compact subset of $M^2\setminus \C$, for any $\eta>0$, $\varepsilon>0$ small enough,  we let $\U_{\varepsilon,\eta}$ be as defined in Lemma~\ref{L:propU}. In particular $\K\subset \U_{\varepsilon,\eta}\subset M^2\setminus\C$.
We introduce a partial statement of Ben Arous expansions, $\mathsf{P}(\varepsilon,\eta,\delta)$.
\\
\noindent
$\mathsf{P}(\varepsilon,\eta,\delta)$: 
\begin{minipage}[t]{.88\linewidth}
Let $\U=\U_{\varepsilon,\eta}\cap \D(\delta)$. There exists $\O\subset M^2 \setminus \C$, open neighborhood of $\U$ on which the following holds.
For all non-negative integer $l$ and  multi-index $\alpha$, 
there exist sequences of smooth functions $c_k^{l,\alpha }:\O\to \R$, $k\in \N$, $r_{k}^{l,\alpha }: \R^+\times \O\to \R$, such that for all $n\in \N$, for all $(x,y)\in \O$, for all $t$ small enough
$$
\partial_t^l
Z_y^{\alpha}
p_t(x,y)
=
t^{-(|\alpha|+2 l+d/2)} \e^{-\frac{d(x,y)^2}{4t}}
\left(
\sum_{k=0}^{n}c_k^{l,\alpha}(x,y)t^k
+
t^{n+1}r_{n+1}^{l,\alpha}(t,x,y)
\right),
$$
and, 
furthermore,
 there exists $t_0$ such that
$$
\sup_{0<t<t_0}
\sup_{(x,y)\in \U}
\left|
	r_{n+1}^{l,\alpha}(t,x,y)
\right|<\infty.
$$
\end{minipage}

In this definition, $\delta$ is the upper bound of distance between pairs of points, and $\eta$ the lower bound. In order to prove Proposition~\ref{P:BenArous}, we need $\mathsf{P}(\varepsilon,\eta,\delta)$ to hold true for $\delta\geq \max_{(x,y)\in \K} d(x,y)$ and  $\eta\leq \min_{(x,y)\in \K}$.
Below we prove that we can increase $\delta$ at the price of increasing $\eta$, but in the end this only means that we need to start with $\eta$ small enough.

\begin{lemma}\label{L:Molchanov2BenArous}
Let $\K$ be a localizable compact subset of $M^2\setminus \C$. 
If there exists $\delta>0$, $\eta_0>0$, $\varepsilon_0>0$  such that $\mathsf{P}(\varepsilon_0,\eta_0,\delta)$ holds true, then there exists $\varepsilon>0$ such that $\mathsf{P}(\varepsilon,3\eta_0,3\delta/2)$ also holds.

\end{lemma}

\begin{proof}

\noindent{\it Step 1: localization.}
Let $l$ be a non-negative integer, $\alpha$ be a multi-index.
Let $(t,x,y)\in \R^+\times \U_{\varepsilon_0,\eta_0}$. As an application of Corollary~\ref{C:MolchanovV1}, for any $\varepsilon>0$ small enough, we have uniformly on $ \U_{\varepsilon_0,\eta_0}$
$$
\partial_t^l
Z_y^{\alpha}
p_t(x,y)
=\int_{\Gamma_\varepsilon} p_{t/2}(x,z)Z_y^{\alpha}\partial_t^lp_{t/2}(z,y)\diff\mu( z)+O\left(\e^{-\frac{d(x,y)^2+\varepsilon^2/2}{4t}}\right).
$$

\noindent{\it Step 2: Ben Arous expansions on the midpoint set.} 
By application of Lemma~\ref{L:propU}, there exists $\varepsilon>0$ such that 
$$
\Gamma_\varepsilon^l(\U_{\varepsilon,3\eta_0})\subset  \U_{\varepsilon_0,\eta_0}
\quad\text{ and }\quad
\Gamma_\varepsilon^r(\U_{\varepsilon,3\eta_0})\subset  \U_{\varepsilon_0,\eta_0}.
$$
In the following, we denote $\U=\U_{\varepsilon,3\eta_0}\cap \D(3\delta/2)$ and $\U'=\U_{\varepsilon_0,\eta_0}\cap \D(\delta)$.  As long as $\varepsilon<\delta/4$, we still have
$$
\Gamma_\varepsilon^l(\U)\subset \U'
\quad\text{ and }\quad
\Gamma_\varepsilon^r(\U)\subset \U',
$$
and we assume that $\mathsf{P}(\varepsilon_0,\eta_0,\delta)$ holds on $\U'$. Furthermore, with $\O$ the interior of $\U_{\varepsilon+\zeta,3\eta_0-\zeta}\cap \D(3\delta/2+\zeta)$, we have for $\zeta$ small enough that $\O\subset M^2\setminus\C$, $\U\subset \O$ and $\Gamma^l_\varepsilon(\O)\subset \U'$. We will now prove that $\mathsf{P}(\varepsilon,3\eta_0,3\delta/2)$ holds for this choice of $\U$ and $\O$.

Since  $\mathsf{P}(\varepsilon_0,\eta_0,\delta)$ holds, there exists an open set $\O'\subset M^2\setminus \C$, $\U'\subset \O'$, and smooth functions $c_k^{l,\alpha}:\O'\to \R$, $k\in \N$, $r_{k}^{l,\alpha}: \R^+\times \O'\to \R$,  satisfying $\mathsf{P}(\varepsilon_0,\eta_0,\delta)$.
For all $(x,y)\in \O'$, for all $t\in \R^+$, we denote
$$
\Sigma_t^{l,\alpha}(x,y)
=
\sum_{k=0}^{n}c_k^{l,\alpha}(x,y)t^k
+
t^{n+1}r_{n+1}^{l,\alpha}(t,x,y).
$$
There also exists $t_0$ such that
$$
\sup_{0<t<t_0}
\sup_{(x,y)\in \U'}
\left|
	r_{n+1}^{l,\alpha}(t,x,y)
\right|<\infty.
$$

Then, by construction, it uniformly holds for all $(x,y)\in \O$ that
\begin{multline}\label{E:dervative_sum}
{t^{d/2+2l+|\alpha|}}\e^{\frac{d(x,y)^2}{4t}}
\partial_t^l
\partial_y^{\alpha}
p_t(x,y)
=
\\
\frac{2^{d+2l+|\alpha|}}{t^{d/2}}
\int_{\Gamma_\varepsilon}
		 \e^{-\frac{h_{x,y}(z)-d(x,y)^2/4}{t}}
		 \Sigma_{t/2}^{0,0}(x,z)
		 \Sigma_{t/2}^{l,\alpha}(z,y)
\diff\mu( z)
+
O\left(\e^{-\frac{\varepsilon^2}{8t}}\right)
\end{multline}
with $h_{x,y}(z)=(d(x,z)^2+d(z,y)^2)/2$.

\noindent{\it Step 3: Cauchy product rearrangement of expansions.}
To alleviate the notations, we write, for integers $0\leq i\leq n$,
$$
a_i=c_i^{0,0},\qquad b_i=c_i^{l,\alpha}, \qquad r^a=r^{0,0}_{n+1},  \qquad r^b=r^{l,\alpha}_{n+1}.
$$

By rearranging terms in the sums, we have
\begin{equation}\label{E:rearrangement}
 \Sigma_{t/2}^{0,0}( x,z) \Sigma_{t/2}^{l,\alpha}( z,y)
	=
\sum_{k=0}^{n}\left(\frac{t}{2}\right)^k
\left(
	\sum_{i=0}^k a_i(x,z)b_{k-i}(z,y)
\right)
+
\left(\frac{t}{2}\right)^{n+1}
\Phi_{n+1}(t,x,y,z),
\end{equation}
with the explicit remainder
\begin{multline*}
\Phi_{n+1}(t,x,y,z)=
r^{a}(t/2,x,z)r^{b}(t/2,z,y)
+
%\\
\sum_{k=1}^{n+1}\left(\frac{t}{2}\right)^k
\Bigg[
	r^{a}(t/2,x,z)b_{n+1-k}(z,y)
\\+
	\sum_{i=1}^{k-1} 
		\bigg(
			a_{n+1-i}(x,z)
			b_{n+1-k+i}(z,y)
		\bigg)
%\\
+
	a_{n+1-k}(x,z)r^{b}(t/2,z,y)
\Bigg].
\end{multline*}

\noindent

{\it Step 4: Normal form of the hinged energy functional in charts.}
We wish to cover the set $\O$ with a finite number of charts allowing to put $h_{x,y}$ in normal form.
For pairs $(x,y)\in M^2\setminus\C$, $\Gamma(x,y)$ is reduced to a single point that we denote $z_0$. The hinged energy functional $h_{x,y}$ reaches a unique minimum at $z_0$, $h_{x,y}(z_0)=d(x,y)^2/4$,  and $(x,y)\mapsto z_0(x,y)$ is smooth on the open set $M^2\setminus\C$.

For any pair $(x,y)\in M^2\setminus\C$, the Hessian of $h_{x,y}$ is always positive definite at $z_0$. Then for any $(x,y)\in M^2\setminus\C$, we can apply the Morse--Bott Lemma to $h:(x,y,z)\mapsto h_{x,y}(z)-h_{x,y}(z_0)$ near the point $(x,y,z_0)$ (see, for instance, \cite[Theorem 2]{BanyagaHurtubise}). This implies the existence of a neighborhood $U_{x,y}$ of $(x,y,z_0)$ and a chart $\xi:U_{x,y}\to \R^{3d}$, with $\xi(x',y',z')=(u_1,\dots, u_{3d})$, such that $\xi (x,y,z_0(x,y))=0$, $\xi(x',y',z_0(x',y'))=(0,\dots,0,u_{d+1},\dots, u_{3d})$ for all $(x',y',z')\in U_{x,y}$ and $h(\xi^{-1}(u))=u_1^2+\dots +u_d^2$.
In particular,  for $(x',y')$ in a small enough neighborhood of $(x,y)$ so that $(x',y',z_0(x',y'))\in U_{x,y}$, the map $\xi_{x',y'}=\xi(x',y',\cdot)$ charts $\Gamma_\varepsilon(x',y')$ for $\varepsilon$ small enough and $h_{x',y'}(\xi_{x',y'}^{-1}(u_1,\dots,u_d))=u_1^2+\dots +u_d^2$ for all $(u_1,\dots,u_d)\in \xi_{x',y'}(\Gamma_\varepsilon(x',y'))$.

Since the closure of $\O$ is a compact subset of $M^2\setminus\C$, there exists a finite collection $(x_1,y_1),\dots,$ $(x_N,y_N)$ such that the union $\cup_{i=1}^N U_{x_i,y_i}$ covers $\{(x,y,z_0(x,y)) \mid (x,y)\in \bar{\O}\}$. By compactness, up to reducing $\varepsilon$, we can assume that for any $(x,y)\in \bar\O$, there exists $i\in \{1,\dots, N\}$ such that $\{(x,y)\}\times \Gamma_{\varepsilon}(x,y)\subset U_i$. 
With $\mathcal{V}_i=\{(x,y)\in \O\mid \{(x,y)\}\times \Gamma_{\varepsilon}(x,y)\subset U_i\}$,
this allows to set for all $(x,y)\in \mathcal{V}_i$, $\xi_{x,y}:\Gamma_{\varepsilon}(x,y)\to \R^d$, smoothly varying with respect to $(x,y)$ in $\mathcal{V}_i$, such that 
$$
h_{x,y}\circ \xi_{x,y}^{-1}(u_1,\dots, u_d)=h_{x,y}(z_0)+u_1^2+\dots +u_d^2, \qquad \forall u\in \xi_{x,y}(\Gamma_\varepsilon(x,y)), \forall (x,y)\in \mathcal{V}_i.
$$

We now prove the existence and properties of the expansion on each of the open sets $\mathcal{V}_i$. The precise expression we obtain necessarily depends on the set $\mathcal{V}_i$, but all properties in the intersections $\mathcal{V}_i\cap \mathcal{V}_j$ follow from the chain rule between different charts.

{\it Step 5: Laplace integrals asymptotics in charts.}
We now compute asymptotic expansions of Laplace integrals, following \cite{EAndK}.  We focus on a specific $\mathcal{V}_i$ and its associated map $\xi_{x,y}$. Let us denote by $\nu_{x,y}$ the density of $(\xi_{x,y})_*\mu$ with respect to the Lebesgue measure on $\R^d$. The density function $(x,y,u)\mapsto \nu_{x,y}$ is smooth and non-vanishing.
For any smooth  function $\varphi:M\to \R$,
$$
\int_{\Gamma_\varepsilon}
\e^{-\frac{h_{x,y}(z)-h_{x,y}(z_0)}{t}}
\varphi(z)\diff\mu( z)
=
\int_{\xi_{x,y}(\Gamma_\varepsilon)}
\e^{-|u|^2/t}
\varphi\circ\xi_{x,y}^{-1}(u)\nu_{x,y}(u)\diff u.
$$

We can recognize a Laplace integral, and we follow \cite{EAndK} for its asymptotic study at $t=0$. In particular, for $f:\R\to\R$, we borrow the notation $
f(t)\simeq \sum_{n=0}^\infty t^n f_n
$
when
$f(t)=\sum_{n=0}^{N} t^n f_n+O(t^{N+1})$ for all $N>0$. 
Note that  if  a map $f:\R^+\times\R^d\to \R$ is such that 
$
f(t,x)\simeq \sum_{n=0}^\infty t^n f_n(x)
$ for all $x$ in an open domain,
and $f_n(x)$ is smooth for all $n$ then $f$ is actually smooth at $(0,x)$ (on the right in $t$), implying uniformity on compacts of the remainders.

From \cite[Equation~(4.36)]{EAndK}, we have  for any smooth map $\phi:\R^d\to \R$ the Laplace integral asymptotic expansion
\begin{equation}\label{E:Laplace_asymptotics_original}
\int_{\R^d}
\e^{-|u|^2/t}
\phi(u)\diff u
\simeq
\sum_{N=0}^{\infty}\frac{(\pi t)^{d/2} t^N}{2^{2N}}
\sum_{|\omega|=N}
\frac{1}{\omega!} \partial^{2\omega} \phi(0),
\end{equation}
where 
$\omega\in \N^{d}$ is a multi-index $(\omega_1,\dots, \omega_d)$ such that $|\omega|=\sum_{i=1}^d\omega_i$, $2\omega=(2\omega_1,\dots, 2\omega_d)$ and 
$\omega!=\prod_{i=1}^{d}\omega_i!$.

%{
Equation~\eqref{E:Laplace_asymptotics_original} holds for smooth compactly supported functions on $\R^d$, only derivatives at $0$ appear in the expansion. Notice that the compactness assumptions imply the existence of $R>0$ such that $B_{\R^d}(0,R)\subset \xi_{x,y}(\Gamma_\varepsilon)$, for all $(x,y)\in \mathcal{V}_i$, thus the derivatives at $0$ are well defined even when we restrict the integral to $\Gamma_\varepsilon$, and we get the same expansion for any smooth compactly supported continuation
of $\varphi\circ\xi_{x,y}^{-1}(u)\nu_{x,y}(u)$ outside of $\xi_{x,y}(\Gamma_\varepsilon)$. In other terms this implies,
\begin{equation}\label{E:Laplace_asymptotics}
\int_{\Gamma_\varepsilon}
\e^{-\frac{h_{x,y}(z)-h_{x,y}(z_0)}{t}}
\varphi(z)\diff\mu( z)
\simeq
\sum_{N=0}^{\infty}\frac{(\pi t)^{d/2} t^N}{2^{2N}}
\sum_{|\omega|=N}
\frac{1}{\omega!} {\left.\partial^{2\omega}\right|}_{u=0}\left(\varphi\circ \xi_{x,y}^{-1}(u)\nu_{x,y}(u)\right)
\end{equation}

To conclude the proof, we apply this expansion on the elements we exhibited in \eqref{E:rearrangement}.

\noindent
{\it Step 6: remainder.}
First, we consider the remainder 
$$
\Psi_{n+1}(t,x,y)=
t^{-d/2}\int_{\Gamma_{\varepsilon}}
\e^{-\frac{h_{x,y}(z)-h_{x,y}(z_0)}{t}}
\Phi_{n+1}(t,x,y,z)\diff \mu(z).
$$

It is a smooth function on $\R^+\times\O$. Let us prove that it is uniformly bounded on $\U$.
As a consequence of the discussion in step 2, if $z\in \Gamma_\varepsilon(x,y)$, then  $(x,z)\in  \Gamma^l_\varepsilon(\U)$ and  $(z,y)\in  \Gamma^r_\varepsilon(\U)$, both  subsets of $\U'$. Hence, as a consequence of $\mathsf{P}(\varepsilon_0,\eta_0,\delta)$, 
there exists $t_{n+1}$ such that 
$$
\sup_{0<t<t_{n+1}}
\sup_{(x,z)\in \overline{\Gamma^l_\varepsilon(\U)}}
\big|
	r^{a}(t,x,y)
\big|<\infty
\qquad
\text{ and }
\qquad
\sup_{0<t<t_{n+1}}
\sup_{(z,y)\in \overline{\Gamma^r_\varepsilon(\U)}}
\big|
	r^{b}(t,x,y)
\big|<\infty.
$$
This implies that 
$$
\sup
\left\{
|\Phi_{n+1}(t,x,y,z)|
\mid
t\in(0,t_{n+1}),(x,y)\in \U,z\in \Gamma_\varepsilon(x,y)
\right\}
< A<\infty.
$$
Then, applying \eqref{E:Laplace_asymptotics}, we get that for all  $(x,y)\in \U$, for all $0<t< t_{n+1}$,
$$
|\Psi_{n+1}(t,x,y)|
\leq 
t^{-d/2}\int_{\Gamma_{\varepsilon}}
\e^{-\frac{h_{x,y}(z)-h_{x,y}(z_0)}{t}}
A\diff \mu(z)=\pi^{d/2} A +O(t)
$$
Hence the boundedness of $\Psi_{n+1}$.

\noindent
{\it Step 7: summands.}
For $k\in \{0,\dots,n\}$, for $N\in\{k,\dots, n\}$, we denote by $\psi_N^k$ the smooth function on $\mathcal{V}_i$ such that
$$
\psi_N^k(x,y)=
\frac{\pi^{d/2} }{2^{2(N-k)}}
\sum_{|\omega|=N-k}
\frac{2^{d/2}	}{\omega!} {\left.\partial^{2\omega}\right|}_{u=0}
\left( \nu_{x,y}(u)
\sum_{i=0}^k 
a_i(x,\xi_{x,y}^{-1}(u))b_{k-i}(\xi_{x,y}^{-1}(u),y)
\right).
$$
Following Equation~\eqref{E:Laplace_asymptotics}, for all
$
(t,x,y)\in \R^+\times \mathcal{V}_i
$,
$$
t^{k-d/2}\int_{\Gamma_\varepsilon}
\e^{-\frac{h_{x,y}(z)-h_{x,y}(z_0)}{t}}
\left(
\sum_{i=0}^k a_i(x,z)b_{k-i}(z,y)
\right)\diff\mu( z)=
\\
\sum_{N=k}^{n}t^N
\psi_N^k(x,y)
+
t^{n+1}\Psi_k(t,x,y)
$$
where $\Psi_k(t,x,y)$ is a smooth function  on 
$
\R^+\times \mathcal{V}_i
$
and  there exists $t_k>0$ such that
$$
\sup_{0<t<t_k}
\sup_{(x,y)\in \U}
\left|
	\Psi_{k}(t,x,y)
\right|<\infty.
$$
Then, plugging these sums in \eqref{E:rearrangement} yields
$$
\frac{1}{t^{d/2}}
\int_{\Gamma_\varepsilon}
		 \e^{-\frac{h_{x,y}(z)-d(x,y)^2/4}{t}}
		  \Sigma_{t/2}^{0,0}(x,z)
		  \Sigma_{t/2}^{l,\alpha}(z,y)
\diff \mu(z)
=
\sum_{N=0}^n t^N\left(\sum_{k=0}^N \psi_{N}^k(x,y)\right)
+
 t^{n+1}\sum_{k=0}^{n+1}\Psi_k(t,x,y).
$$
By construction, the remainder $\sum_{k=0}^{n+1}\Psi_k(t,x,y)$ is uniform. Building this expansion on each of the open sets $\mathcal{V}_i$ introduced in step 4, in conjunction with \eqref{E:dervative_sum}, yields that $\mathsf{P}(\varepsilon,3\eta_0,3\delta/2)$ holds.
\end{proof}

\begin{lemma}\label{L:BenArousDerivatives}
If $\mathsf{P}(\varepsilon,\eta,\delta)$  holds then all the derivatives in $(t,y)$ of the remainders are also uniformly  bounded.
\end{lemma}

\begin{proof}
Let $\U=\U_{\varepsilon,\eta}\cap \D(\delta)$, $\O\subset M^2\setminus \C$ be an open neighborhood of $\U$  and let $\psi(t,x,y):\R^+\times \O\to \R$ be such that there exist sequences of smooth functions $a_k:\O\to \R$, $k\in \N$, $\rho_{k}: \R^+\times \O\to \R$,  such that 
$$
\psi(t,x,y)
=
{t^{-d/2}}\e^{-\frac{d(x,y)^2}{4t}}
\left(
\sum_{k=0}^{n}a_k(x,y)t^k
+
t^{n+1}\rho_{n+1}(t,x,y)
\right),
$$
and there exists $t_0>0$ such that $n\in \N$,
$$
\sup_{0<t<t_0}
\sup_{(x,y)\in \U}
\left|
	\rho_{n}(t,x,y)
\right|<\infty.
$$

Assume there also exist sequences of smooth functions $b_k:\O\to \R$, $k\in \N$, $\bar{\rho}_{k}: \R^+\times \O\to \R$,  such that 
$$
Z_y^i\psi(t,x,y)
=
{t^{-d/2-1}}\e^{-\frac{d(x,y)^2}{4t}}
\left(
\sum_{k=0}^{n}b_k(x,y)t^k
+
t^{n+1}\bar{\rho}_{n+1}(t,x,y)
\right)
$$
and for all $n\in \N$,
$$
\sup_{0<t<t_0}
\sup_{(x,y)\in \U}
\left|
	\bar\rho_{n}(t,x,y)
\right|<\infty.
$$

Then
$$
\sup_{0<t<t_0}
\sup_{(x,y)\in \U}
\left|
	Z^i_y\rho_{n}(t,x,y)
\right|<\infty.
$$
Indeed, we have 
$$
Z^i_y\left({t^{d/2}}\e^{\frac{d(x,y)^2}{4t}}\psi(t,x,y)\right)
=
\sum_{k=0}^{n} t^k Z^i_y a_k(x,y)
+
t^{n+1} Z^i_y\rho_{n+1}(t,x,y)
.
$$
On the other hand, by pushing the expansion to one more order,
$$
\begin{aligned}
Z^i_y\left({t^{d/2}}\e^{\frac{d(x,y)^2}{4t}}\psi(t,x,y)\right)
=&
t^{d/2}\e^{\frac{d(x,y)^2}{4t}}
\left(
\frac{Z^i_yd(x,y)^2}{4t}
\psi(t,x,y)
+
Z^i_y\psi(t,x,y)
\right)
\\
=&
\sum_{k=0}^{n+1}t^{k-1}\left(\frac{Z^i_y d(x,y)^2 }{4}a_k(x,y)+b_k(x,y)\right)
\\
&+
t^{n+1}
\left(
	\frac{Z^i_y d(x,y)^2 }{4}\rho_{n+2}(t,x,y)
	+
	\bar{\rho}_{n+2}(t,x,y)
\right)
\end{aligned}
$$

Other than compatibility conditions such as $a_0(x,y)Z^i_yd(x,y)^2=-4b_0(x,y)$, we have 
$$
Z^i_y\rho_{n+1}(t,x,y)=\frac{Z^i_y d(x,y)^2 }{4}\rho_{n+2}(t,x,y)
	+
	\bar{\rho}_{n+2}(t,x,y)
$$
Similar expressions can be derived for derivatives with respect to $t$ following the same reasoning.
Chaining these arguments for both variables in all orders implies the statement. 
\end{proof}

As a conclusion to the section, we can finally prove Proposition~\ref{P:BenArous}.
\begin{proof}[Proof of Proposition~\ref{P:BenArous}]
Let $\bar{\eta}=1/2\min_{x,y\in \K}d(x,y)$, and let $\bar{\delta}=2\max_{x,y\in \K}d(x,y)$. We prove that there exists $\varepsilon>0$ such that 
such that $\mathsf{P}(\varepsilon,\bar\eta,\bar\delta)$ holds, with $\mathsf{P}$ introduced in Lemma~\ref{L:Molchanov2BenArous}. Once this is proved, Lemma~\ref{L:BenArousDerivatives} then implies Proposition~\ref{P:BenArous}.

Consider the set $\U_{1,0}\subset M^2$, the set of pairs $(x,y)$ such that there exists a strongly normal length-minimizing curve $\gamma:[0,1]\to M$, with $(\gamma(0),\gamma(1))\in \K$, and $d(\gamma,x)\leq 1$, $d(\gamma,y)\leq 1$. It is a compact set, hence Lemma~\ref{L:BenArousSmall} applies: there exists $\delta_0>0$ such that for any compact set contained in $\U_{1,0}\cap \C$, Ben Arous expansions hold uniformly. As a consequence, for any $\eta$, $\mathsf{P}(1,\eta,\delta_0)$ holds.

Let $m\in \N$ be such that $(3/2)^m\delta_0\geq \bar\delta$, and let $\eta_0=\bar\eta/3^{m}$. We have that $\mathsf{P}(1,\eta_0,\delta_0)$ holds. Applying Lemma~\ref{L:Molchanov2BenArous} $m$ times yields that there exists $\varepsilon>0$ such that $\mathsf{P}(\varepsilon, 3^m\eta_0,(3/2)^m\delta_0)$ holds. Consequently, $\mathsf{P}(\varepsilon, \bar \eta,\bar \delta)$ also holds and we have proved the statement.

If Ben Arous expansion holds uniformly for any $\K\in M^2\setminus \C$, then Theorem~\ref{T:UnifBenArous} follows. The issue is the existence and smoothness of the functions $c_k^{l,\alpha}$, $r_k^{l,\alpha}$ on the full set $M^2\setminus \C$, for all $k,l\in \N$, $\alpha$ multi-index. However by covering $M^2\setminus \C$ with compacts, since the functions give an expansion of the heat kernel, we  finally get the statement.
\end{proof}

\subsection{Uniform universal bounds on the heat kernel}

{A natural application of Ben Arous expansions are a priori universal bounds on the heat kernel, that come as a direct consequence of Molchanov method in form of a Laplace integral. These are stated in Proposition \ref{P:universal_bounds}, which we prove in a moment.

As a first step, we can give a proof of Corollary~\ref{C:MolchanovLaplace} of Ben Arous expansions. This is a refinement of our statement of Molchanov method, where we take into account the existence of Ben Arous expansions. This result is the basis for the estimates that follow.

\begin{proof}[Proof of Corollary~\ref{C:MolchanovLaplace}]
Starting from Corollary~\ref{C:MolchanovV1}, we have on the compact $\K\subset M^2\setminus \D$ that, for $\eps>0$ small enough, 
$$
\partial_t^lZ_y^\alpha p_t(x,y)=\int_{\Gamma_\varepsilon} 
	p_{t/2}(x,z)
	\partial_t^l
Z_y^{\alpha}p_{t/2}(z,y)
\diff\mu( z)
+
O\left(\e^{-\frac{d(x,y)^2+{\varepsilon^2/2}}{4t}}\right).
$$
As stated in Lemma~\ref{L:fat_compacts}, $\Gamma_\varepsilon$ avoids the cut loci of $x$ and $y$ entirely for $\varepsilon$ small enough (with $\varepsilon$ uniform over $\K$). Using notations from Lemma~\ref{L:fat_compacts}, this implies that uniform Ben Arous extensions hold on $\Gamma_\varepsilon^l(\K)$ and $\Gamma_\varepsilon^l(\K)$. As a consequence, for any $(x,y)\in \K$, $z\in \Gamma_\varepsilon(x,y)$, 
$$
p_{t/2}(x,z)
=
\left(\frac{2}{t}\right)^{d/2}\e^{-\frac{d(x,z)^2}{2 t}}\Sigma_t^{0,0}(x,z) 
$$
and
$$
\partial_t^l
Z_y^{\alpha}p_{t/2}(z,y)
=
\left(\frac{2}{t}\right)^{d/2+2l+|\alpha|}\e^{-\frac{d(z,y)^2}{2 t}}\Sigma_t^{l,\alpha}(z,y).
$$
With $h_{x,y}(z)=\frac{1}{2}\left(d(x,z)^2+d(y,z)^2\right)$, we obtain the stated formula.
\end{proof}

In the case of sub-Riemannian manifolds, these estimates were initially proved in \cite{BBN-JDG}. The approach for the proof is similar, however  we extend this result by showing the existence of uniform bounds on compact subsets where no two distinct points are joined by abnormal minimizers. This mostly requires a careful setup of coordinates for uniform comparison of the hinged energy functional with simple polynomial functions. Lower bounds do not hold for spatial derivatives due to the necessity of non-vanishing terms in the Ben Arous expansion, which is only guaranteed for time derivatives of the kernel.

}

\begin{proof}[Proof of Proposition \ref{P:universal_bounds}]

We start from Corollary~\ref{C:MolchanovLaplace}, where we have uniformly for $(t,x,y)\in\R^+\times \K$, and for $\varepsilon>0$ small enough that
$$
\partial_t^l
Z_y^{\alpha}
p_t(x,y)
=
\left(\frac{2}{t}\right)^{|\alpha|+2 l+d}
\int_{\Gamma_\varepsilon} 
\e^{-\frac{h_{x,y}(z)}{t}}
\Sigma_{t/2}^{0,0}(x,z)
\Sigma_{t/2}^{l,\alpha}(z,y)
\diff\mu( z)
+
O\left(\e^{-\frac{d(x,y)^2+\varepsilon^2/2}{4t}}\right).
$$
By  Theorem~\ref{T:UnifBenArous} and Lemma~\ref{L:fat_compacts}, $\Sigma_{t/2}^{0,0}$ and $\Sigma_{t/2}^{l,\alpha}$ are upper bounded on the compacts $\Gamma_\varepsilon^l(\K)$ and $\Gamma_\varepsilon^r(\K)$ respectively. Likewise, $\Sigma_{t/2}^{0,0}$ and $\Sigma_{t/2}^{l,0}$ are positively lower bounded on the same compacts. Hence there exists $\bar{c},\ubar{c}>0$ such that, for $t>0$ small enough,
$$
\partial_t^l
\Sigma_{t/2}^{0,0}(x,z)
\Sigma_{t/2}^{l,\alpha}(z,y)
\leq \bar{c} 
\quad 
\text{ and }
\quad 
\partial_t^l
\Sigma_{t/2}^{0,0}(x,z)
\Sigma_{t/2}^{l,0}(z,y)
\geq \ubar{c}.
$$
What remains to show to prove the statement is that there exist $\bar{m},\ubar{m}>0$ such that for all $(x,y)\in \K$ and $t$ small enough
\begin{equation}\label{E:a_priori_Laplace}
	\ubar{m} t^{d/2}
\leq 
	\int_{\Gamma_\varepsilon} 
	\e^{-\frac{h_{x,y}(z)-d(x,y)^2/4}{t}}
	\diff\mu( z)
\leq 
	\bar{m} t^{1/2}.
\end{equation}

We only need to  show that \eqref{E:a_priori_Laplace} holds for each of the elements of a finite cover of the compact $\K$.

For all $x\in M$, let us denote by $\sre_x:T^*_xM\to M$ the sub-Riemannian exponential (at time 1).  Let  $(x_0,y_0)\in \K$. There exists $\eta>0$ such that $TM$ can be trivialized on $B(x_0,2\eta)$, that is $T^*M\simeq  M\times \R^d$. We affix on $\R^d$ a Euclidean structure $|\cdot|$.
We now pull back the set $\Gamma^l_\varepsilon(\K)$ through the exponential. Let
$$
V_{x_0}^\eta=\left\{
	(x,p)\in \bar{B}(x_0,\eta)\times \R^d 
	\mid
	\exists y\in M\st (x,y)\in\K, \sre_x(p)\in \Gamma_\varepsilon(x,y)
\right\}.
$$
For all $x\in \bar{B}(x_0,\eta)$, we also denote by $V_x$ be the set of coverctors $p$ such that $(x,p)\in V_{x_0}^\eta$. 

Once again, $\varepsilon$ has been chosen small enough so that $\Gamma_\varepsilon$ avoids the cut loci of $x$ and $y$. This implies that $\sre_x$ is a diffeomorphism when restricted to $V_x$. Then $V_{x_0}^\eta$ is also the intersection of the closed set $\bar{B}(x_0,\eta)\times \R^d$ with the image of $\Gamma_\varepsilon^l(\K)$ by the smooth map $(x,z)\mapsto (x,\sre_x^{-1}(z))$.
This shows that $V_{x_0}^\eta$ is compact.

Let us denote by $\lambda_{\R ^d}$ the Lebesgue measure on the trivialized fibers of $T^*M$.
The map $\sre_x$ is a diffeomorphism from $V_x$ onto its image, hence there exists a function $\nu_x$ that is the density of $(\sre_x^{-1})_*\mu$ with respect to the measure $\lambda_{\R^d}$ on $V_x$.
Furthermore, 
since $(x,p)\mapsto \sre_x(p)$ is uniformly continuous in $(x,p)$ over $V_{x_0}^\eta$, we also have
$$
\ubar{\nu}=\inf_{(x,p)\in V_{x_0}^\eta} \nu_x(p)>0, \qquad \text{ and } \qquad \bar{\nu}=\sup_{(x,p)\in V_{x_0}^\eta} \nu_x(p)<\infty.
$$

Then for any smooth function $\varphi:M\to \R$, any $(x,y)\in \K$, $x\in \bar{B}(x_0,\eta)$,
\begin{equation}\label{E:BoundIntegrals}
\ubar{\nu}\int_{V_x}\varphi(\sre_x(p)) \diff  p
\leq
 \int_{\Gamma_\varepsilon}\varphi(z)\diff \mu( z) 
\leq 
\bar{\nu}\int_{V_x}\varphi(\sre_x(p)) \diff  p.
\end{equation}

We use \eqref{E:BoundIntegrals} to prove \eqref{E:a_priori_Laplace} by providing uniform upper and lower polynomial bounds of $h_{x,y}(z)-d(x,y)^2/4$. (Recall that for all points $z\in M$, $h_{x,y}(z)\geq \frac{d(x,y)^2}{4}$.)
First the upper bound.

By compactness of $V_{x_0}^\eta$, there exists $\rho>0$ such that for any $(x,y)\in \K$ with $x\in \bar{B}(x_0,\eta)$, and $p_0\in \sre_x^{-1}(\Gamma(x,y))$, $\bar B(p_0,\rho)\subset V_x$.
On $\bar B(p_0,\rho)$, we use a $d$-dimensional Taylor expansion:
$$
\left|
	h_{x,y}\circ\sre_{x}(p)
	- 
	h_{x,y}\circ\sre_{x}(p_0)
\right|
 \leq 
2\sup_{p_1\in \bar B(p_0,\rho)} \left\| \mathrm{Hess} (h_{x,y} \circ\sre_{x})(p_1)\right\| |p-p_0|^2.
$$
Again compactness implies that 
$\kappa=2\sup\left\{\left\| \mathrm{Hess} (h_{x,y} \circ\sre_{x})(p)\right\| \mid (x,p)\in V_{x_0}^\eta \right\}$ is finite, and with $h_{x,y}\circ\sre_{x}(p_0)=\frac{d(x,y)^2}{4}$,
$$
	h_{x,y}(p)-\frac{d(x,y)^2}{4}\leq \kappa| p-p_0|^2, \qquad \forall p\in \bar B(p_0,\rho).
$$
Then 
$$
\int_{\Gamma_\varepsilon} 
	\e^{-\frac{h_{x,y}(z)-d(x,y)^2/4}{t}}
	\diff\mu( z)
 \geq
 \ubar{\nu}\int_{B(p_0,\rho)} \e^{- \kappa\frac{ | p-p_0|^2}{t}}\diff p .
$$
As a classical application of Laplace integrals asymptotics, as $t$ goes to $0$
$$
\int_{B(p_0,\rho)} \e^{-\kappa \frac{|p-p_0|^2}{t}}\diff p
\sim 
\frac{{(2 \pi t)}^{d/2}}{2\kappa^{d/2}}.
$$
Hence the left-hand side of \eqref{E:a_priori_Laplace}.

Now let us give a lower bound of $h_{x,y}(z)-d(x,y)^2/4$.
Following \cite{BBN-JDG}, the triangular inequality 
implies
$$
h_{x,y}\left(z\right)-\frac{d(x,y)^2}{4}\geq \left(d(x,z)-\frac{d(x,y)}{2}\right)^2.
$$
Then we use polar-type coordinates to describe $d(x,z)$. Let $H:TM\to \R$ be the sub-Riemannian Hamiltonian. Since $\Gamma_\varepsilon$ avoids the cut locus, for all $(x,p)\in  V_{x_0}^\eta$, we have $d(x,\sre_x(p))=\sqrt{2 H(x,p)}$. In particular, $ H(x,p)\neq 0$. Furthermore, for any $s>0$ such that $(x,sp)\in  V_{x_0}^\eta$, $d(x,\sre_x(s p))=s d(x,\sre_x( p))$.  Hence we represent the set $\{H\neq 0\}$ in the fibers with $\Phi_x:\R^+\times \{H=1/2\}\to \R^d $ such that $\Phi_x(s,q)=s q$.
Using again that $V_{x_0}^\eta$ is a compact set, we have
$$
\ubar{s}=\inf_{V_{x_0}^\eta} \sqrt{2H(x,p)}>0 \qquad\text{ and }\qquad\bar{s}=\sup_{V_{x_0}^\eta} \sqrt{2H(x,p)}<\infty.
$$
Then there exists $C>0$ such that 
$$
\int_{\Gamma_\varepsilon} \e^{-\frac{h_{x,y}(z)-d(x,y)^2/4}{t}} \diff \mu( z)
\leq
\int_{\Gamma_\varepsilon} \e^{-\frac{ \left(d(x,z)-d(x,y)/2\right)^2}{t}} \diff \mu( z)
\leq 
C\bar{\nu}  \int_{\ubar{s}}^{\bar{s}}\e^{-\frac{1}{t}\left(s-\frac{d(x,y)}{2}\right)^2}\diff s .
$$
Again, as a classical application of Laplace integrals asymptotics, as $t$ goes to $0$,
$$
\int_{\ubar{s}}^{\bar{s}} \e^{-\frac{1}{t}\left(s-\frac{d(x,y)}{2}\right)^2}\diff s
\sim
\frac{{(2 \pi t)}^{1/2}}{2}.
$$
Hence the right-hand side of \eqref{E:a_priori_Laplace}, which concludes the proof.
\end{proof}

\section{Complete asymptotic expansions at the cut locus}
\label{S:Complete expansions}

{

As we aim to illustrate in this section, applying the Molchanov method allows to translate information on the jets of the hinged energy functional on the midpoint set to complete expansions of the heat kernel and its derivatives, while simultaneously sidestepping  heavier methods.
Here we are able to give proofs for complete expansions for some well known singular cases: conjugate minimizing curves of type $A$, and Morse-Bott conjugacy.

One critical point is that the complete expansions draw information from jets of the hinged energy functional. We show in Theorem~\ref{THM:LocalPrescription} that basically any smooth non-negative function can be realized as a hinged energy functional between two points of a Riemannian manifold. This  points towards the idea that full expansions should not always be accessible.

}

\subsection{A-type singularities}\label{Sect:A_n}
{
For some points in the cut locus, it is still possible to give a precise enough expansion of the heat kernel. In particular, we consider here the case where a pair of points $x,y\in M$ are connected by a unique geodesic that is conjugate.  

If we assume that $y$ is a singular value of $\sre_x$, the sub-Riemannian exponential at $x$, and, furthermore, that $\sre_x$ has a $A_n$ singularity, with $n>0$, at a preimage of $y$, then $n$ has to be odd for the normal extremal joining $x$ to $y$ to be minimizing (see Figure~\ref{F:A_p}). Indeed in that case (see, e.g., \cite{BBCN-IMRN}), the hinged energy functional has the normal form
\begin{equation}\label{Eqn:hDiag}
h_{x,y}(z_1,\dots,z_d)=\frac{d^2(x,y)}{4}+z_d^{n+1}+\sum_{i=1}^{d-1}z_i^{2}.
\end{equation}
This fact yields the following expansion.
}

\begin{center}
\begin{figure}[hbt]
\setlength{\unitlength}{.3\linewidth}
  \begin{picture}(1,1.6)%
    \put(0,0){\includegraphics[width=\unitlength,page=1]{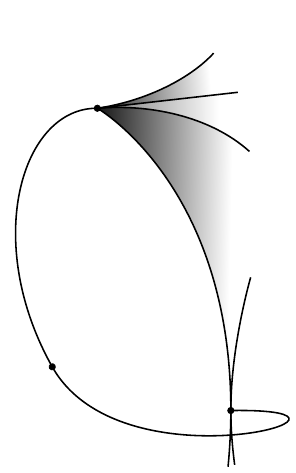}}%
    \put(0.2,0.38531532){$\Gamma$}%
    \put(0.3,1.28){$y$}
    \put(0.75,1.19){$\Cut(x)$}
    \put(0.78,0.83){$\mathrm{Conj}(x)$}
    \put(0.7,0.19){$x$}
  \end{picture}%
  \caption{A minimizing conjugate curve typically appears at the boundary of the cut, where the sub-Riemannian exponential degenerates. A particular example where such a situation occurs correspond to points on a non-degenerate caustic of 3D contact sub-Riemannian manifolds. Indeed for generic $x$ in such a manifold and  $y$ at the boundary of the cut (and at least sufficiently near $x$), the point $y$ belongs to a cuspidal fold of the conjugate locus corresponding to an $A_3$ singularity (see, e.g., \cite{agrachev_1996_exponential,Gauthier_1996_small_SR_balls}). The geodesic linking $x$ and $y$ is unique and $\Gamma$ reduced to a point. }
  \label{F:A_p}
\end{figure}
\end{center}

\begin{proposition}\label{Prop:A_nBasic}
Let $x$ and $y$ be two localizable points of a sub-Riemannian manifold such that the unique length minimizing curve joining $x$ to $y$ is strongly normal and a conjugate curve of type $A_{2p-1}$, $p\in\N$, $p\geq 1$.
Then if $l$ is any non-negative integer in the symmetric case and 0 otherwise, and $\alpha$ is any multi-index, there exists a sequence of real numbers $(c_k)_{k\in \N}$ and a sequence of functions $(\rho_k)_{k\in \N}$ such that for all $n\in \N$,
\begin{equation}\label{Eqn:A_nBasic}
\partial_t^l
Z_y^{\alpha}
p_t(x,y)
=
\frac{\e^{-\frac{d(x,y)^2}{4t}}}{t^{|\alpha|+2 l+\frac{d+1}{2}-\frac{1}{2p}} }
\left(
\sum_{k=0}^{n}c_k t^{k/p}
+
t^{\frac{n+1}{p}}\rho_{n+1}(t)
\right),
\end{equation}
and there exists $t_0>0$ such that 
$$
\sup_{(0,t_0)}|\rho_{n+1}(t)|<\infty.
$$

\end{proposition}

\begin{proof}
The pair $(x,y)$ is a compact subset of $M ^2\setminus\D$, hence by Corollary~\ref{C:MolchanovLaplace}, we have
$$
\partial_t^l
Z_y^{\alpha}
p_t(x,y)
=
\left(\frac{2}{t}\right)^{|\alpha|+2 l+d}
\int_{\Gamma_\varepsilon} 
\e^{-\frac{h_{x,y}(z)}{t}}
\Sigma_{t/2}^{0,0}(x,z)
\Sigma_{t/2}^{l,\alpha}(x,z)
\diff\mu( z)
+
O\left(\e^{-\frac{d(x,y)^2+\varepsilon^2/2}{4t}}\right).
$$
Under the assumptions of the theorem, the normal geodesic joining $x$ and $y$ is such that the hinged energy near the midpoint of the geodesic $z_0$ can be expressed in coordinates  by $h_{x,y}(z)=\frac{d(x,y)}{4}+\Phi\circ \xi(z)$, with $\Phi:\R^d\to \R$ the normal form
$
\Phi(u)=u_d^{2p}+\sum_{i=1}^{d-1}u_i^{2}
$
and $\xi:\Gamma_\varepsilon\to \R^d$ a diffeomorphism centered at $z_0$ (assuming $\varepsilon$ is small enough).

For a smooth function $\phi:\R^d\to \R$, we need to express asymptotics of 
$\int_{\R^d}\e^{-\Phi(u)/t}\phi(u)\diff u$ as $t$ goes to $0$.
We can follow \cite[Equation~(4.61)]{EAndK} and compute the Cauchy product of expansions of singular Laplace integrals of different degrees ($2p$ in the $z_d$ direction and $2$ in the $(u_1,\dots,u_{d-1})$ directions). We then have
\begin{multline*}
\int_{\R^d}\e^{-\Phi(u)/t}\phi(u)\diff u
\simeq\\
\sum_{m=0}^\infty
	\sum_{k=0}^m
	\left(
		\frac{\Gamma\left(\frac{2k+1}{2p}\right)}{p (2k)!}t^{\frac{2k+1}{2p}}\partial_{u_d}^{2k}
	\right)
	\left(
		\frac{\pi^{\frac{d-1}{2}}}{2^{2(m-k)}}t^{\frac{d-1}{2}+m-k}	
		\sum_{\substack{\omega\in \N^{d-1}\\|\omega|=m-k }}\frac{\partial_{(u_1,\dots,u_{d-1})}^{2\omega}}{\omega !}
	\right)\phi(0).
\end{multline*}
As in the proof of Lemma~\ref{L:Molchanov2BenArous}, 
for $f:\R\to\R$, we borrow from \cite{EAndK} the notation $
f(t)\simeq \sum_{m=0}^\infty t^m f_m
$
when
$f(t)=\sum_{m=0}^{N} t^m f_m+O(t^{N+1})$ for all $N>0$ integer. 
We rearrange the terms so that the index $k$ completes the multi-index $\omega\in \N^{d-1}$ (and such that $|\omega|=m-k$), into $\omega'\in \N^{d}$, a multi-index such that $|\omega'|=m$ with $k=\omega'_{d}$. Hence (dropping $\omega'$ in favor of $\omega$)
\begin{equation}\label{E:expansionLaplaceAp}
\int_{\R^d}\e^{-\Phi(u)/t}\phi(u)\diff u
\simeq
\sum_{m=0}^\infty
	\sum_{\substack{\omega\in \N^{d}\\|\omega|=m}}
	\left(
		\frac{
			\pi^{\frac{d-1}{2}} \Gamma\left(\frac{2\omega_d+1}{2p}\right)\omega_d! 
			}{
			2^{2(m-\omega_d)} p (2\omega_d)!
			}
	\right)
		t^{\frac{2\omega_d+1}{2p}+\frac{d-1}{2}+m-\omega_d}
		\frac{\partial_{u}^{2\omega}\phi(0)}{\omega !}.
\end{equation}

Now to obtain Equation~\eqref{Eqn:A_nBasic}, we pick $N=\lfloor n/p \rfloor$.
As a consequence of Theorem~\ref{T:UnifBenArous}, by multiplying together Ben Arous expansions, there exist a sequence $(d_k)$ of smooth functions over $\Gamma_\varepsilon$, and $t_0>0$, such that
$$
\sup_{(0,t_0)}\sup_{z\in\Gamma_\varepsilon}
\frac{1}{t^{N+1}}
\left|
\Sigma_{t/2}^{0,0}(x,z)
\Sigma_{t/2}^{l,\alpha}(z,y)
-
\sum_{k=0}^N d_k(z)t^k
\right|<\infty .
$$
Denoting by $\nu$ the smooth density of $\xi_*\mu$ with respect to the Lebesgue measure on $\R^d$, we apply expansion~\eqref{E:expansionLaplaceAp} with $\varphi(u)=d_k(\xi^{-1}(u))\nu(u)$ to get
$$
t^{|\alpha|+2l+\frac{d+1}{2}-\frac{1}{2p}}\e^{\frac{d(x,y)^2}{4t}}
p_t(x,y)
=
\sum_{k=0}^{N}
\sum_{m=0}^{n-kp}
\sum_{\substack{\omega\in \N^{d}\\|\omega|=m}}
	\sigma_{k,m,\omega}
		t^{k+\frac{\omega_d}{p}+m-\omega_d}
+
t^{\frac{n+1}{p}}\tilde\rho_{n+1}(t)
$$
where
$$
\sigma_{k,m,\omega}=
\left(
		\frac{
			\pi^{\frac{d-1}{2}} \Gamma\left(\frac{2\omega_d+1}{2p}\right)\omega_d! 
			}{
			2^{2(m-\omega_d)+k-d} p (2\omega_d)! \omega !
			}
	\right)
		{\partial_{u}^{2\omega}}_{|u=0}\left[d_k(\xi^{-1}(u))\nu(u)\right]
$$
and 
$\sup_{(0,t_0)}\big|\tilde \rho_{n+1}(t,x,y)\big|<\infty$. By rearranging the terms by increasing powers, and pushing into the remainder terms such that $(p-1)\omega_d< ((k+m)p-n)$ (which implies that $(k+m-\omega_d)p +\omega_d>n$), we obtain the stated result.
\end{proof}

\begin{corollary}
Let $x$ and $y$ be two localizable points of a sub-Riemannian manifold such that each length-minimizing curve joining $x$ to $y$ is strongly normal and a conjugate curve of type $A_{2p-1}$, $p\in\N$, $p\geq 1$ (where $p$ may be different for each curve, and $p=1$ corresponds to a non-conjugate geodesic). Then there are finitely many length-minimizing curves joining $x$ and $y$, and if $l$ is any non-negative integer in the symmetric case and 0 otherwise, and $\alpha$ is any multi-index, $\partial_t^l Z_y^{\alpha} p_t(x,y)$ has an expansion given by a finite sum of sequences, one for each length-minimizing curve, of the type on the right-hand side of \eqref{Eqn:A_nBasic}.
\end{corollary}

\begin{proof}
As illustrated by the normal form \eqref{Eqn:hDiag}, for each point in $z\in \Gamma$, there exists a small neighborhood such that, except at $z$, $h_{x,y}$ is strictly larger than its minimum $d(x,y)^2/4=h_{x,y}(z)$. This illustrates that each element of $\Gamma$ must be isolated. Furthermore, since $\Gamma$ is also compact, this proves that $\Gamma$ is a finite collection of points in $M$, and that there are finitely many length-minimizing curves joining $x$ and $y$. Then if we denote $\Gamma=\cup_{i=1}^m\{z_i\}$, $z_i\neq z_j$ if $i\neq j$, there exists $\varepsilon>0$ small enough so that $\Gamma_{\varepsilon}$ is the union of $m$ disconnected compact sets that we denote $\Gamma_{\varepsilon}^i$, so that $z_i\in\Gamma_{\varepsilon}^i$. By Lemma~\ref{C:MolchanovLaplace}, up to further reducing $\varepsilon$,
$$
\partial_t^l
Z_y^{\alpha}
p_t(x,y)
=
\left(\frac{2}{t}\right)^{|\alpha|+2 l+d}
\sum_{i=1}^m
\int_{\Gamma_\varepsilon^i} 
\e^{-\frac{h_{x,y}(z)}{t}}
\Sigma_{t/2}^{0,0}(x,z)
\Sigma_{t/2}^{l,\alpha}(z,y)
\diff\mu( z)
+
O\left(\e^{-\frac{d(x,y)^2+\varepsilon^2/2}{4t}}\right).
$$
Hence the statement by applying the proof of Proposition~\ref{Prop:A_nBasic} to each integral over $\Gamma_\varepsilon^i$.
\end{proof}

\subsection{Morse-Bott case}\label{Sect:MB}
{  An interesting example of a point in the cut locus is the Morse-Bott case (so called because it corresponds to $h_{x,y}$ being a Morse-Bott function), where the set of geodesics between two points $x$ and $y$ is a continuous family, such that the midpoint set becomes a submanifold in $M$ (of constant dimension) and the Hessian is non-degenerate in the normal directions.
We follow \cite{BBN-JDG,BBCN-IMRN} for the definition of such a pair of points. All the elements necessary to give a full expansion of the heat kernel in that situation are present in \cite{BBN-JDG} but that particular goal was not pursued. Here  we apply Molchanov's method to obtain the full expansion. It should be noted that original methods for obtaining full expansions have been developed in \cite{Inahama} and \cite{Ludewig2,Ludewig1}, where this particular example is treated. One observation that can be made from our technique is that, although these new approaches offer promising steps towards the construction of expansions of heat kernels in various situations, it doesn't appear necessary to introduce an original theory to treat this particular case in our sub-Riemannian situation.

Denoting} $\Lambda_x=\{p\in T^*_xM \mid H(p,x)=1/2\}$ and 
$$
L=\{p\in \Lambda_x\mid \sre_x(p,d(x,y))=y\},
$$
we assume that for a specific pair $(x,y)\in M^2$:
\begin{itemize}
\item the pair $x$ and $y$ are localizable,
\item all minimizers from $x$ to $y$ are strongly normal, hence given by the exponential map applies to elements of $L$,
\item $L$ is a dimension $r$ submanifold of $\Lambda_x$,
\item for every $p\in L$, we have $\dim \ker D_{(p,d(x,y))}\sre_x=r$.
\end{itemize}
(See, for instance, Figure~\ref{F:Pansu_sphere} for an example of a such situation. See also \cite{BBN-BiHeis} for another example where this type of cut points play an essential role.)
Under these assumptions, $\Gamma$ is a compact submanifold of $M$ and it is proved in \cite{BBN-JDG} that  there exists a collection $(U_i)_{1\leq i\leq N}$ of open sets such that for $\varepsilon$ small enough,
$\Gamma_\varepsilon\subset \cup_{i=1}^N U_i$, and on each $U_i$, there exists a set of coordinates $\xi:M\to \R^d$ such that for all $z\in U_i$
$$
\Gamma\cap U_i=\xi^{-1}(\{u_{r+1}=\dotsc=u_d=0\})
$$
and
\begin{equation}\label{E:NF_MorseBott}
h_{x,y}\circ \xi^{-1} (u)=\frac{d(x,y)^2}{4}+\sum_{i=r+1}^du_i^2.
\end{equation}
Furthermore, there exists a partition of unity $(\varphi_i)_{1\leq i\leq N}$, such that ${\varphi_i}_{|U_i^c}=0$ and for all $z\in \Gamma_\varepsilon$,
$$
\sum_{i=1}^N\varphi_i(z)=1,
$$
and on each $U_i$,
$$
\varphi_i\circ \xi^{-1}(u_1,\dots,u_d)=\varphi_i\circ \xi^{-1}(u_1,\dots,u_r,0,\dots,0 ) .
$$

\begin{figure}
\includegraphics[width=.3\linewidth,clip,trim=2.5cm 3cm 2.5cm 3cm]{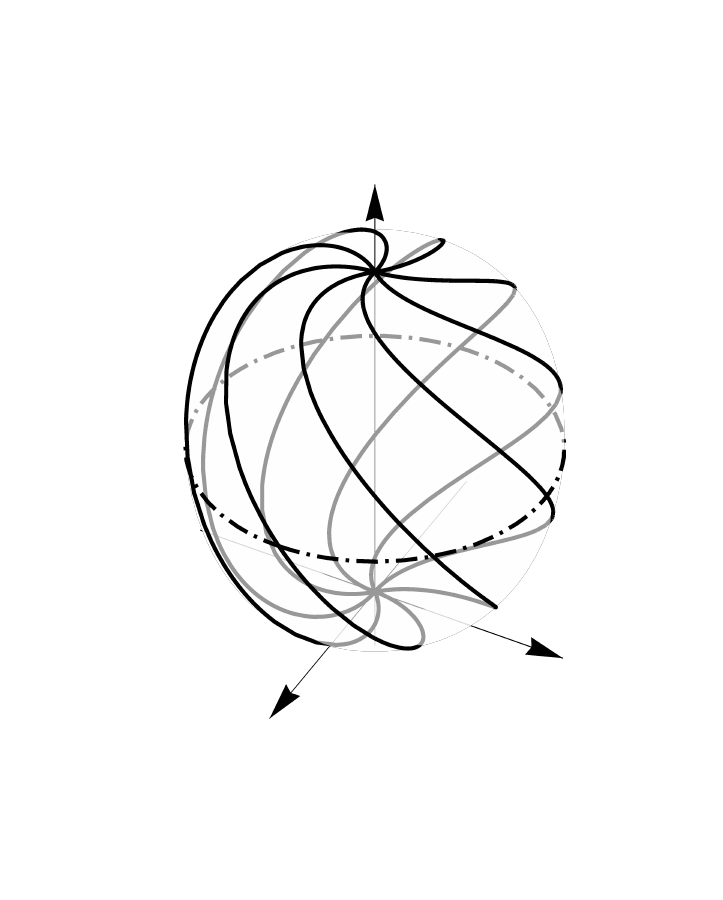}

\caption{ Any point of the cut locus in the Heisenberg group satisfies the definition of the Morse-Bott case. The set of all geodesics emanating from $(0,0,0)$ and becoming cut at some point $(0,0,h)$, $h>0$ form a sphere-like shape with rotational symmetry around the $z$-axis. When considering the pair $(0,0,0),(0,0,h)$, the midpoint set is the equator of the sphere, a circle sitting at altitude $h/2$. Other examples include the Riemannian spheres of dimension at least $2$.}
\label{F:Pansu_sphere}
\end{figure}

In the described situation, we have the following.

\begin{proposition}\label{P:morse-bott}
For any $l$ which is a non-negative integer in the symmetric case and 0 otherwise, and any multi-index $\alpha$, there exists a sequence of real numbers $(c_k)_{k\in \N}$ and a sequence of functions $(\rho_k)_{k\in \N}$ such that for all $n\in \N$,
$$
\partial_t^l
Z_y^{\alpha}
p_t(x,y)
=
\frac{\e^{-\frac{d(x,y)^2}{4t}}}{t^{|\alpha|+2 l+\frac{d+r}{2}} }
\left(
\sum_{k=0}^{n}c_k t^{k}
+
t^{ n+1 }\rho_{n+1}(t)
\right),
$$
and there exists $t_0>0$ such that 
$$
\sup_{(0,t_0)}|\rho_{n+1}(t)|<\infty.
$$

\end{proposition}

\begin{proof}
The proof follows the same classical reasoning of multiplication of series, along with a Fubini theorem argument. We start by applying Corollary~\ref{C:MolchanovLaplace} to get for $\varepsilon>0$ small enough
$$
\partial_t^l
Z_y^{\alpha}
p_t(x,y)
=
\left(\frac{2}{t}\right)^{|\alpha|+2 l+d}
\int_{\Gamma_\varepsilon} 
\e^{-\frac{h_{x,y}(z)}{t}}
\Sigma_{t/2}^{0,0}(x,z)
\Sigma_{t/2}^{l,\alpha}(z,y)
\diff\mu( z)
+
O\left(\e^{-\frac{d(x,y)^2+\varepsilon^2/2}{4t}}\right).
$$
Then, applying our assumptions,
$$
\begin{aligned}
\int_{\Gamma_\varepsilon} \e^{-\frac{h_{x,y}(z)}{t}}
\Sigma_{t/2}^{0,0}(x,z)
\Sigma_{t/2}^{l,\alpha}(z,y)\diff\mu( z)
&=
\int_{\Gamma_\varepsilon}\left(\sum_{i=1}^N \varphi_i(z) \right)\e^{-\frac{h_{x,y}(z)}{t}}
\Sigma_{t/2}^{0,0}(x,z)
\Sigma_{t/2}^{l,\alpha}(z,y) \diff\mu( z)
\\
&=
\sum_{i=1}^N\int_{\Gamma_\varepsilon \cap U_i} \varphi_i(z) \e^{-\frac{h_{x,y}(z)}{t}}
\Sigma_{t/2}^{0,0}(x,z)
\Sigma_{t/2}^{l,\alpha}(z,y) \diff\mu( z) .
\end{aligned}
$$
Now the hypothesis on the shape of $\Gamma$, implies that, up to a rescaling of $\xi$ in the direction transverse to $\Gamma$,
$$
\xi(\Gamma_\varepsilon\cap U_i)\supset\{v+w \mid v\in \xi(\Gamma\cap U_i), w\in \{0_{\R^r}\}\times (-\eps,\eps)^{d-r}\}.
$$
As described earlier, $\xi(\Gamma\cap U_i)\subset\{u_{r+1}=\dotsc=u_d=0\}$. We denote by $\tilde\Gamma_i$ the projection of $\xi(\Gamma\cap U_i)$ onto its first $r$ coordinates, so that $\{v+w \mid v\in \xi(\Gamma\cap U_i), w\in \{0_{\R^r}\}\times (-\eps,\eps)^{d-r}\}=\tilde\Gamma_i\times (-\eps,\eps)^{d-r}$.
By assumption, 
the normal form \eqref{E:NF_MorseBott} implies that for any smooth map $\phi:\R^d\to\R$,
$$
\int_{\xi(\Gamma_\varepsilon\cap U_i)} 
\e^{-\frac{h_{x,y}\circ\xi(u)}{t}}\phi(u)
\diff u
=
\int_{\tilde \Gamma_i\times (-\eps,\eps)^{d-r}} 
\e^{-\frac{h_{x,y}\circ\xi(u)}{t}}\phi(u)
\diff u
+
O\left(\e^{-\frac{d(x,y)^2+\varepsilon^2}{4t}}\right).
$$
Furthermore, the integral can be distributed by Fubini's theorem ($\phi$ is smooth) as
$$
\int_{\tilde \Gamma_i\times (-\eps,\eps)^{d-r}} 
\e^{-\frac{h_{x,y}\circ\xi(u)}{t}}\phi(u)
\diff u
=
\int_{\tilde \Gamma_i}
\left(
\int_{(-\varepsilon,\varepsilon)^{d-r}} \e^{-\frac{h_{x,y}\circ\xi(u)}{t}}\phi(u)
 \diff u_{r+1}\cdots\diff u_d
\right)
\diff u_1\cdots\diff u_r .
$$
As before, we follow \cite{EAndK} to get the expansion (smoothness of $\phi$ allows to then distribute the outer integral)
\begin{multline}\label{E:expansionLaplaceMB}
\int_{\R^{d-r}}
\e^{\frac{1}{t}\sum_{i=r+1}^du_i^2}
\phi(u)
\diff u_{r+1}\cdots\diff u_{d}
\simeq\\
\sum_{j=0}^\infty
	t^{\frac{d-r}{2}+j}	
	\left(
		\frac{\pi^{\frac{d-r}{2}}}{2^{2j}}
		\sum_{\substack{\omega\in \N^{d-r}\\|\omega|=j }}\frac{\partial_{(u_{r+1},\dots,u_{d})}^{2\omega}\phi(u_1,\dots,u_r,0,\dots, 0)}{\omega !}
	\right).
\end{multline}

As a consequence of Theorem~\ref{T:UnifBenArous}, by multiplying together Ben Arous expansions, there exist a sequence $(d_k)$ of smooth functions over $\Gamma_\varepsilon$, and $t_0>0$, such that
$$
\sup_{(0,t_0)}\sup_{z\in\Gamma_\varepsilon}
\frac{1}{t^{N+1}}
\left|
\Sigma_{t/2}^{0,0}(x,z)
\Sigma_{t/2}^{l,\alpha}(z,y)
-
\sum_{k=0}^N d_k(z)t^k
\right|<\infty.
$$
Denoting by $\nu$ the smooth density of $\xi_*\mu$ with respect to the Lebesgue measure on $\R^d$, we apply expansion~\eqref{E:expansionLaplaceMB} with $\varphi(u)=d_k(\xi^{-1}(u))\nu(u)$ to get
$$
t^{|\alpha|+2l+\frac{d+r}{2}}\e^{\frac{d(x,y)^2}{4t}}
p_t(x,y)
=
\sum_{k=0}^{n}
\sum_{j=0}^{n-k}
\sum_{\substack{\omega\in \N^{d-r}\\|\omega|=j}}
	\sigma_{k,j,\omega}
		t^{k+j}
+
t^{n+1}\tilde\rho_{n+1}(t)
$$
where
$$
\sigma_{k,j,\omega}=
\frac{\pi^{\frac{d-r}{2}}}{2^{2j}\omega!}
\int_{\tilde \Gamma_i}
\varphi_i(u_1,\dots,u_r,0,\dots 0)
		{\partial_{(u_{r+1},\dots,u_{d})}^{2\omega}}_{|u=0}\left[d_k(\xi^{-1}(u))\nu(u)\right]
\diff u_1\cdots\diff u_r
$$
and 
$\sup_{(0,t_0)}\big|\tilde \rho_{n+1}(t,x,y)\big|<\infty$. By rearranging the terms by increasing powers, we obtain the stated result.
\end{proof}

\subsection{Prescribed singularities}\label{Sect:Prescribed}
{
We wish to discuss situations where an explicit expansion of the Laplace integral for the small-time asymptotics of the heat kernel appears not be known. The first step is to show that essentially any possible phase function $h$ for a Laplace integral can be realized as the hinged energy functional between two points of a manifold. We first restrict our attention to Riemannian metrics. As already noted, the Molchanov method doesn't distinguish between Riemannian and (properly) sub-Riemannian metrics, and it is simpler to give explicit constructions of Riemannian metrics.

Since the pair of points we consider will be fixed, in this section we use $(q_1,q_2)$ rather than $(x,y)$ to free the notation.
}

\begin{theorem}\label{THM:LocalPrescription}
Let $M$ be a smooth manifold of dimension $d$ (with $d\geq 2$) , $q_1,q_2$ in $M$ such that $q_1\neq q_2$, and $a$ and $\sigma$ be positive real numbers. Let $h$ be a smooth, real-valued function in a neighborhood of $\overline{B^{d-1}(0,a)}\subset \R^{d-1}$ such that $h(0,\ldots,0)=0$, non-negative on $\overline{B^{d-1}(0,a)}$, and positive on $\partial B^{d-1}(0,a)$. Then there exists a (complete) Riemannian metric $g$ on $M$ such that $\Gamma=\Gamma(x,y)$ is contained in a coordinate patch
\[
(u_1,\ldots,u_d): U \rightarrow B^{d-1}(0,a) \times(-\delta,\delta)
\]
such that
\[
h_{q_1,q_2}|_N =  \frac{\sigma^2}{4} + h(u_1,\ldots,u_{n-1})+u_d^2 
\]
for some neighborhood $N$ of $\Gamma$ (thus $\Gamma$ is given by the zero level set of $h$ in the hyperplane $\{u_d=0\}$, and $d(q_1,q_2)=\sigma$). Further, we have the heat kernel representation
\begin{equation}\label{Eqn:PrescribedExp}
 p_t\lp q_1,q_2\rp = \frac{1}{t^d} e^{-\frac{d^2(q_1,q_2)}{4t}}
 \int_{B^{d-1}_{0}(a) \times(-\varepsilon,\varepsilon)} \Phi(t,u)\,\,  e^{-\frac{h(u_1,\ldots,u_{d-1})+u_d^2}{t}} \,\diff u_1\cdots \diff u_d +O\lp e^{-\frac{d^2(q_1,q_2)+c}{4t}} \rp.
\end{equation}
for some positive $\varepsilon$ and a  smooth prefactor function $\Phi$ over $\R^+\times B^{d-1}(0,a) \times(-\varepsilon,\varepsilon) $, smoothly extendable and positive at $t=0$.
\end{theorem}

By rescaling, there is no loss of generality in assuming that the distance between $q_1$ and $q_2$ is prescribed to be 2, which is the same as prescribing $\sigma=1$.

Let $(z_1,\ldots,z_d)$ be the standard Euclidean coordinates on $\R^n$, and let $g_E$ be the Euclidean metric. We will identify $q_1$ with $(0,\ldots,0,1)$, $q_2$ with $(0,\ldots,0,-1)$, and $B_0^{d-1}(a)$ with the corresponding subset of the hyperplane $\{z_d=0\}$. (In particular, this will end up being compatible with the notation used in the theorem). We will use $V^+$ (respectively $V^-$) to denote a neighborhood of $\overline{B^{d-1}_{0}(a)}$ large enough to contain $q_1$ (respectively $q_2$), with further properties of $V^-$ and $V^+$ to be specified later. If $V=V^-\cup V^+$, the main work of the proof is to construct a metric on $V$ which gives a distance function with the desired properties.

\begin{lemma}\label{L:metric_construction}
Let $\xi$ be a smooth non-negative function on a neighborhood of $\overline{B^{d-1}(0,a)}$, $a>0$, everywhere less than $1/8$, with all of its derivatives bounded, and $\xi$ bounded from below by a positive constant outside of $B^{d-1}(0,a)$.
Under assumptions of Theorem~\ref{THM:LocalPrescription}, 
there exist  $V^+$ a neighborhood of $\overline{B^{d-1}(0,a)}\times\{0\}\cup\{q_1\}$ and
 a (smooth) metric on
 $V^+$ 
 such that the graph of $\xi $ in $B^{d-1}(0,a)\times [0,1/8]$ is a subset of the sphere of radius 1 around $q_1$, with none of the minimal geodesics from $q_1$ to this graph conjugate, and such that the metric agrees with $g_{E}$ on a neighborhood of 
$$
\{z\in B^{d-1}(0,a)\times [0,1/8]\mid  0\leq z_d\leq \xi(z_1,\dots ,z_{d-1})\}.
$$

\end{lemma}

\begin{proof}

Let 
$$
G=
\{z\in B^{d-1}(0,a+1)\times \R \mid    z_d = \xi(z_1,\dots ,z_{d-1})\}
$$
denote the portion of the graph of $\xi$ on a neighborhood of $B^{d-1}(0,a)\times [0,1/8]$.
The graph is a smooth hypersurface, and as a consequence of the bounded derivatives property,
 there is a tubular neighborhood $U$ (of diameter $\eta>0$) of $G$ on which its normal lines do not develop focal (or conjugate) points, and any (smooth) coordinates on $G$ extend to smooth coordinates on $U$ by including the signed distance to $G$ as the first coordinate.  In what follows, we use the phrases ``above $G$'' and ``below $G$'' in the natural way to describe the regions on which $z_d$ is larger or smaller than $\xi$, respectively, and similarly for other sets in place of $G$, when it makes sense.

Let $(\theta_1,\ldots, \theta_{d-1})$ be coordinates on the (open) lower hemisphere of the unit tangent sphere at $q_1$. They are assumed to be centered at the south pole.
The lower hemisphere of the unit tangent sphere maps diffeomorphically to the hyperplane $\{z_d=0\}$
by following the Euclidean rays from $q_1$.
Furthermore, lifting $(z_1,\ldots,z_{d-1})$  to the graph  gives coordinates on $G$. Thus, by composition, $(\theta_1,\ldots, \theta_{d-1})$ gives coordinates on $G$, centered at the origin.
Next, let $\rho$ be the signed distance from $G$, with the sign chosen so that $\rho$ is positive below $G$; then $(\rho,\theta_1,\ldots,\theta_{d-1})$ gives coordinates on $U$. For future use, let $\Theta$ denote the open subset of $\S^{d-1}$ for which $(\theta_1,\ldots,\theta_{d-1}$) gives coordinates on $G\cap \lp   B^{d-1}(0,a)\times (-1/4,1/4) \rp$. Now if we write the Euclidean metric on $U$ in these coordinates, it is given by the matrix
\[
\begin{bmatrix}
1 & 0 & \cdots& 0 \\
0& \ip{\partial_{\theta_1}}{\partial_{\theta_1}}_{g_E} &\cdots& \ip{\partial_{\theta_{d-1}}}{\partial_{\theta_{1}}}_{g_E} \\
\vdots & \vdots & \ddots& \vdots \\
0& \ip{\partial_{\theta_{1}}}{\partial_{\theta_{d-1}}}_{g_E} &\cdots& \ip{\partial_{\theta_{d-1}}}{\partial_{\theta_{d-1}}}_{g_E} 
\end{bmatrix}
=
\begin{bmatrix}
1  & 0 \\
0 & \left[ \ip{\partial_{\theta_i}}{\partial_{\theta_j}}_{g_E} \right]_{1\leq i,j\leq d-1} 
\end{bmatrix} ,
\]
where the last expression is understood in terms of the $1\times1$ and $(d-1)\times(d-1)$ diagonal block decomposition of this matrix. Note that the $\ip{\partial_{\theta_i}}{\partial_{\theta_j}}_{g_E} $ are smooth, positive functions on $U$ for all $1\leq i,j\leq d-1$.

We can now describe the metric $g$ on $V^+$ that we're looking for. We will give $g$ on part of $V^+$ including $U$ and the region above $U$ in polar coordinates around $q_1$; that is, we will specify the $\ip{\partial_{\theta_i}}{\partial_{\theta_j}}_{g_E}$, which determines the metric (since all inner products with $\partial_r$ are determined by the condition of being polar coordinates). In a ball around $q_1$, of Euclidean radius $1/8$, let $g$ agree with the Euclidean metric, so that the coordinate singularity at $q_1$, which is $r=0$, is the usual one from polar coordinates on $\R^n$ and the metric is in fact smooth there. For $(\theta_1,\ldots, \theta_{d-1})\in \Theta$ and $r\in (1-\eta,1+\eta)$, we let 
\[
\ip{\partial_{\theta_i}}{\partial_{\theta_j}}_g (r,\theta_1,\ldots,\theta_{d-1}) = \ip{\partial_{\theta_i}}{\partial_{\theta_j}}_{g_E} (\rho=r-1,\theta_1,\ldots,\theta_{d-1}) .
\]
For $r\in (1/8,1-\eta)$ and $(\theta_1,\ldots, \theta_{d-1})\in \Theta$, we let $\ip{\partial_{\theta_i}}{\partial_{\theta_j}}_g$ be some smooth, positive interpolation between the values we just fixed. This gives a metric on the part of $V^+$ including $U$ and the region above $U$, but these coordinates may not extend to a neighborhood of $B^{d-1}(0,a)$ (because $U$ might lie above $\{z_d=0\}$ away from the zeroes of $u$). Nonetheless, if we put the Euclidean metric on the region below $U$, then the metric extends, because the transition map from the polar coordinates $(r,\theta_1,\ldots,\theta_{d-1})$ to the Cartesian coordinates $(z_1,\ldots,z_d)$ is an isometry by construction.

In particular, we have given a metric $g$ on a neighborhood of the origin in $T_{q_1}M$ (which includes $(0,1+\eta)\times\Theta)$) along with an isometry from a subset of that neighborhood to a neighborhood of $B^{d-1}(0,a)\subset\R^n$ (with the Euclidean metric) that includes $U$, such that $G\cap \lp   B^{d-1}(0,a)\times (-1/4,1/4) \rp$ is an open subset of the $g$-sphere of radius 1 around $q_1$. Moreover, we did this by deforming the metric in between a small ball around $q_1$ and $U$ so that Euclidean rays from $q_1$ ``matched up'' to the normal lines to $G$ after passing through this ``in between'' region.
\end{proof}

\begin{figure}
\begin{center}
\setlength{\unitlength}{8cm}
  \begin{picture}(1,0.59383466)%
    \put(0,0){\includegraphics[width=\unitlength,page=1]{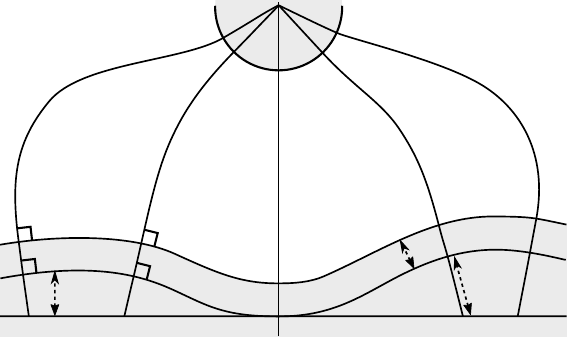}}%
    \put(0.805,0.007){\makebox(0,0)[lb]{\smash{$x$}}}%
    \put(0.82,0.09){\makebox(0,0)[lb]{\smash{$\zeta(x)$}}}%
    \put(0.73,0.145){\makebox(0,0)[lb]{\smash{$\eta$}}}%
    \put(0.52,0.58){\makebox(0,0)[lb]{\smash{$q_1$}}}%
    \put(0.11,0.065){\makebox(0,0)[lb]{\smash{$\xi(y)$}}}%
    \put(0.085,0.007){\makebox(0,0)[lb]{\smash{$y$}}}%
  \end{picture}%
\caption{Schematic representation of the constructed geodesic front radiating from $q_1$ at time $1$. Sections of the picture in grey correspond to regions where $g$ agrees with $g_E$.}
\label{F:front_time_1}
\end{center}
\end{figure}

For the given hinged energy functional in Theorem~\ref{THM:LocalPrescription}, we now extrapolate a possible front at time $1$ emanating from $q_1$ in accordance with the construction of Lemma~\ref{L:metric_construction}.

Let $\tilde{\Gamma}=\{h=0\}\subset B^{d-1}(0,a)$. For all $R\geq 0$, we denote by $N_0(R)\subset \R^{d-1}$ the set
$$
N_0(R)=\tilde{\Gamma}+B^{d-1}(0,R).
$$
Let $\zeta$ be the non-negative real valued smooth function in a bounded neighborhood $\mathcal{V}$ of $\overline{B^{d-1}(0,a)}\subset \R^{d-1}$ such that
$$
\zeta(x)=\sqrt{1+h(x)}-1.
$$
The function $\zeta$ is introduced to allow the construction of $h_{q_1,q_2}$ on the plane $z_d=0$.
\begin{lemma}\label{L:distance_to_q_1}
There exist neighborhoods $N_0\subset  B^{d-1}(0,a)$ and $\tilde{N}_0\subset  B^{d-1}(0,a)$ of $\tilde{\Gamma}$ and a function $\xi:\R^{d-1}\to \R$ 
that satisfy the following geometric property. In Cartesian coordinates $(z_1,\dots,z_d)$,
for all $y\in \tilde{N}_0$, the normal line to the graph of $\xi$ at $(y,\xi(y))$ crosses the plane $\{z_d=0\}$ at a point $(x,0)$, $x\in N_0$,  such that the Euclidean distance between $(y,\xi(y))$ and $(x,0)$ is $\zeta(x)$.

Furthermore, $\xi$ is a smooth non-negative function, everywhere less than $1/8$,
with all of its derivatives bounded, and bounded from below by a positive constant outside of $B^{d-1}(0,a)$
\end{lemma}

\begin{proof}

Let $\psi:\mathcal{V}\to \R^{d-1}$ be defined by $\psi(x)=x-\zeta(x)\nabla \zeta(x)$. 
 For all $x\in\mathcal{V}$, 
$$
D\psi(x)=\mathrm{id}_{\R^{d-1}}-\nabla \zeta(x)\cdot\nabla \zeta(x)^*- \zeta(x) \mathrm{Hess} \zeta (x).
$$
Since $h\geq 0$ and $h$ vanishes of $\tilde{\Gamma}$, $\nabla h$ and $\nabla \zeta$ also vanish on $\tilde{\Gamma}$. Thus on $\tilde{\Gamma}$, $D\psi(x)=\mathrm{id}_{\R^{d-1}}$. Since $\psi$ is smooth, $D\psi(x)$ is uniformly continuous on the compact $\overline{B^{d-1}(0,a)}$ and there exists $R_0$ small enough such that $D\psi$ is invertible on $N_0(R_0)\subset \overline{B^{d-1}(0,a)} $.

In addition, there exists $R_1\in(0,R_0)$ such that $\psi$ is a diffeomorphism from $N_0(R_1)$ onto its image. This is shown by contradiction: assume for any $R>0$ there exists a pair $(x,y)\in N_0(R)$ such that $x\neq y$ but $f(x)=f(y)$. Then let's define for each integer $n>0$ such a pair $(x_n,y_n)\in N_0(1/n)$. Since the sequences evolve in the compact set  $\overline{B^{d-1}(0,a)}$, they are convergent up to extraction. Furthermore, since $\tilde{\Gamma}$ is a compact set, the only possible attractors for $x_n$ and $y_n$ belong to  $\tilde{\Gamma}$. Hence there exists $\tilde{x},\tilde{y}\in \tilde{\Gamma}$ such that $x_n\to \tilde{x}$, $y_n\to \tilde{y}$. Since $\psi$ is continuous, $\psi(x_n)=\psi(y_n)$, we conclude that $\tilde{x}=\psi(\tilde{x})=\psi(\tilde{y})=\tilde{y}$. Hence $x_n-y_n\to 0$. This allows to conclude: indeed $\psi$ is a local diffeomorphism, hence the compact set  $\overline{N_0(R_0)}$ can be finitely covered with open balls on which the restriction of $\psi$ is a diffeomorphism onto its image. There must exist one such open ball containing both $x_n$ and $y_n$ for $n$ large enough (since $x_n,y_n\to \tilde{x}$). This imposes that $x_n=y_n$, which is a contradiction.

Let us pick $R_2\in(0,R_1)$ small enough so that we also have that  $\psi( N_0(R_2) )\subset  B^{d-1}(0,a)$, and $|\nabla \zeta|<1$, $\zeta<1/16$ on $N_0(R_2)$.

Let $r\in(0,R_2)$.  
We can set the map
$\xi_0:\psi\left(\overline{N_0(r)}\right)\to \R$ such that 
$$
\xi_0(y)=\zeta(\psi^{-1}(y))\sqrt{1-|\nabla \zeta(\psi^{-1}(y))|^2}.
$$

By definition, $\xi_0$ is the restriction of a $C^\infty$ map on $\psi\left( N_0(R_2)\right)$ to the closed set $\psi\left(\overline{N_0(r)}\right)$.  Hence $\xi_0$ automatically satisfies the Whitney compatibility condition from Whitney extension theorem and as a result can be extended to a smooth function with domain $\R^{d-1}$. Using smooth cut-off functions, we can ensure the existence of such an extension that has all of its derivatives bounded, and is bounded from below by a positive constant outside of $B^{d-1}(0,a)$, and since $\zeta<1/16$ on $N_0(R_2)$, is strictly smaller than $1/8$. We pick one such extension as  function $\xi$, $N_0=N_0(r)$ and $\tilde{N}_0=\psi(N_0(r))$.

Now let us check that the exhibited function $\xi$ satisfies the stated geometric property.
This statement is equivalent to 
$$
(\psi(x),\xi(\psi(x)))-(x,0)=\zeta(x) \frac{\left(-\nabla \xi(\psi(x)),1\right)}{\sqrt{1+|\nabla \xi(\psi(x))|^2}}, \qquad \forall x\in N_0.
$$
This is translated to the pair of equations
\begin{align}
\zeta(x)&=\xi(\psi(x)) \sqrt{1+|\nabla \xi(\psi(x))|^2},\label{E:cond_1}
\\
x&=\psi(x)+\frac{\zeta(x)  \nabla \xi(\psi(x))}{\sqrt{1+|\nabla \xi(\psi(x))|^2}}.\label{E:cond_2}
\end{align}

From the definition of $\xi$, we have for all $x\in N_0$
\begin{equation}\label{E:def_g}
\xi(x-\zeta(x)\nabla \zeta(x))=\zeta(x)\sqrt{1-|\nabla \zeta(x)|^2}.
\end{equation}
Differentiating the left-hand side,
$$
\nabla( \xi(x-\zeta(x)\nabla \zeta(x)))
=
\left(
\mathrm{id}-\nabla \zeta(x)\cdot\nabla \zeta(x)^*- \zeta(x) \mathrm{Hess} \zeta (x)
\right)
\cdot \nabla \xi(\psi(x)).
$$
Differentiating the right-hand side,
$$
\nabla\left(\zeta(x)\sqrt{1-|\nabla \zeta(x)|^2}\right)=\frac{\nabla \zeta(x)}{\sqrt{1-|\nabla \zeta(x)|^2}}(1-|\nabla \zeta(x)|^2)-\zeta(x)\frac{\mathrm{Hess}\zeta(x)\cdot \nabla \zeta(x)}{\sqrt{1-|\nabla \zeta(x)|^2}}.
$$
For any vector $v\in \R^n$, denoting $v^*$ its transpose, we have the identity
$
v|v|^2=v\cdot( v^*\cdot v)=(v\cdot v^*)\cdot v.
$
Hence 
$$
\nabla\left(\zeta(x)\sqrt{1-|\nabla \zeta(x)|^2}\right)=\left(
\mathrm{id}-\nabla \zeta(x)\cdot\nabla \zeta(x)^*- \zeta(x) \mathrm{Hess} \zeta (x)
\right)
\cdot\frac{\nabla \zeta(x)}{\sqrt{1-|\nabla \zeta(x)|^2}}.
$$

The radius $r$ has been chosen so that $\mathrm{id}-\nabla \zeta(x)\cdot\nabla \zeta(x)^*- \zeta(x) \mathrm{Hess} \zeta (x)$ is invertible.
Therefore \eqref{E:def_g} implies after differentiation
$$
\nabla \xi(\psi(x))=\frac{\nabla \zeta(x)}{\sqrt{1-|\nabla \zeta(x)|^2}}
$$
and, equivalently, 
$$
\nabla \zeta(x)=\frac{\nabla \xi(\psi(x))}{\sqrt{1+|\nabla \xi(\psi(x))|^2}}.
$$
Since $x=\psi(x)+\zeta(x)\nabla \zeta(x)$ by definition of $\psi$, we then have \eqref{E:cond_2}.

Likewise,
$$
\xi(\psi(x))=\zeta(x)\sqrt{1-|\nabla \zeta(x)|^2}
=
\zeta(x)\sqrt{1-\frac{|\nabla \xi(\psi(x))|^2}{1+|\nabla \xi(\psi(x))|^2}}
=
\frac{\zeta(x)}{\sqrt{1+|\nabla \xi(\psi(x))|^2}},
$$
which implies \eqref{E:cond_1}, and concludes the proof of the lemma.
\end{proof}

We are now able to give a proof of Theorem~\ref{THM:LocalPrescription}.

\begin{proof}[Proof of Theorem~\ref{THM:LocalPrescription}]
Let $\xi:\R^{d-1}\to \R$ be as in the statement of Lemma~\ref{L:distance_to_q_1}. By application of Lemma~\ref{L:metric_construction},
there exist  $V^+$ a neighborhood of $\overline{B^{d-1}(0,a)}\times\{0\}\cup\{q_1\}$ and
 a (smooth) metric on
 $V^+$ 
 such that the graph of $\xi $ in $B^{d-1}(0,a)\times [0,1/8]$ is a subset of the sphere of radius 1 around $q_1$. Furthermore, the metric agrees with $g_{E}$ on a neighborhood of 
$$
\{z\in B^{d-1}(0,a)\times [0,1/8]\mid  0\leq z_d\leq \xi(z_1,\dots ,z_{d-1})\}.
$$

Now reflect $V^+$ around $\{z_d=0\}$ to get $V^-$ (and note that $q_2$ is the image of $q_1$) and reflect $g$ to get a metric on $V^-$ such that $-G$ (which denotes the graph of $-\xi$) is an open subset of the sphere of radius 1 around $q_2$ in this metric. Note that this metric on $V^-$ is compatible with $g$ because, on a neighborhood of $B^{d-1}(0,a)$, they are both isometric to the Euclidean metric (and thus reflection induces a valid transition function), and thus we can extend $g$ to $V^-$ via reflection. 

In summary, we have built a metric $g$ on a neighborhood $V$ of $\overline{B^{d-1}(0,a)}\times\{0\}\cup\{q_1\}\cup\{q_2\}$ such that 
$$
d_g\left(q_1,(y,\xi(y))\right)=d_g(q_2,(y,-\xi(y)))=1 \qquad \forall y\in B^{d-1}(0,a),
$$
and $g$ coincides with $g_E$ on a neighborhood of 
$$
\{z\in B^{n}(0,a) \mid  |z_d|\leq \xi(z_1,\dots ,z_{d-1})\}.
$$

Let $N_0$ be as in the statement of Lemma~\ref{L:distance_to_q_1}. Then for any $x\in N_0$, we have 
$$
d_g(q_1,(x,0))=1+\zeta(x).
$$
Indeed this is a consequence of the geometric property in Lemma~\ref{L:distance_to_q_1}. Since $g$ is flat on a neighborhood of $\{z\in B^{d-1}(0,a) \mid  |z_d|\leq \xi(z_1,\dots ,z_{d-1})\}$,
$$
d_g(q_1,(x,0))=d_g(q_1,(y,\xi(y)))+d_g((y,\xi(y)),(x,0)).
$$
(See Figure~\ref{F:front_time_1}.)
Likewise
$$
d_g(q_2,(x,0))=1+\zeta(x).
$$
Hence for all $x\in N_0$,
$$
h_{q_1,q_2}(x,0)=\frac{1}{2} \left(d(q_1,(x,0))^2+d(q_2,(x,0))^2\right)=\left(1+\zeta(x)\right)^2=1+h(x).
$$

We can now work on extending $h_{q_1,q_2}$ to a neighborhood of $\Gamma$ in $\R^d$.
Let $\pi:\R^{d}\to \R^{d-1}$ such that $\pi(z_1,\dots,z_d)=(z_1,\dots,z_{d-1})$.
To prove the statement, we show that there exists $r>0,\varepsilon>0$ such that on $N_0(r)\times (-\varepsilon,\varepsilon)$, 
the map defined by
$$
v(z)=
\begin{cases}
+\sqrt{
	h_{q_1,q_2}(z)-h_{q_1,q_2}(\pi(z),0)
	}
	& \text{ if } z_{d}\geq 0, 
\\
-\sqrt{
	h_{q_1,q_2}(z)-h_{q_1,q_2}(\pi(z),0)
	}
	& \text{ if } z_d< 0 
\end{cases}
$$
is smooth and $z\mapsto v(x,z)$ is a diffeomorphism for each $z\in N_0(r)\times (-\varepsilon,\varepsilon)$. As a result, $\Phi:z\mapsto (\pi(z),v(z))$ is a diffeomorphism and $h_{q_1,q_2}$ is right equivalent to 
$$
u\mapsto 1 + h(u_1,\ldots,u_{d-1})+u_d^2 ,\qquad \forall u\in \Phi(N_0(r)\times (-\varepsilon,\varepsilon)).
$$

By symmetry of the metric with respect to the hyperplane $\{z=0\}$,
$$
h_{q_1,q_2}(z)=\frac{1}{2} \left(d_g(q_1,(\pi(z),z_d))^2+d_g(q_1,(\pi(z),-z_d))^2\right), \qquad \forall z\in V.
$$
Thus by symmetry, on $B_0^{d-1}(a)$,
$$
\left.\frac{\partial h_{q_1,q_2}}{\partial z_n}\right|_{z_d=0}=0.
$$
Furthermore,
$$
\left.\frac{\partial^2 h_{q_1,q_2}}{\partial z_d^2}\right|_{z_d=0}
=
2\left.\left(\frac{\partial d_g(q_1,\cdot)}{\partial z_d}\right)^2\right|_{z_d=0}
+
2 d_g(q_1,\cdot )\left. \frac{\partial^2 d_g(q_1,\cdot)}{\partial z_n^2}\right|_{z_d=0}.
$$

Notice that if $x_0\in \tilde{\Gamma}$, then the geodesic joining  $q_1$ to $q_2$ passing through $(x_0,0)$  
is supported near $(x_0,0)$ by the straight line $\{x=x_0\}$. This implies that
$\left.\dfrac{\partial d_g(q_1,\cdot)}{\partial z_n}\right|_{z_n=0}=1$ and $\left. \dfrac{\partial^2 d_g(q_1,\cdot)}{\partial z_n^2}\right|_{z_n=0}=0$.
Hence
$$
\frac{\partial^2 h_{q_1,q_2}}{\partial z_n^2}(x_0,z)=2.
$$
This allows to apply Malgrange preparation theorem: there exists $\alpha:\R^{d} \to \R$, smooth, such that 
$$
h_{q_1,q_2}(z)=\alpha(z)z_n^2+h_{q_1,q_2}(\pi(z),0)
$$
and $\alpha(x_0,0)=1$ for all $x_0\in \tilde{\Gamma}$.

The function $\alpha$ admits a uniform positive lower bound on a sufficiently small neighborhood of $\tilde{\Gamma}\times \{0\}$, hence, up to reducing $r$ and $\varepsilon$,  $\sqrt{\alpha}$ is a smooth function on this neighborhood.
As a consequence,
$$
v(z)=z_d\sqrt{\alpha(z)}
$$
and is a smooth function. This implies furthermore that $z\mapsto (\pi(z),v(z))$ is a diffeomorphism on $N_0(r)\times (-\varepsilon,\varepsilon)$ since $\partial_{z_d} u(x_0,0)=\sqrt{\alpha(x_0)}>0$ for all $x_0\in\tilde{\Gamma}$.

Notice that for any  $z\in B^d(0,a)$ such that $0<|z_d|<\zeta(\pi(z))$, $h_{q_1,q_2}(z)>2$.
Hence 
$\Gamma=\tilde{\Gamma}\times\{0\}$ and  $N=N_0(r)\times (-\varepsilon,\varepsilon)$ is a neighborhood of $\Gamma$.
Furthermore, we have proved that 
$$
h_{q_1,q_2}\left(\Phi^{-1}(u_1,\dots, u_n)\right)=1+h(u_1,\dots ,u_{d-1})+u_d^2,
$$
for all $u\in\Phi(N)$.
Once this fact is proved, what remains to be shown is the shape of the heat kernel. However this is a direct application of Corollary~\ref{C:MolchanovLaplace}, hence the statement.
\end{proof}

\begin{remark}

Our treatment of prescribing singularities for the hinged energy function in the Riemannian case appears local; for example, the case of antipodal points on the standard sphere goes beyond the framework of Theorem \ref{THM:LocalPrescription}. However, that is essentially the only situation not included in the theorem. To be more precise, for fixed points $q_1$ and $q_2$, $\Gamma$ can be identified as a subset of the sphere of radius $\dist(q_1,q_2)/2$ in $T_{q_1}M$. If $\Gamma$ is the entire sphere, then necessarily we have the Morse-Bott case as covered by Proposition~\ref{P:morse-bott}. Otherwise, since $\Gamma$ is closed, for a point $q$ on the sphere not in $\Gamma$, it has a neighborhood which is not in $\Gamma$, and stereographic projection around $q$ maps $\Gamma$ to a subset of $B^{d-1}(0,a)$ for some $a>0$. Thus every case in which $\Gamma$ is not the entire sphere of radius $\dist(q_1,q_2)/2$ in $T_{q_1}M$ can be realized as in Theorem \ref{THM:LocalPrescription}.
\end{remark}

We follow on the preceding proof by showing it can be extended to construct prescribed singularities for the hinged energy function also on contact sub-Riemannian structures, and we consider this to be sufficient for this line of inquiry.

\begin{theorem}
Let $M$ be a $2d+1$-dimensional contact manifold, let $a$ and $\sigma$ be positive real numbers, and let $h$ be a smooth, real-valued function in a neighborhood of $\overline{B^{2d-1}(0,a)}\subset \R^{d-1}$ such that $h(0,\ldots,0)=0$, $h$ is non-negative on $\overline{B^{2d-1}(0,a)}$, and $h$ is positive on $\partial B^{2d-1}(0,a)$. Then there exists a sub-Riemannian metric on $M$ (compatible with the contact structure), and some points $q_1$ and $q_2$ such that $\Gamma=\Gamma(q_1,q_2)$ is contained in a coordinate patch
\[
(u_1,\ldots,u_{2d+1}): U \rightarrow B^{2d-1}(0,a) \times(-\delta,\delta)\times(-\delta,\delta)
\]
such that
\[
h_{q_1,q_2}|_N= \frac{\sigma^2}{4} + h(u_1,\ldots,u_{2n-1}) + u^2_{2n}+u^2_{2n+1} 
\]
and the analogue of \eqref{Eqn:PrescribedExp} holds.
\end{theorem}

\begin{proof}

By the Darboux theorem, any point has a neighborhood that is contactomorphic to the standard contact structure. Thus we may take $N$ to be a neighborhood for the origin in $\bR^{2d+1}$ with the standard contact structure. Moreover, by rescaling, we can take $q_1=(-1,\ldots,0)$, $q_2=(1,0,\ldots,0)$, and $N$ a ball around the origin of Euclidean radius 3. Also recall that contact sub-Riemannian structures don't admit non-trivial abnormals, so we don't need to worry that the metric we construct will have any.

It's convenient to use more standard notation for our coordinates, so let $(v_1,w_1,\ldots, v_d, w_d, u)$ be coordinates on $\lp \bR^2 \rp^d \times \bR$. Then every admissible curve is given as the lift of a curve in $\lp \bR^2 \rp^d$. In particular, let $\tilde{\gamma}(t)=\lp v_1(t),w_1(t),\ldots, v_d(t), w_d(t) \rp$ be a curve in $\lp \bR^2 \rp^d$, let 
$$A_i(t)=\int_0^t v_i(t)\, \diff w_i -w_t(t)\, \diff v_i$$ 
be twice the enclosed signed area of the projection to the $i$th $\bR^2$ factor, and let $u(t) = \sum_{i=1}^d A_i(t)$. Then $\gamma(t) = \lp \tilde{\gamma}(t), u(t)\rp=\lp v_1(t),w_1(t),\ldots, v_d(t), w_d(t), u(t) \rp$ is the lift of $\tilde{\gamma}(t)$.

Moreover, given a Riemannian metric $\tilde{g}$ on $\lp \bR^2 \rp^d$, it lifts to a sub-Riemannian metric $g$ on the contact structure, which is invariant under translation in the $u$-direction and which has the property that the length of any admissible curve $\gamma$ is the Riemannian length of its projection $\tilde{\gamma}$ (with respect to $\tilde{g}$, of course). By the previous theorem, we can choose $\tilde{g}$ such that, if $\tilde{h}_{x,y}$ is the hinged energy function on $\lp \bR^2 \rp^d$ with respect to $\tilde{g}$, then $\tilde{h}_{q_1,q_2}$ has normal form 
$$ \frac{\sigma^2}{4} + h(u_1,\ldots,u_{2d-1}) + u^2_{2d}.$$ 
Also, recall that the metric is symmetric under reflection in the $v_1$ axis ($v_1 \mapsto -v_1$), and thus the midpoint set $\tilde{\Gamma}$ is contained in the hyperplane $\{ v_1=0 \}$.

Now consider the corresponding sub-Riemannian lifted metric $g$ and associated hinged energy function $h_{q_1,q_2}$--- we claim that $h_{q_1,q_2}$ has the desired normal form. First, consider a point
\[
z=(v_1,w_1,\ldots, v_d, w_d, u)\in \lp \bR^2 \rp^d \times \bR,
\]
and let $\pi(z)= (v_1,w_1,\ldots, v_d, w_d)\in \lp \bR^2 \rp^d$ be the projection. Then letting $d$ and $\tilde{d}$ denote the distance functions on $\lp \bR^2 \rp^d \times \bR$ and $\lp \bR^2 \rp^d$, respectively, we see that $d(q_1,z)\geq \tilde{d}(q_1,\pi(z))$, with equality if and only if there is a minimizing geodesic $\tilde{\gamma}$ from $q_1$ to $\pi(z)$ such that the endpoint of the lift $\gamma$ is $z$ (that is, if and only if there is a minimizing geodesic that encloses the ``right'' signed area).

Take $\tilde{z}\in\tilde{\Gamma}$. We know that there is a unique (and non-conjugate) minimizing geodesic $\tilde{\gamma}$ from $q_1$ to $\tilde{z}$, and thus there is a unique $z$ such that $\pi(z)=\tilde{z}$ and $h(q_1,z)=\tilde{h}(q_1,\tilde{z})$; we write this $z$ as $(\tilde{z},\overline{u}(\tilde{z}))$. Further, by the reflection symmetry of the metric, the minimal geodesic from $q_2$ to $\tilde{z}$ is given by the reflection of $\tilde{\gamma}$ (under the map $v_1\mapsto -v_1$), and thus $(\tilde{z},\overline{u}(\tilde{z}))$ is also the unique $z$ such that $\pi(z)=\tilde(z)$ and $h(q_2,z)=\tilde{h}(q_2,\tilde{z})$. It follows that $\Gamma$ is the lift of $\tilde{\Gamma}$ under the map $\tilde{z}\mapsto (\tilde{z},\overline{u}(\tilde{z}))$ (which is well defined on $\tilde{\Gamma}$), and that $h(\Gamma)=\tilde{h}(\tilde{\Gamma})$. Moreover, we know that there is a neighborhood of $\tilde{\Gamma}$ such that every point is joined to $q_1$ by a unique, non-conjugate minimizing geodesic. If we let $U$ be the intersection of this neighborhood with the hyperplane $\{ v_1=0 \}\subset \lp \bR^2 \rp^d $, then the map $\tilde{z}\mapsto (\tilde{z},\overline{u}(\tilde{z}))$ extends to $U$, by the same argument. Further, by the smoothness of the exponential map (and of the enclosed area as a function of the curve) this is a smooth embedding of $U$ into the hyperplane $\{ v_1=0 \}\subset \lp \bR^2 \rp^d \times \bR$ such that $h((\tilde{z},\overline{u}(\tilde{z}))=\tilde{h}(\tilde{z})$, for any $\tilde{z}\in U$. Denote this embedding by $\overline{U}$.

We are now in a position to show that $h_{q_1,q_2}$ has the desired normal form. First, restricting our attention to $\overline{U}$, it follows from the above that there exist coordinates on $\overline{U}$ such that 
$$
h_{q_1,q_2}|_{\overline{U}}= \frac{\sigma^2}{4} + h(u_1,\ldots,u_{2d-1})
$$ 
(that is, $h|_{\overline{U}}$ has the same ``normal form'' as $\tilde{h}|_U$). If we show that the Hessian of $h$ on the normal bundle of $\overline{U}$ is non-degenerate, which is 2-dimensional and spanned by $\partial_{v_1}$ and $\partial_u$, then the Malgrange preparation theorem (or parametrized Morse lemma) for smooth functions implies that we can find coordinates in which $h$ has the desired expression on all of $N$. Consider the Hessian (as a quadratic form), at a point $z_0\in\Gamma$, along a vector $\alpha \partial_u + \beta \partial_{v_1}$. If $\beta\neq 0$, then because $d(x,z)\geq \tilde{d}(x,\pi(z))$ and the Hessian of $\tilde{d}(x,\pi(z_0))$ along $\beta \partial_{v_1}$ is positive, the Hessian of $h$ is also positive. So it remains only to show that the Hessian along $\partial_u$ is positive. 

Again consider $z_0=(v_1,w_1,\ldots, v_d, w_d,u_0)\in\Gamma$, and let $z_s=(v_1,w_1,\ldots, v_d, w_d,u_0+s)$. Let $\tilde{\gamma}$ be the unique, non-conjugate minimal geodesic from $q_1$ to $\pi(z_0)$ and $\gamma_0$ its lift. Now let $\gamma_s$ be the unique, non-conjugate sub-Riemannian minimal geodesic from $q_1$ to $z_s$, and let $\tilde{\gamma}_s$ be its projection (this is well defined for $s$ near 0 because the complement of the cut locus is open, in both Riemann and sub-Riemannain geometry, and in particular, the relevant exponential maps are local diffeomorphisms). Then $\tilde{\gamma}_s$ is the unique shortest curve from $q_1$ to $\pi(z_0)$ (with respect to the Riemannian metric on $\lp \bR^2 \rp^d$), subject to the constraint that the endpoint lifts to $z_0$ (that is, subject to the constraint that it encloses the right area, making it the solution to the appropriate Dido problem). Further, $\tilde{\gamma}_s$ is a one-parameter family of proper deformations of $\tilde{\gamma}$ (meaning the endpoints are kept fixed), and the variation field at $s=0$ is non-trivial because the enclosed area in changing to first-order. But, by the classical theory of the second variation of energy near a minimizing, non-conjugate Riemannain geodesic, this means that the second derivative of the length of $\tilde{\gamma}_s$ is positive at $s=0$. Since this length is also $d(q_1,z_s)$, and since $d(q_2,z_s)=d(q_1,z,s)$ by symmetry, it follows that $\frac{\partial^2}{\partial s^2}h(z_s)>0$. Recalling the definition of $z_s$, this completes the construction of the metric.

From here, the heat kernel representation follows as before.
\end{proof}

One virtue of Theorem \ref{THM:LocalPrescription} is that many of the real-analytic normal forms appearing in \cite{Arnold_singularities_diff_maps2} and corresponding to local minima can be realized as $h_{q_1,q_2}$
on some Riemannian manifold $M$ of a high enough dimension to support the normal form in question. The only restriction is that one needs the geodesic direction to be separate from the others.
The corresponding Laplace asymptotic expansions can be realized as heat kernel asymptotics on such manifolds, which means that there are cases when the heat kernel asymptotics contain powers of $\log t$, for instance. Indeed, we have the following corollary.

\begin{corollary}
 For any integers $d\geq 2$, $p\geq 1$, and $0\leq k\leq d-2$, there exists a smooth Riemannian manifold $M$ of dimension $d$, and $q_1,q_2$ in $M$, $q_1\neq q_2$,
such that for some $C\neq 0$,
$$
 p_t\lp q_1,q_2\rp
 =
 e^{-\frac{d^2(q_1,q_2)}{4t}} t^{\frac{1}{2}+\frac{1}{2p}-d}\log(t)^k\big( C+o(1)\big).
$$
\end{corollary}

\begin{proof}
This is a matter of applying Theorem~\ref{THM:LocalPrescription} to the right function $h$.

From \cite[Theorems 7.3-7.4]{Arnold_singularities_diff_maps2}, we have that for any smooth non-negative function $\phi:\R^{k+1}\to \R$, positive at $0$, we have the following Laplace integral asymptotics near $t=0$:
\begin{equation}\label{E:normal_form_Arnold}
\int_{\R^k}
\exp\lp\frac{u_1^{2p}\cdots u_{k+1}^{2p}}{t}\rp
\phi(u_1,\dots u_{k+1})\diff u_1 \cdots\diff u_{k+1}=
t^{1/2p}\log(t)^k \big(C+o(1)\big).
\end{equation}
(With $C$ a non-zero constant on the only condition that $\phi(0)\neq 0$.)

Let $h:\R^{d-1}\to \R$ be defined by
$$
h(u_1,\dots ,u_{d-1})=u_1^{2p}\cdots u_{k+1}^{2p}+\chi(u_1,\dots , u_{d-1}).
$$
Here $\chi:\R^{d-1}\to [0,1]$ is a smooth function, equal to $0$ on $B^{d-1}(0,a)$, and equal to $1$ on the complement of $B^{d-1}(0,a+1)$, for some $a>0$.

Thus there exists a smooth Riemannian manifold $M$ of dimension $d$, $q_1,q_2$ in $M$ such that $q_1\neq q_2$, and 
\begin{equation}\label{E:Molchanov_prescribed_1st_form}
 p_t\lp q_1,q_2\rp = \frac{1}{t^d} e^{-\frac{d^2(q_1,q_2)}{4t}} 
 \int_{(-\varepsilon,\varepsilon)^d} \Phi(t,u)\,\,  e^{-\frac{h(u_1,\ldots,u_{d-1})+u_d^2}{t}} \,\diff u_1\cdots \diff u_d +O\lp e^{-\frac{d^2(q_1,q_2)+c}{4t}} \rp.
\end{equation}
for some positive $\varepsilon$ and a  smooth prefactor function $\Phi$ over $\R^+\times (-\varepsilon,\varepsilon)^d $, smoothly extendable and postitive at $t=0$.

For $\varepsilon$ small enough, if $\sum_{i=1}^{d-1}u_i^2<\varepsilon^2$, $h(u_1,\dots, u_{d-1})=u_1^{2p}\cdots u_{k+1}^{2p}$.
Thus equation~\eqref{E:Molchanov_prescribed_1st_form} implies
\begin{equation}\begin{split}\label{E:Molchanov_prescribed_2nd_form}
 & p_t\lp q_1,q_2\rp  = \\  &\quad  \frac{1}{t^d} e^{-\frac{d^2(q_1,q_2)}{4t}} 
 \int_{(-\varepsilon,\varepsilon)^d} (\psi_0(u)+t \psi_1(t,u))\,\,  e^{-\frac{u_1^{2p}\cdots u_{k+1}^{2p}+u_d^2}{t}} \,\diff u_1\cdots \diff u_d +O\lp e^{-\frac{d^2(q_1,q_2)+c}{4t}} \rp
\end{split}\end{equation}
with $\psi_0:(-\varepsilon,\varepsilon)^d\to \R$, positive at $0$ and smooth, and  $\psi_1: \R^+\times(-\varepsilon,\varepsilon)^d\to \R$, smoothly extendable at $t=0$.

Then
$$
\left|\int_{(-\varepsilon,\varepsilon)^d}t \psi_1(t,u)\,\,  e^{-\frac{u_1^{2p}\cdots u_{k+1}^{2p}+u_d^2}{t}} \,\diff u_1\cdots \diff u_d\right|\leq Ct
\int_{(-\varepsilon,\varepsilon)^d}e^{-\frac{u_1^{2p}\cdots u_{k+1}^{2p}+u_d^2}{t}} \,\diff u_1\cdots \diff u_d.
$$
From the formula~\eqref{E:normal_form_Arnold}, we have that for some $C\neq 0$,
$$
\int_{(-\varepsilon,\varepsilon)^d}e^{-\frac{u_1^{2p}\cdots u_{k+1}^{2p}+u_d^2}{t}}\diff u_1\cdots \diff u_{k+1}\diff u_d
=
t^{1/2}t^{1/2p}\log(t)^k \big(C+o(1)\big).
$$
Likewise
\begin{multline*}
 \int_{(-\varepsilon,\varepsilon)^d} \psi_0(u)e^{-\frac{u_1^{2p}\cdots u_{k+1}^{2p}+u_d^2}{t}} \diff u_1\cdots \diff u_d
 =\\
 \int_{(-\varepsilon,\varepsilon)^{k+2}} \left(\int_{(-\varepsilon,\varepsilon)^{d-k-2}} \psi_0(u)\diff u_{k+2}\cdots \diff u_{d-1}\right)e^{-\frac{u_1^{2p}\cdots u_{k+1}^{2p}+u_d^2}{t}} \diff u_1\cdots \diff u_{k+1}\diff u_d.
\end{multline*}
Then $\Psi_0:(-\varepsilon,\varepsilon)^{k+2}\to \R$ given by $\Psi_0(u_1,\dots,u_{k+1},u_d)=\int_{(-\varepsilon,\varepsilon)^{d-k-2}} \psi_0(u)\diff u_{k+2}\cdots \diff u_{d-1}$ is a smooth positive function and 
$$
 \int_{(-\varepsilon,\varepsilon)^{k+2}} \Psi_0(u)e^{-\frac{u_1^{2p}\cdots u_{k+1}^{2p}+u_d^2}{t}} \diff u_1\cdots \diff u_{k+1}\diff u_d
 =
 t^{1/2}t^{1/2p}\log(t)^k \big(C+o(1)\big).
$$

Putting all three parts of \eqref{E:Molchanov_prescribed_2nd_form} together, we get the asymptotic expansion first term:
$$
 p_t\lp q_1,q_2\rp
 =
 e^{-\frac{d^2(q_1,q_2)}{4t}} t^{\frac{1}{2}+\frac{1}{2p}-d}\log(t)^k\big( C+o(1)\big).
$$
\end{proof}

Theorem~\ref{THM:LocalPrescription} also allows to go beyond functions admitting an analytic normal form, such as present in \cite{Arnold_singularities_diff_maps2}. In that case, an asymptotic expansion of the Laplace integral in the theorem is not accessible by the methods of \cite{Arnold_singularities_diff_maps2}, 
and moreover, appear not to be known. To illustrate, we offer the following examples.

\begin{example} Let
\[
h(u_1)=\begin{cases}
e^{-1/u_1^2} & \text{ for $u_1\neq0$} \\
0 & \text{ for $u_1=0$} 
\end{cases}
\]
on $(-\varepsilon,\varepsilon)\subset \R$. Then it's well known that $h$ satisfies the hypotheses of Theorem \ref{THM:LocalPrescription}.

\end{example}

\begin{example}
Let $g(\theta)$ be a smooth function on $\S^1$ which is equal to $\theta^2$ near $\theta=0$ and is strictly positive elsewhere. Then
in polar coordinates on $\R^2$, let
\[
h(r,\theta)=g(\theta)(r-1)^2 + (r-1)^4
\]
near the circle $\{r=1\}$ in $\R^2$, and extended to be greater than some $\eps>0$ elsewhere. This gives a situation where $\Gamma=\S^1$ and where $h_{q_1,q_2}$ is locally Morse-Bott away from $\theta=0$, but where the Hessian in the normal direction degenerates as $\theta$ approaches 0. Thus, the usual Morse-Bott expansion of Section \ref{Sect:MB} does not apply. Of course, the $(r-1)^4$ can be replaced by $(r-1)^{2k}$ for any positive integer $k$, or even by
\[
\begin{cases}
e^{-1/(r-1)^2} & \text{ for $r\neq1$} \\
0 & \text{ for $r=1$} 
\end{cases}
\]
to produce other examples in a similar vein, and similarly, $g(\theta)$ can have behavior near $\theta=0$ modeled on any even power of $\theta$ or on $e^{-1/\theta^2}$.
\end{example}

\begin{example}
Let $h(u_1)$ be a smooth, non-negative function with zeroes at $\pm\frac{1}{n}$ for all positive integers $n$ and at 0. The existence of such functions is well-known, and while the Hessian can be made non-degenerate at all of the $\pm\frac{1}{n}$ (although it need not be), $h$ necessarily vanishes to all order at 0. Moreover, in this case, $\Gamma$ is not a union of smooth submanifolds (the condition to respect the submanifold topology is not satisfied at 0).

\end{example}

\section{Logarithmic derivatives}

\subsection{Molchanov-type expansions of logarithmic derivatives}
{
We start by introducing an alternative representation of Molchanov method that will be useful in following computations. This is a direct consequence of L\'eandre estimates coupled with Ben Arous expansions on compact sets with no abnormal geodesics (Proposition~\ref{P:BenArous}).

\begin{lemma}[Folding the remainder]\label{L:Molchanov_no_remainder}

Let $\Sigma:\R^+\times M^2\setminus \C\to \R$ denote the smooth function such that 
$$
\Sigma_t(x,y)=
t^{d/2} \e^{\frac{d(x,y)^2}{4t}}
p_t(x,y).
$$

Let $\K$ be a localizable compact subset of $M^2\setminus \D$ such that all minimizers between pairs $(x,y)\in\K$ are strongly normal.

Let $\mathcal{V}$ be an open subset of $M^2$ containing $\K$, such that the closure of $\mathcal{V}$ is a compact localizable subset of $M^2\setminus \D$.
For $\varepsilon,t_0>0$, we set
$\Omega$ to be the open set
$$
\left\{
(t,x,y,z)\in (0,t_0)\times \mathcal{V}\times M
\mid 
d(x,z)< \frac{d(x,y)}{2}+\varepsilon \text{ and }  d(z,y)< \frac{d(x,y)}{2}+\varepsilon
\right\},
$$
where $\varepsilon$ is assumed small enough so that $(x,z)$ and $(z,y)$ avoid $\C$. (By definition, $(t,x,y,z)\in\Omega$ for all $t\in (0,t_0)$, $(x,y)\in \K$, $z\in \Gamma_\varepsilon$.)

Suppose we are in the symmetric case. Then there exists a  continuous  map $\bar{\Sigma}:\Omega \to \R$, smooth as a map of $(t,y)$,
such that for all $(x,y)\in \K$, for all $t<t_0$,
\begin{equation}\label{Eqn:Folded}
p_{t}(x,y)=\int_{\Gamma_{\varepsilon}}
\lp\frac{2}{t}\rp^{d}
\e^{-\frac{h_{x,y}(z)}{4t}}
\Sigma_{t/2}(x,z)
\Sigmax_{t/2}(z,y)
\diff \mu(z)
\end{equation}
and  for all $l\in \N$ and  $\alpha$ multi-index , there exists $C>0$ such that for all $(x,y)\in \K$,
$$
\partial^l_t Z^\alpha_y
\lb
	\Sigmax_t(z,y)-\Sigma_t(z,y)
\rb
\leq 
C\e^{-\frac{\varepsilon^2}{8t}}.
$$
In particular, for all $l\in \N$ and  $\alpha$ multi-index, for all $t,x,y,z\in \Omega$,
$$
{\partial^l_t Z^\alpha_y}_{|t=0}
	\Sigmax_t(z,y)=
{\partial^l_t Z^\alpha_y}_{|t=0}	
	\Sigma_t(z,y).
$$
In the general (non-symmetric) case, all of the above holds with $l=0$.
\end{lemma}

Before proving this statement, we detach the intermediate step of proving that some $\mathcal{V}$ must exist for both possible localization conditions.

\begin{lemma}\label{L:existsV}
Let $\K\subset M^2$ be a localizable compact. For $\rho>0$, let $\K'(\rho)$ be defined by
$$
\K'(\rho)=\left\{
(\xi,\zeta)\in M^2\mid \exists(x,y)\in \K \text{ s.t. } d(\xi,x)\leq \rho,d(\zeta,y)\leq \rho
\right\}.
$$
There exists $\rho_0>0$ such that $\K'(\rho)$ is localizable for all $0\leq \rho\leq \rho_0$.
\end{lemma}

\begin{proof}
Assuming the strong localization condition holds for $\K$. If $M$ is complete, all compacts are localizable and the results follows ($K'(\rho)$ is compact), so we assume incompleteness.
Regarding distance to infinity, observe that it still satisfies a form of triangular inequality, in the sense that for any two $x,y\in M$, $d(x,\infty)\leq d(x,y)+d(y,\infty)$. This implies that the map $\Phi(x,y)=d(x,\infty)+d(y,\infty)-d(x,y)$ is (uniformly) continuous on $M\times M$ and positively lower-bounded on a  compact $\K$ if and only if it satisfies the strong localization condition. By triangular inequality, for $(\xi,\zeta)\in \K'(\rho)$ and $(x,y)$ as in the definition of $\K'(\rho)$,
$$
\begin{aligned}
d(\xi,\infty)+d(\zeta,\infty)-d(\xi,\zeta)
&\geq
d(x,\infty)-d(x,\xi)+d(y,\infty)-d(y,\zeta)-(d(\xi,x)+d(x,y)+d(y,\zeta))
\\
&\geq
\Phi(x,y)-4\rho
\end{aligned}
$$
By picking $\rho_0=\inf_\K\Phi/8$, we get that $\K'(\rho)$ satisfies the strong localization condition for all $0\leq \rho\leq \rho_0$.

Assuming the weak localization condition holds for $\K$. The sector condition on $\Delta$ holds for all $M$, there exists $\varepsilon$ such that $\mathcal{E}(x,y,\varepsilon):=\{z\mid d(x,z)+d(z,y)<d(x,y)+\varepsilon\}$ has compact closure for all $(x,y)\in \K$. Let $(\xi,\zeta)\in \K'(\rho)$ and $(x,y)$ be as in the definition of $\K'(\rho)$. For any $z\in M$,
$$
d(x,z)+d(z,y)\leq d(x,\xi)+d(\xi,z)+d(z,\zeta)+d(\zeta,y)\leq d(\xi,z)+d(z,\zeta)+2\rho
$$
If $z\in \mathcal{E}(\xi,\zeta,\rho )$ then
$$
d(\xi,z)+d(z,\zeta)\leq d(\xi,\zeta)+\rho\leq d(x,y)+3\rho.
$$ 
As a consequence,
$$
d(x,z)+d(z,y)\leq d(x,y)+5\rho.
$$ 
Pick $\rho_0=\varepsilon/8$ and we get $\mathcal{E}(\xi,\zeta,\rho )\subset \mathcal{E}(x,y,\varepsilon)$ for all $0<\rho<\rho_0$, implying that $\mathcal{E}(\xi,\zeta,\rho)$ also has compact closure for all $(\xi,\zeta)\in \K'(\rho)$. Hence $\K'(\rho)$ satisfies the weak localization condition for all $0\leq \rho\leq \rho_0$.
\end{proof}

This lemma implies that some $\mathcal{V}$ as in the statement of Lemma~\ref{L:Molchanov_no_remainder} must exist, as we can pick $\mathcal{V}$ to be the interior of $\K'(\rho)$, with $0<\rho<\rho_0$ small enough that $\K'(\rho)$ avoids the diagonal.

\begin{proof}[Proof of Lemma~\ref{L:Molchanov_no_remainder}]
We write the proof of the symmetric case; setting $l=0$ gives the proof in the general case. 
We set 
$$
\zeta_t^x(y)=\int_{\Gamma_\varepsilon}\e^{-\frac{h_{x,y}(z)}{t}}\Sigma_{t/2}(x,z) \diff \mu(z).
$$
By definition, $\zeta$ is continuous on  $\Omega$  and smooth with respect to $(t,y)$. Furthermore, following Laplace integrals asymptotics, the strategy given in the proof of Proposition~\ref{P:universal_bounds} yields for all $l\in \N$ and $\alpha$ multi-index, the existence of a constant $C>0$ such that on $\Omega$
\begin{equation}\label{E:der_zeta}
\partial^l_t Z^\alpha_y \zeta_t^x(y)\leq \frac{C\e^{-\frac{d(x,y)}{4t}}}{t^{2l+|\alpha|+d-1/2}}
\end{equation}
Furthermore, when $l=0$, $\alpha=0$,
\begin{equation}\label{E:der_zeta2}
\frac{\e^{-\frac{d(x,y)}{4t}}}{C t^{d/2}}\leq  \zeta_t^x(y)\leq \frac{C \e^{-\frac{d(x,y)}{4t}}}{t^{d-1/2}}.
\end{equation}

Now set
$$
R_t^x(y)=\lp\frac{t}{2}\rp^{d} \int_{M\setminus \Gamma_\varepsilon}p_{t/2}(x,z)p_{t/2}(z,y)\diff \mu (y).
$$
By Corollary~\ref{C:MolchanovV1}, for all $l\in \N$ and $\alpha$ multi-index, 
 there exists $C$ such that for all $(t,x,y,z)\in \Omega$, 
\begin{equation}\label{E:der_R}
\partial_t Z^\alpha_y R_t^x(y)\leq C\e^{-\frac{d(x,y)^2}{4t}}\e^{-\frac{\varepsilon^2}{4t}}
\end{equation}
Furthermore,  $R$ is continuous on $\Omega$ and smooth as a function of $(t,y)$.

We now pick $\psi$ to be
$$
\Sigmax_t(z,y)=\Sigma_t(z,y)+\frac{R_t^x(y)}{\zeta_t^x(y)}.
$$
Then
\begin{multline*}
\int_{\Gamma_{\varepsilon}}
\lp\frac{2}{t}\rp^{d}
\e^{-\frac{h_{x,y}(z)}{4t}}
\Sigma_{t/2}(x,z)
\Sigmax_{t/2}(z,y)
\diff \mu(z)=
\int_{\Gamma_{\varepsilon}}
p_{t/2}(x,z)p_{t/2}(z,y)
\diff \mu(z)
\\+
\underbrace{
	\int_{\Gamma_\varepsilon}
		\e^{-\frac{h_{x,y}(z)^2}{t}}\Sigma_{t/2}(x,z)
	\diff \mu (z)
}_{=\zeta_t^x(y)} 
 \frac{R_t^x(y)}{\zeta_t^x(y)}\lp\frac{2}{t}\rp^{d}.
\end{multline*}
By construction, this equation simplifies to
\begin{multline*}
\int_{\Gamma_{\varepsilon}}
\lp\frac{2}{t}\rp^{d}
\e^{-\frac{h_{x,y}(z)}{4t}}
\Sigma_{t/2}(x,z)
\Sigmax_{t/2}(z,y)
\diff \mu(z)=\\
\int_{\Gamma_{\varepsilon}}
p_{t/2}(x,z)p_{t/2}(z,y)
\diff \mu(z)
+
\int_{M\setminus \Gamma_{\varepsilon}}
p_{t/2}(x,z)p_{t/2}(z,y)
\diff \mu(z)
\\
=\int_{M}
p_{t/2}(x,z)p_{t/2}(z,y)
\diff \mu(z)
=p_t(x,y).
\end{multline*}

To conclude the proof, we only have to combine \eqref{E:der_zeta}, \eqref{E:der_zeta2} and \eqref{E:der_R} to check that for all $l\in \N$ and $\alpha$ multi-index, there exists $C>0$, and $m\in M$ such that
$$
\partial_t^lZ_y^\alpha\frac{R_t^x(y)}{\bar\zeta_t^x(y)}\leq  C\frac{e^{-\frac{\varepsilon^2}{4t}}}{t^m}\leq  C'e^{-\frac{\varepsilon^2}{8t}}.
$$
\end{proof}

}

We now explain how Lemma \ref{L:Molchanov_no_remainder} can be used to apply Molchanov's method to logarithmic derivatives. Suppose that $x$ and $y$ are localizable and all minimizers between them are strongly normal.

We can rewrite \eqref{Eqn:Folded}, mimicking the Chapman-Kolmogorov equation, as 
\[
p_{t}(x,y)=\lp\frac{2}{t}\rp^{d} \int_{\Gamma_{\varepsilon}} \lb \e^{-\frac{\dist^2(x,z)}{2t}} \Sigma_{t/2}(x,z) \rb\cdot
\lb  \e^{-\frac{\dist^2(z,y)}{2t}} \Sigmax_{t/2}(z,y) \rb
\diff \mu(z) .
\]
If $Z$ is a smooth vector field in a neighborhood of $y$, we have
\begin{equation}\label{Eqn:ProdRule}\begin{split}
& Z_y \lb  \e^{-\frac{\dist^2(z,y)}{2t}} \Sigmax_{t/2}(z,y) \rb = \\
& \qquad -\frac{d(z,y)}{t}\cdot Z_y d(z,y)\cdot \lb  \e^{-\frac{\dist^2(z,y)}{2t}} \Sigmax_{t/2}(z,y) \rb +Z_y\lp \log \Sigmax_{t/2}(z,y)\rp \cdot \lb  \e^{-\frac{\dist^2(z,y)}{2t}} \Sigmax_{t/2}(z,y) \rb ,
\end{split}\end{equation}
so that (since it's clear we can differentiate under the integral sign)
\begin{multline*}
 Z_y \log p_t(x,y) = \frac{Z_y p_t(x,y)}{p_t(x,y)}
 \\
= \frac{\int_{\Gamma_{\eps}} \lb -\frac{d(z,y)}{t}Z_y d(z,y)+ Z_y\lp \log \Sigmax_{t/2}(z,y)\rp \rb \Sigma_{t/2}(x,z)\Sigmax_{t/2}(z,y) e^{-\frac{h_{x,y}(z)}{t}} \, \diff\mu(z)}{\int_{\Gamma_{\eps}} \Sigma_{t/2}(x,z)\Sigmax_{t/2}(z,y) e^{-\frac{h_{x,y}(z)}{t}} \, \diff\mu(z)}
\end{multline*}

 To better understand the right-hand side of the above, note that 
\begin{equation}\label{Eqn:M_tDef}
 \frac{\Sigma_{t/2}(x,z)\Sigmax_{t/2}(z,y) e^{-\frac{h_{x,y}(z)}{t}} }{\int_{\Gamma_{\eps}} \Sigma_{t/2}(x,z)\Sigmax_{t/2}(z,y) e^{-\frac{h_{x,y}(z)}{t}}\, \diff\mu(z)} \cdot \Ind_{\Gamma_{\eps}}(z)
\end{equation}
 is the density, with respect to $\mu$, of a probability measure supported on $\Gamma_{\eps}$. We call this probability measure $m_t$, for $t>0$. Since $\Gamma_{\eps}$ is a compact subset of a smooth manifold, $m_t$ is determined by its integrals against smooth functions (on $M$), and also weak convergence of probability measures on $\Gamma_{\eps}$ can be characterized by the convergence of their integrals against smooth functions. The upshot of this is that we have
\begin{equation}\label{Eqn:1stLogDer}
  Z_y \log p_t(x,y) =  \E^{m_t}\lb  -\frac{d(\cdot,y)}{t}Z_y d(\cdot,y)+ Z_y\lp \log \Sigmax_{t/2}(\cdot,y)\rp  \rb .
\end{equation}
Further, while $m_t$ (and in particular $\Sigmax$) is defined so as to make this an equality, we are interested in asymptotic behavior. To this end, observe that we can write $\Sigma_{t/2}(x,z)$ and $\Sigmax_{t/2}(z,y)$ as $c_0(x,z)+O(t)$ and $c_0(z,y)+O(t)$, and recall that the $c_0$ are smooth and strictly positive. We see that, if $m_{t_n}\rightarrow m_0$ for some sequence of times $t_n\searrow 0$ and some probability measure $m_0$, then, for any smooth $f$, we have
\begin{equation}\label{Eqn:M_tLim}
 \frac{\int_{\Gamma_{\eps}} f(z) c_0(x,z)c_0(z,y) e^{-\frac{d^2(x,z)+d^2(z,y)}{2t_n}} \, \diff\mu(z)}{\int_{\Gamma_{\eps}} c_0(x,z)c_0(z,y) e^{-\frac{d^2(x,z)+d^2(z,y)}{2t_n}} \, \diff\mu(z)} \rightarrow  \E^{m_0}\lb f\rb ,
\end{equation}
and conversely, if there is some $m_0$ and some sequence of times $t_n$ such that \eqref{Eqn:M_tLim} holds for all smooth $f$, then $m_{t_n}\rightarrow m_0$.

The derivation of \eqref{Eqn:1stLogDer} naturally extends to higher-order derivatives.  If $Z'$ is another smooth vector field in a neighborhood of $y$, we see that
\[
 Z'_yZ_y\lp \log p_t(z,y)\rp = \frac{Z'_yZ_y p_t(z,y)}{p_t(z,y)} - Z'_y\lp \log p_t(z,y)\rp\cdot Z_y\lp \log p_t(z,y)\rp .
\]
Then applying $Z'_y$ to both sides of \eqref{Eqn:ProdRule} to get an expression for $Z'Z \lp \log p_t(x,y)\rp$, we see (writing $ Z'=Z'_y$ and $ Z=Z_y$ to unburden the notation) that
\begin{equation}\label{Eqn:2ndLogDer} \begin{split}
& Z'Z \lp \log p_t(x,y)\rp = \E^{m_t}\lb -\frac{1}{t}Z'd(\cdot,y)\cdot Zd(\cdot,y)-\frac{d(\cdot,y)}{t}Z'Zd(\cdot,y)+\frac{Z'Z\lp\Sigmax_{t/2}(\cdot,y)\rp}{\Sigmax_{t/2}(\cdot,y)}\rb \\
&+ \E^{m_t} \lb \lb -\frac{d(\cdot,y)}{t}Zd(\cdot,y)+ Z\lp \log \Sigmax_{t/2}(\cdot,y)\rp\rb \cdot\lb -\frac{d(\cdot,y)}{t}Z'd(\cdot,y)+ Z'\lp \log \Sigmax_{t/2}(\cdot,y)\rp\rb \rb \\
& -  \E^{m_t}\lb  -\frac{d(\cdot,y)}{t}Z d(\cdot,y)+ Z\lp \log \Sigmax_{t/2}(\cdot,y)\rp  \rb
\cdot  \E^{m_t}\lb  -\frac{d(\cdot,y)}{t}Z' d(\cdot,y)+ Z'\lp \log \Sigmax_{t/2}(\cdot,y)\rp  \rb
\end{split}
\end{equation}

We note that while the above is valid whenever $x$ and $y$ are localizable and all minimizers between them are strongly normal, it is of interest primarily when, in addition, $y$ is in the cut locus of $x$. That's because otherwise the asymptotics of the log-derivatives of $p_t(x,y)$ are fairly straightforward, as we now show.

\begin{theorem}\label{THM:CleanAsymptotics}
For any localizable compact $\K\subset M^2\setminus\C$, and any finite family $\lc Z_1,\dots Z_m\rc$ of smooth vector fields on  $\K$, we have that for any multi-index $\alpha$,
\begin{equation}\label{Eqn:DistMollified}
 \lim_{t\searrow 0} t Z^{\alpha}_y \log p_t(x,y) = -\frac{1}{4} Z^{\alpha}_y d^2(x,y),
\end{equation}
uniformly for $(x,y)\in \K$.
\end{theorem}

\begin{proof}
On such a $\K$, Theorem \ref{T:UnifBenArous} gives that
\[
t\log p_t(x,y) = -\frac{d}{2}t\log t - \frac{d^2(x,y)}{4} +t\log \lp c_0(x,y)+tR(t,x,y) \rp
\]
where $R$ is some remainder function which is bounded (uniformly) along with all its derivatives as $t\rightarrow 0$. Since $c_0(x,y)$ is bounded above and below by (strictly) positive constants, taking spatial derivatives gives
\begin{equation}\label{Eqn:BADer}
t Z^{\alpha}_y \log p_t(x,y) = -\frac{1}{4} Z^{\alpha}_y d^2(x,y) + tR'(t,x,y)
\end{equation}
where $R'$ is (uniformly) bounded along with all its derivatives as $t\rightarrow 0$. The result follows.
\end{proof}

On the other hand, on the non-abnormal cut locus, \eqref{Eqn:1stLogDer} and \eqref{Eqn:2ndLogDer} give, to leading order,
\begin{equation}\label{E:First2moments}
\begin{aligned}
t \cdot Z_y \log p_t(x,y) =& \E^{m_t}\lb -d(\cdot,y) Z_yd(\cdot,y) \rb +O(t) \\
\text{and}\quad t\cdot Z'_yZ_y \log p_t(x,y) 
=& \frac{1}{t}\Big\{ \E^{m_t} \lb d^2(\cdot,y) Z_yd(\cdot,y)Z'_yd(\cdot,y) \rb \\
&\qquad  - \E^{m_t} \lb d(\cdot,y) Z_yd(\cdot,y) \rb \E^{m_t} \lb d(\cdot,y) Z'_yd(\cdot,y)\rb \Big\}   +O(1) \\
=& \frac{1}{t} \Cov^{m_t} \Big( d(\cdot,y) Z_yd(\cdot,y), d(\cdot,y) Z'_yd(\cdot,y) \Big) +O(1) .
\end{aligned}
\end{equation}
Extending this to higher-order derivatives is the content of Theorem \ref{THM:Cumulants}, which we now prove.

\begin{proof}[Proof of Theorem \ref{THM:Cumulants}]
Fa\`a di Bruno's formula implies that
\begin{equation}\label{Eqn:Faa}
Z^N\cdots Z^1 \log p_t(x,y)= \sum_{\pi\in\Pi}\lp \frac{(-1)^{|\pi|-1}\lp|\pi|-1\rp !}{p_t^{|\pi|}(x,y)}\prod_{B\in\pi}Z^Bp_t(x,y)\rp
\end{equation}
where the sum is over all partitions $\pi$ of $\{N,N-1,\ldots, 2,1\}$, $|\pi|$ denotes the number of blocks in the partition $\pi$, the product is over all blocks $B$ in $\pi$, and $Z^Bp_t(x,y)$ means $Z^{k_m}\cdots Z^{k_1}p_y(x,y)$ where $k_m>\cdots>k_1$ are the elements of $B$. As above, we can use Molchanov's method to write derivatives of $p_t$ as
\[
Z^B p_t(x,y) = \int_{\Gamma_{\eps}} \lp \frac{2}{t}\rp^d \Sigma_{t/2}(x,z)e^{-\frac{d^2(x,z)}{2t}} \cdot Z^B\lb e^{-\frac{d^2(z,y)}{2t}}\Sigmax_{t/2}(z,y)\rb \, \diff\mu(z) .
\]
Further, we see that
\[
Z^B\lb e^{-\frac{d^2(z,y)}{2t}}\Sigmax_{t/2}(z,y)\rb = \sum_{I\subset B} Z^I\lb e^{-\frac{d^2(z,y)}{2t}}\rb \cdot Z^{I^c} \lb\Sigmax_{t/2}(z,y)\rb ,
\]
where the sum is over all subsets $I$ of $B$ and $I^c$ is the complement of $I$ relative to $B$ (if $I$ is empty, we understand $Z^I\lb e^{-\frac{d^2(z,y)}{2t}}\rb$ to be $e^{-\frac{d^2(z,y)}{2t}}$ and similarly if $I^c$ is empty). Finally, another application of Fa\`a di Bruno's formula shows that
\[
Z^I\lb e^{-\frac{d^2(z,y)}{2t}}\rb = e^{-\frac{d^2(z,y)}{2t}} \sum_{\pi'\in\Pi'} \lp-\frac{1}{2t}\rp^{|\pi'|} \prod_{B'\in\pi'} Z^{B'} \lb d^2(z,y) \rb
\]
where the sum is over all partitions of $I$ (and the vector fields are applied ``in order'' as above).

Combining the above is a bit messy. Nonetheless, we let $\pi(1),\ldots,\pi(i),\ldots$ enumerate the partitions of $\{N,N-1,\ldots, 2,1\}$, $B(j,i)$ enumerate the blocks of $\pi(i)$, $I(k,j,i)$ enumerate the subsets of $B(j,i)$, $\pi'(\ell,k,j,i)$ enumerate the partitions of $I(k,j,i)$, and $B'(m,\ell,k,j,i)$ enumerate the blocks of $\pi'(\ell,k,j,i)$. Then if $c_i=(-1)^{|\pi_i|-1}\lp|\pi_i|-1\rp !$, we have
\begin{multline}
\label{Eqn:Combos} 
 Z^N\cdots Z^1 \log p_t(x,y)= \\ 
 \left(\sum_i c_i  \right)
\prod_j \sum_k \E^{m_t}\lb \frac{Z^{I^c(k,j,i)}\Sigmax_{t/2}(z,y)}{\Sigmax_{t/2}(z,y)} \cdot \sum_{\ell} \lp -\frac{1}{2t}\rp^{|\pi'_{\ell}|}
\prod_m Z^{B'(m,\ell,k,j,i)} d^2(z,y) \rb
\end{multline}
where if $I(k,j,i)=\emptyset$, the expectation is understood as simply $ \E^{m_t}\lb Z^{I^c(k,j,i)}\Sigmax_{t/2}(z,y)/\Sigmax_{t/2}(z,y)\rb$ while if $I^c(k,j,i)=\emptyset$, we have $Z^{I^c(k,j,i)}\Sigmax_{t/2}(z,y)= Z^{\emptyset}\Sigmax_{t/2}(z,y) =\Sigmax_{t/2}(z,y)$.

Explicitly expanding this and collecting terms based on the power of $-1/(2t)$ is, fortunately, unnecessary. First, note that $\Sigmax_{t/2}(z,y)$ and $d^2(z,y)$ are both smooth on a neighborhood of $\Gamma_{\eps}$ (and $\Sigmax_{t/2}(z,y)$ is bounded from below by a positive constant), so that, after factoring out the $\lp -\frac{1}{2t}\rp^{|\pi'_{\ell}|}$, all of the remaining expectations are finite, and moreover, bounded for all sufficiently small $t$ solely in terms of bounds on $\Sigmax_{t/2}(z,y)$ and $d^2(z,y)$ and their first $N$ derivatives (with respect to the $Z^i$). Further, we see that the largest power of $-1/(2t)$ we get in the expansion of the right-hand side of \eqref{Eqn:Combos} is $\lp-\frac{1}{2t}\rp^N$, which, for any given partition $\pi(i)$, occurs exactly when each $I(k,j,i)=B(j,i)$ and each $I(k,j,i)$ is partitioned into singletons. This gives
\[\begin{split}
 Z^N\cdots Z^1 \log p_t(x,y) &=  \lp-\frac{1}{2t}\rp^N \sum_i c_i \prod_j    \E^{m_t}\lb \prod_{k\in B(i,j)} Z^k d^2(z,y) \rb
 +O\lp \lp\frac{1}{2t}\rp^{N-1} \rp \\
&=   \lp-\frac{1}{t}\rp^N \sum_i c_i \prod_j    \E^{m_t}\lb \prod_{k\in B(i,j)} d(z,y) Z^k d(z,y) \rb
 +O\lp \lp\frac{1}{2t}\rp^{N-1} \rp .
\end{split}\]
Then the theorem follows after noting that the coefficient of $(-1/t)^N$ in this expression is exactly the formula for the joint cumulant of $d(z,y) Z^1 d(z,y), \ldots, d(z,y) Z^N d(z,y)$ in terms of their joint (raw) moments.
\end{proof}

Recall that the (first) cumulant of a single random variable is its expectation, while the cumulant of two random variables is their covariance (that is, $\kappa(X,Y)=\Cov(X,Y)$), so that this generalizes the above results for $N=1,2$.

\begin{remark}
Following up on Theorem \ref{THM:CleanAsymptotics}, we note that we have previously shown, in Theorem \ref{THM:TrueLocal},  that $\lim_{t\searrow 0} t \log p_t(x,y) = -\frac{1}{4}d^2(x,y)$ holds uniformly on any localizable compact. And away from the cut locus, this limit commutes with derivatives on the $y$-variable. However, this is no longer the case when $y$ is in the cut locus of $x$; see Theorem 11.8 of \cite{ABB_2018} which shows that $d^2(x,y)$ is smooth at $y$ if and only if there is a unique length-minimizing curve from $x$ to $y$ which is strictly normal and non-conjugate. So the asymptotic behavior of the log-derivatives of $ t \log p_t(x,y)$ at the cut locus is more complicated, and reflects the geometry of the minimizing geodesics between $x$ and $y$. While Theorem \ref{THM:Cumulants} addresses the (potential) leading term in the expansion, in principle one could consider \eqref{Eqn:1stLogDer} and \eqref{Eqn:2ndLogDer}, or the analogues for higher-order derivatives, and use the expansion of $\Sigmax_{t/2}$ to try to understand further terms in an asymptotic expansion.
\end{remark}

To continue, we need to better understand the asymptotic behavior of $m_t$.

\begin{theorem}\label{THM:Tightness}
Let $x$ and $y$ be localizable and such that all minimal geodesics from $x$ to $y$ are strictly normal, and let $\{m_t : t\in(0,1]\}$ be the family of probability measures defined by \eqref{Eqn:M_tDef}. Then this family is precompact in the topology of weak convergence, and in particular, for any sequence of times $t_n\rightarrow 0$, there is a subsequence $t_{n(i)}$ such that $m_{t_{n(i)}}$ converges weakly to a probability measure $m_0$ supported on $\Gamma$.
\end{theorem}

\begin{proof}
By definition, the $m_t$ are supported on $\bar{\Gamma}_{\eps}$, which is compact (and which is thus a compact, separable metric space, when equipped with the metric inherited from $M$), so $\{m_t : t\in(0,1]\}$ is tight. Thus the pre-compactness of the $m_t$ (and resulting sequential compactness) follows from Prokhorov's theorem. Now let $U_n$ be the subset of $\bar{\Gamma}_{\eps}$ consisting of points $x$ such that $d(\Gamma,x)>1/n$. It's clear from Laplace asymptotics and \eqref{Eqn:M_tLim}  that, for any $n$, $m_t(U_n)\rightarrow 0$ as $t\rightarrow 0$ (indeed, this is implicit in the fact that all of the heat kernel asymptotics we've been considering are valid with respect to $\Gamma_{\eps}$ for any sufficiently small $\eps$). Since $U_n$ is open as a subset of $\bar{\Gamma}_{\eps}$, the portmanteau theorem implies that $m_0(U_n)=0$ for any limiting measure $m_0$. Since $n$ is arbitrary, this shows that $m_0(\Gamma)=1$ for any limiting measure $m_0$.
\end{proof}

In the real-analytic case, one can say more, including that $m_t$ converges. However, such results are most naturally discussed in connection with the bridge process, and we refer the reader to Theorem \ref{THM:RA-LLN} (proven in Section \ref{Sect:RA-Methods}) for the convergence, and Sections \ref{Sect:A_n-LLN} and \ref{Sect:MB-LLN} for the determination of $m_0$ in the $A$-type and Morse-Bott cases.

Observe that if $z\in\Gamma$, we have $d(z,y)=\frac{1}{2}d(x,y)$.  Also, both $d(\cdot,y)$ and $Z_y d(\cdot,y)$, for any smooth vector field $Z$ near $y$, are continuous, bounded functions on $\Gamma_{\eps}$. Thus, since the cumulant can be written as a polynomial in products of such functions, if $m_0$ is a limiting measure and $t_n\searrow 0$ is a sequence of times corresponding to this $m_0$, we have
\begin{equation}\begin{split} \label{Eqn:Moments}
\lim_{n\rightarrow \infty} t_n Z_y \log p_{t_n}(x,y) &= -\frac{1}{2}d(x,y) \E^{m_0}\lb Z_yd(\cdot,y) \rb , \\
\lim_{n\rightarrow \infty} t^2_n Z'_yZ_y \log p_{t_n}(x,y) &=
\frac{d^2(x,y)}{4}\Cov^{m_0} \Big(Z_yd(\cdot,y), Z'_yd(\cdot,y) \Big) , \\
\text{and}\quad \lim_{n\rightarrow \infty} t^N_n Z_y^N \cdots Z^1_y \log p_{t_n}(x,y) &=
\lp -\frac{d(x,y)}{2} \rp^N  \kappa^{m_0} \Big(Z^1_yd(\cdot,y),\ldots,  Z^N_yd(\cdot,y) \Big) .
\end{split}
\end{equation}
Of course, if $m_0$ is a point mass, which is always the case if $y\not\in\Cut(x)$ and is possible also when $y\in \Cut(x)$, then all of the cumulants after the first (the expectation) are zero. Nonetheless, the rate at which the variance goes to zero distinguishes the cut locus, as we now discuss.

\subsection{Characterizing the cut locus}\label{Sect:CutChar}

We know that if $x$ and $y$ are not cut, $t Z'_yZ_y \log p_t(x,y)$ converges as $t\searrow 0$. Our goal here is to prove that, conversely, if $x$ and $y$ are in the ``non-abnormal'' cut locus, then for any sequence of times going to zero, there is a subsequence $t_n$ and a vector $Z\in T_yM$ such that $t_n Z^2_y \log p_{t_n}(x,y)$ blows up at rate at least $t^{-1/2d}$. We begin with two preliminary lemmas.

\begin{lemma}\label{Lem:Diffeo}
Let $x$ and $y$ be localizable and such that all minimal geodesics from $x$ to $y$ are strictly normal. Then the map $\Gamma\rightarrow T^*_yM$ that takes $z\in \Gamma$ to $d_y d(z,y)$ (that is, the differential of $d(z,\cdot)$ at $y$) is a diffeomorphism onto its image.
\end{lemma}
\begin{proof}
Recall that normal geodesics are given as the projections of curves in $T^*M$ under the Hamiltonian flow. In particular, let $e^{sH}:T^*M\rightarrow T^*M$ denote the time $s$ Hamiltonian flow. Recall that $d(z,y)$ is constant for $z\in\Gamma$, so we can write this distance as $d(\Gamma,y)$. Since no point in $\Gamma$ is in $\Cut(y)$, for each $z\in\Gamma$ there is a unique $\lambda_z\in T^*_yM$ such that (the projection of) $e^{sH}\lambda_z$ for $s\in[0,d(\Gamma,y)]$ is the (unique) unit-speed minimal geodesic from $y$ to $z$, and moreover, the map taking $z\in\Gamma$ to $\lambda_z$ is a diffeomorphism onto its image. Since the Hamiltonian flow is reversible, for $z\in\Gamma$, the unique unit-speed minimal geodesic from $z$ to $y$ has terminal co-vector $-\lambda_z$. It follows from  \cite[Corollary 8.43]{ABB_2018} that for $z\in\Gamma$ we have $d_y d(z,y)=-\lambda_z$. This proves the lemma.
\end{proof}

\begin{lemma}\label{Lem:PointMass}
Let $x$ and $y$ be localizable and such that $y\in\Cut(x)$ and all minimal geodesics from $x$ to $y$ are strongly normal. Suppose there is a sequence of times $t_n\searrow 0$ such that $m_{t_n}$, converges to a point mass $m_0=\delta_{z_0}$, for some $z_0\in\Gamma$. Then $x$ and $y$ are conjugate along $\gamma_{z_0}$.
\end{lemma}
\begin{proof}
Recalling that $\Gamma$ parametrizes the minimal geodesics from $x$ to $y$, \cite[Theorem 8.72]{ABB_2018} implies that if $z_0$ in the only point in $\Gamma$, then $\gamma_{z_0}$ is conjugate. Thus, we are left with the situation when there is at least one other point, which we denote $w_0$, in $\Gamma$. Further, it is enough to show that if $\gamma_{z_0}$ is not conjugate, then there is no sequence $t_n$ such that $m_{t_n}$ converges to $\delta_{z_0}$. So assume that $\gamma_{z_0}$ is not conjugate. Then there exist coordinates $z_1,\ldots,z_d$ defined on a neighborhood $U$ of $z_0$ such that $h(z)=\sum_{i=1}^dz_i^2$ on $U$. Also, there exist coordinates $w_1,\ldots,w_d$ defined on a neighborhood $V$ of $w_0$ such that $h(w)\leq\sum_{i=1}^d w_i^2$ on $V$, and we can assume that $U$ and $V$ are disjoint. Then we have that
\[
\frac{m_t(V)}{m_t(U)} \geq \frac{\int_V \lp c_0(x,w)c_0(w,y)+O(t) \rp\frac{\diff\mu}{\diff w}(w) e^{-\sum_{i=1}^d w_i^2/t} \, \diff w}{\int_U \lp c_0(x,z)c_0(z,y)+O(t) \rp\frac{\diff\mu}{\diff z}(z) e^{-\sum_{i=1}^d z_i^2/t} \, \diff z}
\]
and it follows from the basic Laplace asymptotics of \cite{EAndK} that the right-hand side is bounded from below by a positive constant as $t\rightarrow 0$. It follows that there is no sequence $t_n$ such that $m_{t_n}$ converges to $\delta_{z_0}$, as desired.
\end{proof}

We can now establish the basic estimate for the variance in \eqref{Eqn:Moments} on the cut locus.

\begin{theorem}\label{THM:VarRate}
Let $x$ and $y$ be localizable and such that all minimal geodesics from $x$ to $y$ are strongly normal, and $y\in \Cut(x)$. For any sequence of times going to 0, there is a subsequence $t_n$ and a vector $Z\in T_yM$ such that, for any smooth extension of $Z$ to a neighborhood of $y$,
\[
\liminf_{n\rightarrow \infty} t_n^{1-\frac{1}{2d}} \lb t_n Z_y Z_y \log p_{t_n}(x,y) \rb >0 ,
\]
and the value on the left-hand side depends only on $Z$ and not on the choice of extension.
\end{theorem}

\begin{proof}
By Theorem \ref{THM:Tightness}, we know that for any sequence of times, after perhaps passing to a subsequence, the family $m_{t_n}$ converges to a limiting probability measure, which we denote by $m_0$, supported on $\Gamma$. Then, in light of \eqref{E:First2moments}, in order to prove Theorem \ref{THM:VarRate}, it is sufficient to show that there is some $Z\in T_yM$ such that
\begin{equation}\label{Eqn:BlowUp}
\liminf_{n\rightarrow \infty} t_n^{-\frac{1}{2d}} \Var^{m_{t_n}}\lp Z_y d(\cdot,y) \rp >0
\end{equation}
(noticing also that the quantity on the left only depends on $Z$, and not the extension).
There are two cases, depending on whether or not $m_0$ is a point mass, which we now treat. Further, in order to simplify the notation, we will simply write $m_t$ and $t\rightarrow 0$ in place of $m_{t_n}$ and $n\rightarrow \infty$, with the understanding that we always let $t$ go to zero along an appropriate sequence of times, corresponding to $m_0$.

Suppose $m_0$ is not a point mass (that is, it is not deterministic). Then \ref{Lem:Diffeo} implies that the pushforward under the map $z\mapsto d_y d(z,y)\in T^*_yM$ is also not a point mass. Thus, because of the perfect pairing between $T_yM$ and $T^*_yM$, there exists some $Z\in T_yM$ such that the random variable $Z_y d(\cdot,y)$ is, under $m_0$, not a.s./ constant, and thus, for this sequence of times and this $Z$, 
\[
\liminf_{n\rightarrow \infty} \Var^{m_{t_n}}\lp Z_y d(\cdot,y) \rp >0 ,
\]
which certainly implies \eqref{Eqn:BlowUp}.

The more interesting case is when $m_0$ is a point mass, which we now assume. In particular, we let $z_0\in \Gamma$ be such that $m_0=\delta_{z_0}$. By Lemma \ref{Lem:PointMass}, we know that the minimal geodesic from $x$ to $y$ through $z_0$ is conjugate.  Thus,  by \cite{BBCN-IMRN}, we can choose coordinates $(z_1,\ldots,z_d)$ around $z_0$ so that $h(z)=h_{x,y}(z_1,\ldots,z_d)\leq z_1^4+\sum_{i=2}^d z_i^2$ on $U$, where  $U\in\R^d$ is a neighborhood of the origin contained in (the image of) $\Gamma_{\eps}$. If we let $u(z) = c_0(x,z)c_0(z,y)\frac{\diff\mu}{\diff z}(z)$ on $U$, then $u$ is a smooth, positive function on $\overline{U}$, so that it is bounded above and below by positive constants, say $C$ and $1/C$ for some $C>0$, and we have that 
\[\begin{aligned}
\phi_t(z) &= \frac{1}{\zeta(t)}\Ind_{U}(z) u(z) e^{-h(z)/t} \,\diff z \\
\text{and}\quad \zeta(t) &= \int_U u(z) e^{-h(z)/t} \,\diff z 
\end{aligned}\]
is a family of probability densities (for $t>0$) on $\R^d$ supported on $U$. Let $\tilde{m}_t$ be the probability measures determined by the densities $\phi_t(z)$ (and note that $\tilde{m}_t$ is $m_t$ conditioned to be in $U$).

We now show that we can restrict our attention to $\tilde{m}_t$. For fixed $Z$, for ease of notation, we temporarily let $f=Zd(\cdot,y)$ and $\alpha=\E^{m_t}\lb Zd(\cdot,y) \rb$. Then we have
\[\begin{split}
\Var^{m_t}\lp f \rp &= \E^{m_t}\lb \lp f -\alpha \rp^2\rb \\ 
&\geq \E^{m_t}\lb \Ind_U \lp f -\alpha \rp^2\rb \\ 
&= m_t(U)\E^{\tilde{m}_t}\lb \lp f -\alpha \rp^2\rb \\ 
& \geq m_t(U) \Var^{\tilde{m}_t}\lp f \rp ,
\end{split}\]
where we've used that $(f-\alpha)^2$ is non-negative and the variance of a random variable is the best $L^2$-approximation by a constant. Since $m_t(U)\rightarrow 1$ by assumption, it is enough to show that  \eqref{Eqn:BlowUp} holds for $\tilde{m}_t$, rather than for $m_t$ itself.

We now recall some basic facts about entropy. If $\varphi$ is a probability density function on $\R^d$, we let
\[
H(\varphi) = -\int_{\R^d} \varphi(x)\log\varphi(x)\, \diff x = \E^{\varphi}\lb -\log\varphi\rb
\]
be the (differential) entropy. Further, let $Q(\varphi)$ be the covariance matrix of $\varphi$ (or equivalently, the covariance of the identity function under the probability measure on $\R^d$ determined by $\varphi$). Then we have the entropy inequality
\[\begin{split}
H(\varphi) &\leq \frac{1}{2}\log\lb (2\pi e)^d \det Q(\varphi)\rb \\
\text{which implies}\quad \det Q(\varphi) &\geq \frac{1}{(2\pi e)^d}e^{2 H(\varphi)} .
\end{split}\]
(See \cite{Rioul} for instance.)

Returning to the case at hand, we have a one-parameter family of probability densities $\phi_t$, given by the above, where we view $U$ as a subset of $\R^d$. We can estimate the entropy of $\phi_t$ by 
\[\begin{split}
H(\phi_t) &=  \E^{\phi_t}\lb -\log\phi_t\rb \\
&= \E^{\phi_t}\lb \frac{h(z)}{t}- \log\lp u(z) \rp \rb + \E^{\phi_t} \lb \log \zeta(t)\rb  \\
&\geq \log \zeta(t) +\log C ,
\end{split}\]
where we've used that $h(z)$ is positive. Next, we have that
\[
\zeta(t) \geq  \int_U u(z) e^{-(z_1^4+z_2^2+\cdots+z_d^2)/t} \,\diff z
\]
and so Laplace integral asymptotics (see \cite{EAndK}) show that there is a positive constant $C'$ such that
\[
\log \zeta(t) \geq \log \lp C' \lp t^{1/4} + \prod_{i=2}^d t^{1/2}\rp \rp=\log \lp C' t^{\frac{d}{2}-\frac{1}{4}}\rp
\]
for all sufficiently small $t$ (so $\log \zeta(t)$ can go to $-\infty$, but at a controlled rate). Using this in the entropy inequality, we find that, for some constant $C''>0$,
\[
\det Q(\phi_t) > C''  t^{d-\frac{1}{2}} .
\]

Now  $\det Q(\phi_t)$ is the product of the $d$ eigenvalues of the covariance of $\phi_t$, and it follows that, for all sufficiently small $t$, there is at least one eigenvalue which is greater than $C''' t^{1-\frac{1}{2d}}$. Since we can choose the corresponding eigenvector to be a unit vector in the $z_1,\ldots,z_d$ coordinates, by compactness, there exists a linear random variable $v=c_1 z_1+\cdots + c_d z_d \in \R^d \simeq T_{z_0}M$ with $c_1^2+\cdots+z_d^2=1$ such that $\liminf_{t\rightarrow 0} t^{1-\frac{1}{2d}}\Var^{\tilde{m}_{t}}(v)>0$. By Lemma \ref{Lem:PointMass}, this implies that there is a vector $Z\in T_yM$ such that
\[
\liminf_{n\rightarrow \infty} t_n^{-\frac{1}{2d}} \Var^{m_{t_n}}\lp Z_y d(\cdot,y) \rp >0 ,
\]
which concludes the proof.
\end{proof}

If we wish to consider sets of vector fields, then because we use Lie derivatives, we need to control the size of their derivatives as well as the size of the vectors themselves (this will be especially true in the next section). With this in mind, we say that a set $\SZ$ of (smooth) vector fields defined on a neighborhood of a compact set $K$ is $C^m$-bounded on $K$ if any $z\in K$ has a neighborhood $U$ such that, for any system of coordinates on $U$, the $C^m$-norm of $Z\in\SZ$, restricted to $U$ and with respect to this system of coordinate, is uniformly bounded over $\SZ$. Note that if the $C^m$-norm of $Z\in\SZ$ restricted $U$ is uniformly bounded with respect to one system of coordinates on $U$, then it is also uniformly bounded with respect to any other system of coordinates on $U$ which extends to a neighborhood of $\overline{U}$. 

We now have the natural context in which to state and prove the characterization of the (non-abnormal) sub-Riemannian cut locus, which was given as Corollary \ref{Cor:CutLocus}.

\begin{proof}[Proof of Corollary \ref{Cor:CutLocus}]
If $y\not\in \Cut(x)$, then the $C^1$-boundedness implies that $Z_yZ_yd^2(x,y)$ is uniformly bounded for $Z\in\SZ$, which in light of \eqref{Eqn:DistMollified}, completes the proof in that case. If $y\in \Cut(x)$, let $t_n$ be any sequence of times going to 0. After possibly passing to a subsequence, the fact that $\SZ|_{T_yM}$ contains a neighborhood of the origin means that there is some $c>0$ (possibly 1) such that $cZ\in\SZ$, where $Z$ is the vector (field) from Theorem \ref{THM:VarRate}. By linearity of differentiation, we see that any sequence of times going to zero has a subsequence $t_n$ such that
\[
\lim_{n\rightarrow\infty} \lb \sup_{Z\in \Z} t_n Z_y Z_y \log p_{t_n}(x,y) \rb =\infty ,
\]
and this gives the desired result.
\end{proof}

Finally, in the Riemannian case, there is no need to avoid abnormal minimizers, and we can work with covariant derivatives.

\begin{corollary}
Let $M$ be a (possibly-incomplete) Riemannian manifold, and $x$ and $y$ any two localizable points in $M$. Then $y\not\in \Cut(x)$ if and only if
\[
\limsup_{t\searrow 0} \lb \sup_{\substack{Z\in T_yM \\ \|Z\|=1}} t\lab \nabla^2_{Z,Z} \log p_t(x,y)\rab \rb <\infty ,
\]
and $y\in \Cut(x)$ if and only if 
\[
\lim_{t\searrow 0} \lb \sup_{\substack{Z\in T_yM \\ \|Z\|=1}} t \nabla^2_{Z,Z} \log p_t(x,y) \rb =\infty ,
\]
where $\nabla^2$ is the (covariant) Hessian, acting on the $y$-variable.
\end{corollary}

\subsection{Sheu-Hsu-Stroock-Turetsky type bounds}

For a compact Riemannian manifold $M$, with the Riemannian volume and Laplace-Beltrami operator, a result of Stroock-Turetsky and Hsu, improving an earlier result of Sheu, is that, for each $N$, there exists a constant $C_N$ depending on $M$ and $N$ such that, for all $(t,x,y)\in (0,1]\times M\times M$,
\begin{equation}\label{Eqn:Sheu}
| \nabla^N \log p_t(x,y) |\leq C_N \lp \frac{d(x,y)}{t}+\frac{1}{\sqrt{t}}\rp ^N
\end{equation}
which then implies that, for each $N$, there exists a constant $D_N$ depending on $M$ and $N$ such that, for all $(t,x,y)\in (0,1]\times M\times M$,
\[
| \nabla^N p_t(x,y) |\leq D_N \lp \frac{d(x,y)}{t}+\frac{1}{\sqrt{t}}\rp ^N p_t(x,y) .
\]
(In both cases, the differentiation is in the $y$-variable.)

Note that the $\frac{1}{\sqrt{t}}$ term is only relevant near the diagonal, since on any set where $d(x,y)$ is bounded from below by a positive constant, the $\frac{d(x,y)}{t}$ term dominates. On a (strictly) sub-Riemannian manifold, the diagonal is abnormal, and uniform bounds even for the heat kernel itself appear not to be generally known.

In light of this, we see that the natural generalization of  to sub-Riemannian manifolds, using the Molchanov approach as above, is the following.

\begin{theorem}\label{THM:SRBounds}
Let $\K\in M^2\setminus \D$ be a compact localizable subset such that  all length minimizers between pairs $(x,y)\in \K$ are strongly normal.
Let $N$ be a positive integer and  let $\SZ$ be a set of vector fields on a neighborhood of $\pi_2\lp \K\rp$ which is $C^{N-1}$-bounded. Then there exist constants $C_N$ and $D_N$, depending on $M$, $\K$, $\SZ$, and $N$, such that, for all $t\in (0,1]$ and $(x,y)\in \K$, and for all $Z^1,\ldots, Z^N \in \SZ$, we have
\[\begin{aligned}
\lab Z^N\cdots Z^1 \log p_t(x,y) \rab &\leq \frac{C_N}{t^N}  \\
\text{and}\qquad  \lab Z^N\cdots Z^1 p_t(x,y) \rab & \leq \frac{D_N}{t^N}p_t(x,y) ,
\end{aligned}\]
where, as usual, the derivatives act on the $y$-variable.
\end{theorem}

Note that, on a set $\K$ as in the theorem, $d(x,y)$  is bounded from above and below by positive constants, and thus, compared to the Riemannian analogues above, $d(x,y)$ does not appear, instead being absorbed into the $C_N$ and $D_N$.

\begin{proof}
As already noted in the proof of Theorem  \ref{THM:Cumulants}, equation \eqref{Eqn:Combos} shows that $Z^N\cdots Z^1 \log p_t(x,y)$ can be expanded in powers of $1/t$, up to the $N$th power, with coefficients given in terms of (products of) the expectations of $\Sigmax_{t/2}(z,y)$ and $d^2(z,y)$ and their first $N$ derivatives (with respect to the $Z^i$). Further, $\K$ is chosen so that, for small enough $\eps$, $\Sigmax_{t/2}(z,y)$ and $d^2(z,y)$ are smooth on a neighborhood of the compact set
$$
\Gamma^r(\K)=
\left\{
(z,y)\mid \exists x\in M \st (x,y)\in \K, z\in \Gamma_\varepsilon(x,y)
\right\}.
$$ 
(See Lemma~\ref{L:fat_compacts}.) 
This, plus the definition of $C^{N-1}$-boundedness, implies that these expectations can be uniformly bounded over $\K$ (recall that for each $(x,y)$, the corresponding probability measure is supported on $\Gamma_{\eps}(x,y)$). This gives the result for $\log p_t(x,y)$. Then the result for $p_t$ itself follows from this and Fa\`a di Bruno's formula for the exponential.
\end{proof}

The bounds on derivatives of $p_t$ itself should be compared to those of Proposition \ref{P:universal_bounds}. Indeed, using the upper bound on $p_t$ from Proposition \ref{P:universal_bounds} in Theorem \ref{THM:SRBounds} implies the spatial derivative bounds in Proposition \ref{P:universal_bounds}. At a single pair $(x,y)$, one should compare with Corollary \ref{C:MolchanovLaplace}.

\section{Law of large numbers}\label{Sect:LLN}

We finally turn to a more stochastic topic, the law of large numbers (LLN) for the bridge process associated to the diffusion $X_t$, which is essentially the ``leading term'' of the small-time asymptotics of the bridge process. We begin with the basic definitions needed to state and prove the results.

Let $\Omega^{[0,t]}_M$ be the space of continuous paths $\omega_{\tau}: [0,t]\rightarrow M$, for some $t\in(0,\infty)$. We define a metric
\[
d_{\Omega^{[0,t]}_M}\lp \omega,\tilde{\omega}\rp = \frac{\sup_{0\leq \tau\leq t}d(\omega_{\tau},\tilde{\omega}_{\tau})}{1+\sup_{0\leq \tau\leq t}d(\omega_{\tau},\tilde{\omega}_{\tau})} 
\]
on $\Omega^{[0,t]}_M$. This metric gives $\Omega^{[0,t]}_M$ the topology of uniform convergence. (Note also that this topology makes $\Omega^{[0,t]}_M$ into a Polish space, though realized with a different, but equivalent, metric if $M$ is incomplete.) We let $\Omega_M= \Omega^{[0,1]}_M$ and note that the map $\Omega^{[0,t]}_M\rightarrow \Omega_M$, $\omega_{\tau}\mapsto \omega_{\tau/t}$ is an isometry.

For $x$ and $y$ in $M$ and $t\in(0,\infty)$, we can consider the bridge process $X_{\tau}^{x,y,t}$, which is the diffusion started from $x$, conditioned to be at $y$ at time $t$. More concretely, this process is determined by its finite-dimensional distributions, given in terms of $p_t$ by
\[
\frac{1}{p_t(x,y)}\lb p_{s_1}(x,z_1)\cdot p_{s_2-s_1}(z_1,z_2)\cdot \ldots \cdot p_{s_k-s_{k-}1}(z_{k-1},z_k)\cdot p_{t-s_k}(z_k,y) \rb 
\]
for any finite collection of times $0<s_1<\cdots<s_k<t$ and points $z_1,\ldots,z_k\in M$. (We see that this distribution is smooth in the $s_i$ and $z_i$, which also implies the strong Markov property for  $X_{\tau}^{x,y,t}$.)
The law of $X_{\tau}^{x,y,t}$ is a probability measure $\hat\mu^{x,y,t}$ on $\Omega^{[0,t]}_M$. We let $\mu^{x,y,t}$ be the pushforward of $\hat{\mu}^{x,y,t}$ to $\Omega_M$, that is, $\mu^{x,y,t}$ is the law of the bridge process rescaled to take unit time. For fixed $x$ and $y$, this gives a family of probability measures on $\Omega_M$, and we are interested in the weak convergence of these measures as $t\searrow 0$. A result determining such convergence is generally called a law of large numbers for the bridge process.

Again for fixed $x$ and $y$, recall that $\Gamma$ parametrizes the set of minimal geodesics from $x$ to $y$. More precisely, for any $z\in \Gamma$, let $\gamma^z_{\tau}$  be the constant speed geodesic going from $x$ to $y$ in unit time, through $z$ (so that $\gamma^z_{1/2}=z$). Then this gives an embedding of $\Gamma$ into $\Omega_M$, and we write the image as $\tilde{\Gamma}$.

\subsection{Extension to the cut locus}
Bailleul and Norris \cite{NorrisB} proved the law of large numbers in the case when there is a single minimizer from $x$ to $y$. In that case, if $z$ is the unique point in $\Gamma$, $\mu^{x,y,t}$ converges (weakly) to a point mass at $\gamma^z$ as $t\searrow 0$. The idea of the Molchanov technique in this context is to determine a law of large numbers on the (non-abnormal) cut locus by conditioning on the midpoint of the bridge and ``gluing together'' the result of  Bailleul and Norris for the first and second halves of the path.

The connection with the previous heat kernel asymptotics is as follows. For each $t$, let $\nu_t$ be the pushforward of $\mu^{x,y,t}$ under the map $\Omega_M\rightarrow M$, $\omega_{\tau} \mapsto \omega_{1/2}$, so that $\nu_t$ is the distribution of the midpoint of the bridge process.

\begin{lemma}\label{Lem:Nu}
Let $x$ and $y$ be points of $M$, satisfying either localization condition, such that all minimizers from $x$ to $y$ are strongly normal, let $\nu_t$ be as above, and let $m_t$ be defined by \eqref{Eqn:M_tDef}. Then for a sequence of times $t_n$ going to 0, $\nu_{t_n}$ converges if and only if $m_{t_n}$ does, in which case they have the same limit.
\end{lemma}
\begin{proof}
In terms of $p_t$, the density of $X_{\tau}^{x,y,t}$ at time $t/2$ is given by
\[
\frac{\diff \nu_t}{\diff\mu}(z) = \frac{p_{t/2}(x,z)p_{t/2}(z,y)}{p_t(x,y)} .
\]
Let $f:M\rightarrow \bR$ be any continuous, bounded function. Then from Theorem \ref{THM:TrueLocal} and Corollary \ref{C:MolchanovV1}, we have
\[\begin{split}
p_t(x,y) &= e^{-\frac{d(x,y)^2+o(1)}{4t}} \\
\text{and}\quad p_t(x,y) &= \int_{\Gamma_\varepsilon} p_{t/2}(x,z) p_{t/2}(z,y)\diff\mu(z)+O\lp e^{-\frac{d(x,y)^2+{\frac{\eps^2}{2}}}{4t}}\rp
\end{split}\]
which imply
\[\begin{split}
 \int_{\Gamma_\varepsilon} p_{t/2}(x,z) p_{t/2}(z,y)\diff\mu(z) &= e^{-\frac{d(x,y)^2+o(1)}{4t}} \\
 \text{and then}\quad \frac{1}{p_t(x,y)} & = \frac{1}{\int_{\Gamma_\varepsilon} p_{t/2}(x,z) p_{t/2}(z,y)\diff\mu(z)}
 \lb 1+ O\lp e^{-\frac{\eps^2/4}{4t}}  \rp\rb .
\end{split}\]
From this and another application of Corollary \ref{C:MolchanovV1}, we see that (recalling that $f$ is bounded)
\[\begin{split}
\E^{\nu_t}\lb f\rb &= \E^{\nu_t}\lb f\Ind_{\Gamma{\eps}}\rb + \E^{\nu_t}\lb f\Ind_{M\setminus \Gamma{\eps}}\rb  \\
 &= \frac{\int_{\Gamma_\varepsilon} f(z) p_{t/2}(x,z) p_{t/2}(z,y)\diff\mu(z)}{\int_{\Gamma_\varepsilon} p_{t/2}(x,z) p_{t/2}(z,y)\diff\mu(z)}\lb 1+ O\lp e^{-\frac{\eps^2/4}{4t}}  \rp \rb +O\lp e^{-\frac{\eps^2/4}{4t}}\rp
\end{split}\]
Comparing with Equation \eqref{Eqn:M_tLim}, we see that $\E^{m_t}\lb f\rb-\E^{\nu_t}\lb f\rb\rightarrow 0$ as $t\rightarrow 0$. But this proves the result.
\end{proof}

Suppose that for some sequence of times $t_n\searrow 0$, $m_{t_n}\rightarrow m_0$ for some $m_0$ supported on $\Gamma$ (recall that $\{m_t : t\in(0,1]\}$ is subsequentially compact). Then $m_0$ maps to a probability measure $\tilde{m}_0$ on $\tilde{\Gamma}$ under the inclusion of $\Gamma$ into $\Omega_M$. (Of course, $m_0$ can be recovered from $\tilde{m}_0$ by the inverse map $\tilde{\Gamma}\rightarrow \Gamma$.)

As noted, we wish to ``glue together'' the LLN on the two halves of paths from $x$ to $y$. Obviously, this requires the LLN for the two halves, which we now give a version of. In particular, if there is a unique minimal geodesic from $x$ to $z$, we let $g^{x,z}=g^{x,z}_{\tau}$ be that unique geodesic traveling from $x$ to $z$ in unit time.

\begin{lemma}\label{Lem:UniformBN}
Let $x$ and $y$ be points of $M$, satisfying either localization condition, such that all minimizers from $x$ to $y$ are strongly normal. Then there exists $\eps>0$ such that $\KR_1=\lc (x,z): z\in \Gamma_{\eps}\rc$ and $\KR_2=\lc (z,y): z\in \Gamma_{\eps}\rc$ are both compact and localizable, and such that there is a unique, strongly normal minimizer from $x$ to $z$ and from $z$ to $y$ for all $z$ in a neighborhood of $\Gamma_{\eps}$. Further, we have that $\mu^{x,z,t}$ converges to the point mass at $g^{x,z}$ and $\mu^{z,y,t}$ converges to the point mass at $g^{z,y}$ as $t\searrow 0$, uniformly over $z\in \Gamma_{\eps}$.
\end{lemma}

Here the uniformity is understood with respect to the L\'evy-Prokhorov metric on probability measures on $\Omega_M$, which metrizes weak convergence. But since the limiting measure is a point mass, this simplifies to saying that for every $\delta$, there exists $t_0>0$ such that
\[
\mu^{x,z,t}\lp \dist_M\lp \omega_{\tau}, g^{x,z}_{\tau}\rp<\delta \text{ for all $\tau\in[0,1]$} \rp >1-\delta
\]
for all $t<t_0$ and all $z \in\Gamma_{\eps}$, and analogously for $\mu^{z,y,t}$ and $g^{z,y}$.

That this holds pointwise for each $z$, under either localization condition, follows from Theorem 1.3 of Bailleul and Norris \cite{NorrisB}, under the assumption that $Z_0$ is in the span of $Z_1,\ldots, Z_k$. The uniformity is then a consequence of the smoothness of the heat kernel and compactness. The result, without this restriction on $Z_0$, also follows from what is essentially a space-time version of an argument from \cite{HsuIncomplete}. The situation is similar, as are the techniques, to that of Theorem \ref{THM:StrongCond}, and we again relegate a brief proof to Appendix \ref{App:Local}.

We can now prove our law of large numbers for the bridge process, as given in Theorem \ref{THM:LLN}.

\begin{proof}[Proof of Theorem \ref{THM:LLN}]

Assume that $m_{t_n}$ converges to $m_0$. For simplicity of notation, we will assume that $m_t\rightarrow m_0$, with the general case following by restricting to a subsequence. We know that, under $\mu^{x,y,t}$, $\omega|_{[0,t/2]}$ and $\omega|_{[t/2,t]}$ are conditionally independent given $\omega_{t/2}$, by the Markov property. Thus, we can decompose (or disintegrate) $\mu^{x,y,t}$ by first drawing $z$ from $\nu_t$ and then drawing $\omega|_{[0,t/2]}$ and $\omega|_{[t/2,t]}$ independently from $\mu^{x,z,t/2}$ and $\mu^{z,y,t/2}$, respectively.

Let $F:\Omega_M\rightarrow \bR$ be Lipschitz continuous and bounded. Let $f:M\rightarrow \bR$ be as follows. For any $z\in \Gamma_{\eps}$, let $\gamma^z$ be the (possibly) broken geodesic which travels the minimal geodesic from $x$ to $z$ at constant speed in time $1/2$, and then travels the minimal geodesic from $z$ to $y$ at constant speed in time $1/2$. Note that this is well defined and agrees with our earlier definition when $z\in\Gamma$. On $\Gamma_{\eps}$, let $f(z)=F(\gamma^z)$; since this is bounded and continuous (recall that $\gamma^z$ is continuous in $z$ by the smoothness of the exponential map), we can extend it to a bounded and continuous function on $M$ in an arbitrary way. Now choose $\delta>0$. By weak convergence of $m_t$ and Lemma \ref{Lem:Nu}, for all small enough $t$, we have
\[
\lab\E^{\nu_t}\lb f\rb - \E^{m_0}\lb f\rb \rab < \delta .
\]
Next, by Lemma \ref{Lem:UniformBN}, applied to the ``two halves'' of $\gamma^z$ (that is, for $\mu^{x,z,t/2}$ and $\mu^{z,y,t/2}$), and the fact that $F$ is Lipschitz and bounded, we have that for all small enough $t$,
\[
\lab \E^{\mu^{x,z,t/2} \otimes \mu^{z,y,t/2}}\lb F \rb- f(z)\rab <\delta
\]
 uniformly for all $z\in \Gamma_{\eps}$. But in light of the above decomposition of $\mu^{x,y,t}$, these inequalities imply that, for all sufficiently small $t$,
 \[
\lab \E^{\mu^{x,y,t}}\lb F \rb- \E^{m_0}\lb f\rb\rab <2\delta .
 \]
 Recalling the definition of $f$ and that $\delta$ is arbitrary, it follows from the portmanteau theorem that $\mu^{x,y,t}\rightarrow \tilde{m}_0$.
 
For the other direction, assume that $\mu^{x,y,t_n}$ converges to some $\mu_0$. Then $\nu_{t_n}$ converges to the pushforward of $\mu_0$ under the evaluation at $\tau=1/2$, and in particular, $m_{t_n}$ also converges, by Lemma \ref{Lem:Nu}. Then applying the part of the theorem (just proven) when $m_{t_n}$ converges to some $m_0$ shows that $\mu_0=\tilde{m}_0$. This completes the proof.
\end{proof}

\subsection{Real analytic methods}\label{Sect:RA-Methods}
We can give more precise results in the real-analytic case. By this, we mean the case when $h_{x,y}$ has a real-analytic normal form in a neighborhood of any point of $\Gamma$. This includes the case when $M$ (and its sub-Riemannian structure) are real-analytic, as we show below, but does not require this. For example, if a minimizing geodesic $\gamma$ from $x$ to $y$ is non-conjugate, then we have already seen (essentially via the Morse lemma) that if $z\in\Gamma$ is the midpoint of $\gamma$, then there are coordinates $(u_1,\ldots,u_d)$ in a neighborhood of $z$ such that $h_{x,y}=d^2(x,y)/4 +\sum_{i=1}^d u_i^2$ on this neighborhood. Certainly, $h_{x,y}$ is real-analytic in this coordinate system, without any assumptions that the sub-Riemannian structure itself is real-analytic.

Our approach is  essentially a direct application of results from Laplace asymptotics to the behavior of $m_t$. Hsu \cite{HsuBridge} already gave this application for Brownian motion on real-analytic Riemannian manifolds (after having established Theorem \ref{THM:LLN} in the Riemannian case via large deviations), so we simply summarize his results, for completeness and to emphasize that they hold in the present sub-Riemannian context as well.

Suppose that every $z\in\Gamma$ is contained in a coordinate patch such that $h_{x,y}$ is real-analytic (in the coordinates--- that is, there exists a local real-analytic stiffening of the structure with this property). Then for any $z\in \Gamma$, there is a rational $\alpha(z)\in \lb d/2,d-(1/2)\rb$, a non-negative integer $\beta(z)$, and an $r_0>0$ such that, for any open ball $B(z,r)$ around $z$ with radius $r\in (0,r_0)$, we have
\begin{equation}\label{Eqn:RA-Leading}
\int_{B(z,r)}e^{-\frac{h_{x,y}(u)-h_{x,y}(z)}{4t}}\diff \mu(u) \sim \frac{C}{t^{\alpha}}\lp \log \frac{1}{t}\rp^{\beta}
\end{equation}
where $C$ is some positive constant depending on $z$ and $r$. If we put the lexigraphical order on $\mathbb{Q}\times \mathbb{Z}$, so that $(\alpha_1,\beta_1)< (\alpha_2,\beta_2)$ if either $\alpha_1<\alpha_2$ or $\alpha_1=\alpha_2$ and $\beta_1<\beta_2$, then we see that $(\alpha(z_1),\beta(z_1))< (\alpha(z_2),\beta(z_2))$ means that the integral \eqref{Eqn:RA-Leading} around $z_2$ dominates the integral around $z_1$ as $t\searrow 0$. Moreover, the resulting map $\Gamma \rightarrow \mathbb{Q} \times \mathbb{Z}$ is upper semi-continuous, and since $\Gamma$ is compact, this means that $(\alpha(z),\beta(z))$ attains its maximum, which we denote $(\alpha_m,\beta_m)$. We let
\[
\Gamma^m_{x,y}=\Gamma^m = \lc z\in \Gamma: (\alpha(z),\beta(z))=(\alpha_m,\beta_m) \rc
\]
and note that $\Gamma^m$ is a non-empty, closed subset of $\Gamma$ (corresponding to geodesics of ``maximal degeneracy'').

The significance of these considerations is given by Theorem \ref{THM:RA-LLN}, and we indicate its proof.

\begin{proof}[Proof of Theorem \ref{THM:RA-LLN}]
The convergence of $m_t$ to a measure with support $\Gamma^m$ follows from the definition of $m_t$ and the expansions \eqref{Eqn:RA-Leading}; the details are given in the proofs of Theorems 4.1 and 4.2 in \cite{HsuBridge}. Once we have that, the limit of the log-derivatives of $p_t(x,y)$ was already derived in \eqref{Eqn:Moments}.

It only remains to justify that $h_{x,y}$ is real-analytic in a neighborhood of any $z\in \Gamma$ when $M$ and $\Delta$ are themselves real-analytic. So assume that $M$ and $\Delta$ are real-analytic and $z\in\Gamma$. We already know that $d(x,\cdot)$ and $d(\cdot, y)$ are smooth in a neighborhood $U$ of $z$. Then Corollary 1 of \cite{Agrachev-RA} says that $d(x,\cdot)$ and $d(\cdot, y)$ are in fact real-analytic on $U$. Then it is immediate from the definition that that $h_{x,y}$ is also real-analytic on $U$.
\end{proof}

Finally, we can identify $\Gamma^m$ and $m_0$ more explicitly if we have more information on the normal form of $h_{x,y}$. We illustrate this with the two most important special cases.

%%%%%%%%%%%%%%%%%%%%%%%%%%%%%%%%%%%%%%%%%%%%%%

\subsection{LLN for $A$-type singularities}\label{Sect:A_n-LLN}
In parallel to Section \ref{S:Complete expansions}, we give an explicit treatment of the asymptotics in two cases-- the case when each minimal geodesics is $A_n$-conjugate (in this section) and the Morse-Bott case (in the next).

As usual, let $x$ and $y$ be distinct points such that every minimal geodesic from $x$ to $y$ is strongly normal. Further, we assume that there is some $\ell\in\{1,3,5,\ldots\}$ such that for every $z\in\Gamma$, $\gamma^z$ is $A_m$-conjugate for $1\leq m\leq \ell$ and that there is at least one $z\in\Gamma$ for which $\gamma_z$ is $A_{\ell}$-conjugate. We refer to Section \ref{Sect:A_n} for the relevant results about the normal form of $h_{x,y}$ and the resulting leading term in the expansion coming from each geodesic. In particular, let $z_1,\ldots,z_N$ be the points of $\Gamma$ corresponding to $A_{\ell}$-conjugate geodesics, and around each $z_i$, let $(u_{i,1},\ldots, u_{i,d})$ be local coordinates diagonalizing $h_{x,y}$ as in Equation \eqref{Eqn:hDiag}. Then $m_t$ converges, and the limit is given by
\[
m_0=  \frac{\sum_{i=1}^N c_0(x,z_i)c_0(z_i,y)\frac{\diff\mu}{\diff (u_{i,1},\ldots, u_{i,d})}(z_i) \cdot\delta_{z_i}}{\sum_{i=1}^N c_0(x,z_i)c_0(z_i,y)\frac{\diff\mu}{\diff (u_{i,1},\ldots, u_{i,d})}(z_i)}  .
\]
(To see this, integrate any smooth $f$ against $m_t$ and take the leading term of the resulting Laplace asymptotics.) In particular, $\Gamma^m=\{z_1,\ldots,z_N\}$, and it may certainly be a proper subset of $\Gamma$.

However, note that if none of the minimal geodesics from $x$ to $y$ is conjugate (which in this terminology means being ``$A_1$-conjugate'' and implies that $h_{x,y}-d^2(x,y)/4$ can be written as a sum-of-squares around each $z_i$), then $\Gamma^m=\Gamma$. In particular, if $M$ is a Riemannian manifold of non-positive sectional curvature, this is the only possibility.  Indeed, in this case, the asymptotics of $p_t(x,y)$ can be written directly in terms of the Ben Arous expansion applied to the universal cover, and a slightly simpler formula for $m_0$ can be given; see Example 3.7 of \cite{HsuBridge} where the case of only non-conjugate geodesics is treated for a compact Riemannian manifold.

\subsection{LLN for the Morse-Bott case}\label{Sect:MB-LLN}
Again, we let $x$ and $y$ be distinct points such that every minimal geodesic from $x$ to $y$ is strongly normal, but now, as in Section \ref{Sect:MB}, we assume that $\Gamma$ is an $r$-dimensional submanifold (where necessarily we have $r<k$ and we recall that $\Gamma$ is compact) and that the kernel of the differential of the exponential map has dimension $r$ at $\gamma^z$ for any $z\in\Gamma$. Then around any point of $\Gamma$ we can find local coordinates $(u_1,\ldots,u_k)$ such that $\Gamma$ is (locally) given by $u_{r+1}=\cdots=u_k=0$, $(u_1,\ldots,u_r)$ gives (local) coordinates on $\Gamma$, and 
\[
h_{x,y}=\frac{d^2(x,y)}{4}+u_{r+1}^2+\cdots+u_d^2 .
\]

In this case, another use of Laplace asymptotics to integrate any smooth $f$ against $m_t$ shows that $m_t$ converges and $m_0$ has a smooth, non-vanishing density on $\Gamma$ with respect to any local coordinates. Hence $\Gamma^m=\Gamma$. Moreover, the density of $m_0$ with respect to $du_1\ldots du_r$ as above can be written in terms of the density of $\mu$, the Hessian of $h_{x,y}$ along the normal bundle over $\Gamma$, and the $c_0$ (see Section 3 of \cite{BBN-BiHeis} for the basic framework of the computation), but the expression is messy and unenlightening, so we omit it. Instead, we note that the Morse-Bott case typically arises when $M$ possess some rotational symmetry, in which case $m_0$ can be deduced via symmetry arguments. That is, let $\Iso_{x,y}$ be the subgroup of the isometry group of $M$ that fixes $x$ and $y$, where isometries must also preserve $\mu$ and the sub-Laplacian. Suppose that $\Iso_{x,y}$ acts transitively on $\Gamma$. Then $m_0$ must be the uniform probability measure on $\Gamma$, in the sense that $m_0$ is the unique probability measure on $\Gamma$ invariant under the action of $\Iso_{x,y}$.

For example, if $M$ is the standard Riemannian sphere with the Laplace-Beltrami operator and Riemannian volume, and $x$ and $y$ are antipodal points, $m_0$ is the uniform probability measure on the equator, as observed in Example 3.6 of \cite{HsuBridge}.

The natural sub-Riemannian analogue is the Heisenberg group. By symmetry, we can take $x$ to be the origin (using $\bR^3$ to give global coordinates, in the usual way). Then $y$ is in the (non-abnormal) cut locus exactly when $y=(0,0,h)$ for some $h\neq 0$, in which case $\Gamma$ is a circle, invariant under rotation around the vertical axis (see Figure \ref{F:Pansu_sphere} again). We see that $m_0$ is the uniform probability measure on $\Gamma$.

\appendix

\section{Strong localization and pathspace concentration}\label{App:Local}

Our goal in this section is to provide a (fairly) brief proof of Theorem \ref{THM:StrongCond} and Lemma \ref{Lem:UniformBN}. The common theme of both proofs is that the process ``pays a cost of $e^{-d^2/4t}$'' to move a distance $d$ in time $t$, uniformly on compacts.

We recall that Theorem \ref{THM:StrongCond} contains two related assertions. First, there is the localization estimate, namely that if $A\subset M$ is closed such that $M\setminus A$ has compact closure, then for any compact subset $K$ of $M\setminus A$
\[
\limsup_{t\searrow 0} 4t \log p_t(x,A,y) \leq - \lp d(x,A)+d(y,A)\rp^2 
\]
uniformly for $x\in K$ and $y\in K$. Second, there is the Varadhan asymptotics, namely that if $\KR$ is a compact subset of $\{(x,y)\in M\times M: d(x,y)< d(x,\infty) +d(y,\infty)\}$, then we have $4t\log p_t(x,y)\rightarrow -d^2(x,y)$ uniformly for $(x,y)\in \KR$.

%%%%%%%%%%%%%%%%%%%%%%%%%%%%%%%%%%%%%%%%%%%%%%%%%%%%%%%%%%

Unsurprisingly, the argument uses many of the same ideas as in Section \ref{Sect:Local}. We note that while Theorem \ref{THM:StrongCond} is stated in Section \ref{Sect:Local1}, it is not used in any of the other material in Sections \ref{Sect:Local1} and \ref{Sect:Local2}. Thus we are free to use the results of those two sections without risk of circular reasoning.

 The first step is a more precise bound on $\sigma_a$, which we recall is the first time the process travels a distance $a$ from its starting point.

\begin{lemma}\label{App:ExitTime}
Let $K \subset M$ be compact. Then there exists $\rho>0$ such that, for any $a\in(0,\rho)$,
\[
\limsup_{t\searrow0} 4t\log\lp\Prob^x\lp \sigma_a \leq t\rp\rp \leq -a^2 , 
\]
uniformly over $x\in K$.
\end{lemma}

\begin{proof}
First note that there exists $\rho>0$ such that the $2\rho$-neighborhood of $K$ has compact closure. Let $K_{\rho}$ denote the closure of the $\rho$-neighborhood of $K$. Then we see that for any $\eps\in(0,\rho)$, there is $T>0$ such that
\begin{equation}\label{Eqn:App-HeatKernel}
\int_{B_{\eps}(z)} p_t(z,y) \mu(dy) >1/2 \quad \text{for any $z\in K_{\rho}$ and any $t\in(0,T)$.}
\end{equation}
This follows directly from Lemma \ref{Lem:SigA} (with $a=\eps$) and the monotonicity of the integral with respect to $t$ coming from fact that $\{ \sigma_{\eps}<t \}\subset\{\sigma_{\eps}<T  \}$.

Now choose $x\in K$ and $a<\rho$, and consider the heat kernel at time $t$ on the annulus $A_{\eps}(x;a)=\{\dist(z,x)\in [a-\eps,a+\eps]\}$ for some $\eps\in(0,a)$ and $t\leq T$ as above. If we consider the diffusion $X_t$ started from $x$, then by the strong Markov property and the estimate \eqref{Eqn:App-HeatKernel}, we have
\[
\Prob^x\lp X_t\in A_{\eps}(x;a)\rp = \int_{A_{\eps}(x;a)} p_t(x,z) \, d\mu(z) > \frac{1}{2}\Prob^x\lp \sigma_{a} \leq t\rp .
\]
This integral can be estimated uniformly in $x$ by the L\'eandre estimate on the heat kernel. More precisely, for $\delta>0$, after possibly making $T$ smaller,
\[
\int_{A_{\eps}(x;a)} p_t(x,z) \, d\mu(z) \leq \mu\lp A_{\eps}(x;a)\rp \exp\lb -\frac{(a-\eps)^2-\delta}{4t}\rb
\]
 for any $t\leq t_0$ and for any $x\in K$. Since the measure of $A_{\eps}(x;a)$ is bounded from below, uniformly in $x$ (by smoothness and compactness), we conclude that
\[
\Prob^x\lp \sigma_{a}\leq t\rp \leq C \exp\lb -\frac{(a-\eps)^2-\delta}{4t}\rb
\]
for some $C>0$ independent of $x$, for all $t\leq T$. Since $\eps$ and $\delta$ are (small and) arbitrary, standard algebraic manipulations then give
\[
\limsup_{t\searrow0} 4t\log\lp\Prob^x\lp   \sigma_{x,a}     \leq t\rp\rp \leq -\rho^2 , 
\]
uniformly for $x\in K$.
\end{proof}

Let $A$ and $K$ be as in the theorem, and to simplify notation, let $U=M\setminus A$. Then $\overline{U}$ is compact. Further, we can find $\eps>0$ such that the closed $\eps$-neighborhood of $A$, has complement $U'$, such that $U'$ is open with compact closure. By taking $\eps$ small enough, we have that $K\subset U'$, so that
\[\begin{split}
& K\subset U' \subset \overline{U'}\subset U= M\setminus A \\
& \quad \text{with $\dist\lp K,\lp U'\rp^c \rp > \eps$ and $\dist\lp \overline{U'},A \rp =\eps$.}
\end{split}\]

Considering $X_t$ started from $x\in K$, we let $\tau$ be the first hitting time of $A$. This is motivated by the fact that $p_t(x,A,y) = \Prob^x\lp X_t\in dy \text{ and } \tau<t\rp$. 

For clarity in the following arguments, we note that the asymptotic relation $\limsup_{t\searrow0} 4t\log f(t)=-d^2$, for some $d\in\bR$, means that
\begin{equation}\label{Eqn:Leandre-App}
f(t) = \exp\lb -\frac{d^2+o(1)}{4t}\rb 
\end{equation}
where $o(1)$, as usual, denotes some function that goes to 0 with $t$. If $f(t)$ and $d$ are also functions of $x$ in some $S\subset M$, then this asymptotic relation is uniform in $x$ if the $o(1)$ goes to zero uniformly for all $x\in S$. The corresponding notion for inequalities or when $f(t)$ depends on $(x,y)\in M\times M$ are obvious modifications.

When taking convolutions of functions satisfying \eqref{Eqn:Leandre-App}, the following will be useful (the proof is an exercise in calculus, so we omit it). 
\begin{lemma}\label{Lem:Calc}
For some $n$, consider the function $\sum_{i=1}^n \frac{d_i^2}{4t_i}$ on the ``double simplex'' $d_i\geq 0$, $\sum d_i = D$ and $t_i\geq 0$, $\sum t_i =T$. The minimum of this function is $\frac{D^2}{4T}$, achieved on the the set
\[
\lc \frac{d_i}{t_i}=\frac{d_j}{t_j} \text{ for all $1\leq i,j\leq n$}\rc .
\]
In particular, if we fix the $t_i$, there is a unique choice of the $d_i$ minimizing this function, and vice versa. Also, note the minimum does not depend on $n$.
\end{lemma}

Next, we show how to extend this type of estimate to $\tau$.

\begin{lemma}\label{Lem:SigmaBound}
We have
\[
\limsup_{t\searrow0} 4t\log\lp\Prob^x\lp \tau  \leq t\rp\rp \leq -\dist(x,A)^2 , 
\]
uniformly for $x\in K$.
\end{lemma}

\begin{proof}
Choose $\rho>0$ small enough so that the $\rho$-neighborhood of $\overline{U}$ has compact closure, Lemma \ref{App:ExitTime} holds with this $\rho$ for all $x\in \overline{U}$, and $\rho < \eps = \dist\lp \overline{U'},A \rp $. Now choose $a\in (0,\rho)$ and let (with slight abuse of notation) $\sigma_1$ be the first time $X_t$ moves a distance $a$ from its starting point, $\sigma_2$ the first time after $\sigma_1$ that $X_t$ moves a distance $a$ from $X_{\sigma_1}$, and so on for $\sigma_i$ with $i=3,4,\ldots$. If $m_x$ is the largest integer such that $am_x\leq \dist\lp x,A \rp$, we see that, if $X_0=x$, then $\tau \geq \sigma_{m_x}$.

Note that $\rho$, and hence $a$, is such that $m_x\geq 1$ for all $x\in K$, and further, the set of $x$ with $m_x=1$ is a compact subset of $K$. Then, from Lemma\ref{App:ExitTime} and the definition of $m_x$, we have that, for any (small) $\delta>0$, there exists some $t_0>0$, such that
\[
\Prob^x\lp \sigma_1  \leq t\rp \leq  \exp\lb -\frac{a^2-\delta}{4t}\rb 
\]
for all $t<t_0$ and all $x$ with $m_x=1$.

Now the set of $x\in K$ with $m_x=2$ is also compact, and by the strong Markov property, $\sigma_2-\sigma_1$ satisfies the same estimate as $\sigma_1$, and further, if $F_1$ is the cdf of $\sigma_1$, then taking convolution gives
\[
\Prob^x\lb \sigma_2\leq t \rb \leq \int_0^t \exp\lb -\frac{a^2-\delta}{4(t-s)} \rb\diff F_1(s)
\]
for all small enough $t$, uniformly in $x$ with $m_x=2$, where the integral is understood as a Lebesgue-Stieljes integral. (Of course, the $F_1$ depends on $X_{\sigma_1}$, but the bound we use is uniform over $X_{\sigma_1}$, so we don't emphasize this.) Now since $F_1$ is non-decreasing and has bounded variation and the integrand is non-increasing and continuous, has bounded variation, and is differentiable on $s\in(0,t)$, we have an integration by parts formula for the integral. The boundary terms vanish, since the integrand goes to zero as $s\nearrow t$ and $F_1(s)$ goes to zero as $s\searrow 0$, and we find
\[
\Prob^x\lb \sigma_2\leq t \rb \leq \int_0^t F_1(s) \frac{a^2-\delta}{4(t-s)^2} \exp\lb -\frac{a^2-\delta}{4(t-s)} \rb\diff s .
\]
As long as $t$ is small enough, we can absorb the $\frac{a^2-\delta}{4(t-s)^2}$ factor into the exponential at the cost of replacing $\delta$ with $2\delta$. This plus the previous estimate for $F_1$ gives
\[
\Prob^x\lb \sigma_2\leq t \rb \leq \int_0^t \exp\lb -\frac{a^2-\delta}{4s}-\frac{a^2-2\delta}{4(t-s)} \rb\diff s
\]
for all small enough $t$ and all $x$ with $m_x=2$. Using Lemma \ref{Lem:Calc} and the fact that $\delta$ is arbitrary (so we can reduce it as necessary), a naive estimate for the integral gives that
\[
\Prob^x\lb \sigma_2\leq t \rb \leq \exp\lb -\frac{(2a)^2-\delta}{4t} \rb
\]
for all $x$ with $m_x=2$ and all small enough $t$.

Because $K$ has finite diameter, $m_x$ is bounded on $K$, and thus, iterating the above argument a finite number of times, we have that, for $\delta>0$, there exists $t_0>0$ such that
\[
\Prob^x\lb \sigma_{m_x}\leq t \rb \leq \exp\lb -\frac{(m_xa)^2-\delta}{4t} \rb
\]
for all $x\in K$ and all $t<t_0$. By construction, we have $m_xa\leq \dist\lp x,A \rp< (m_x+1)a$ and  $\tau \geq \sigma_{m_x}$, and it follows that
\[
\Prob^x\lb \tau \leq t\rb \leq \Prob^x\lb \sigma_{m_x}\leq t \rb \leq  \exp\lb -\frac{(d(x,A)-a)^2-\delta}{4t}\rb ,
\]
again for all $x\in K$ and all $t<t_0$. Since $a$ and $\delta$ are arbitrary, the lemma follows.
\end{proof}

Choose any $y_0\in K$. Then in reference to Lemma \ref{Lem:Our7}, let $\eta>0$ be small enough so that the ball of radius $(7/2)\eta$ around $y_0$, which we denote $B''$, has its closure contained in $U'$, and let $B$ and $B'$ be as in Lemma \ref{Lem:Our7}. Let $\tilde{\tau}$ be the first hitting time of $B'$ after $\tau$. Further, let $\tau'_1$ be the first hitting time of $\overline{U'}$ after $\tau$, let $\tau_1$ be the first hitting time of $A$ after $\tau'_1$, and then recursively let $\tau'_i$ be the first hitting time of $\overline{U'}$ after $\tau_{i-1}$, let $\tau_i$ be the first hitting time of $A$ after $\tau'_i$, for $i=2,3,\ldots$. By construction, the process has to travel distance $\eps$ between $\tau_{i-1}$ and $\tau'_i$ and between $\tau'_i$ and $\tau_i$, and thus for any path, only finitely many of the $\tau_i$ and $\tau'_i$ can be less than any given $t$.

\begin{lemma}\label{Lem:TTBound}
For $y_0\in K$, consider the notation above. Then we have
\[
\limsup_{t\searrow0} 4t\log\lp\Prob^x\lp \tau  \leq \tilde{\tau} \leq t\rp\rp \leq -\lp \dist(x,A) +\dist(y_0,A)-\eps-(3/2)\eta\rp^2, 
\]
uniformly for $x\in K$.
\end{lemma}
\begin{proof}
For small enough $t$, Lemma \ref{Lem:SigA} (and the definition and discussion of $\tau_i$ and $\tau'_i$ above) implies that the probability that $\tau-\tau'_1$ is less than $t$ is less than than $1/2$, and more generally, the probability that $\tau'_{i+1}-\tau'_i$ is less than $t$ is less than $1/2$. This plus the strong Markov property implies that
\[\begin{split}
\Prob^x\lp \tilde{\tau}-\tau \leq t \rp  &\leq \sum_{i=1}^{\infty}  \Prob^{X_{\tau}}\lp \tau'_i \leq \tilde{\tau} \leq t< \tau'_{i+1}  \rp \\
& \leq \sum_{i=1}^{\infty} \lp\frac{1}{2}\rp^i \Prob^{X_{\tau'_i}}\lp \tilde{\tau} \leq t< \tau  \rp \\
& \leq \sup_{z\in \partial U'} \Prob^{z}\lp \tilde{\tau} \leq t< \tau  \rp .
\end{split}\]
Now since the distance from $\partial U'$ to $B'$ is $\dist(y_0,A)-\eps-(3/2)\eta$, we can argue just as in the proof of Lemma \ref{Lem:SigmaBound} (namely, the process has to exit some number of balls of radius $\delta$ contained in $U$, for all small enough $\delta$) to see that, for any $\delta>0$, there is $t_0>0$ such that, for any $x\in K$ and $t< t_0$,
\[
\Prob^x\lp \tilde{\tau}-\tau \leq t \rp \leq \sup_{z\in \partial U'} \Prob^{z}\lp \tilde{\tau} \leq t< \tau  \rp
\leq  \exp\lb -\frac{(\dist(y_0,A)-\eps-(3/2)\eta)^2-\delta}{4t}\rb .
\]

Since $\tau$ and $\tilde{\tau}-\tau$ are conditionally independent given $X_{\tau}$, and we have uniform upper bounds on the cdfs of both, taking the convolution gives, for all small enough $t$,
\[\begin{split}
& \Prob^x\lp \tau  \leq \tilde{\tau} \leq t\rp \leq 
 \int_0^t \exp\lb -\frac{d(x,A)^2-\delta}{4(t-s)} \rb\diff F(s) \\
&\quad \text{where} \quad F(t) = \exp\lb -\frac{(\dist(y_0,A)-\eps-(3/2)\eta)^2-\delta}{4t}\rb .
\end{split}\]
Again as in the proof of Lemma \ref{Lem:SigmaBound}, we can use Lemma \ref{Lem:Calc} to see that, for any $\delta>0$, there exists $t_0>0$ such that
\[
\Prob^x\lp \tau  \leq \tilde{\tau} \leq t\rp
\leq  \exp\lb -\frac{(\dist(x,A)+\dist(y_0,A)-\eps-(3/2)\eta)^2-\delta}{4t}\rb 
\]
for all $t<t_0$.
The conclusion of the lemma follows.
\end{proof}

The next step is to include the contribution to $p_t(x,A,y)$ from the piece of the path in $B''$.

\begin{lemma}\label{Lem:LastStep}
For any $y_0\in K$ and $\delta>0$, and any (small enough, so that the stopping times above are well defined) $\eps>0$ and $\eta>0$, there exists $t_0>0$ such that
\[
4t \log p_t(x,A,y) \leq - \lp d(x,A)+d(y,A)-\eps-3\eta-\delta \rp^2
\]
for any $t\in (0,t_0)$, $x\in K$, and $y$ with $\dist(y,y_0)<\eta/2$.
\end{lemma}

\begin{proof}
We introduce one more family of interlaced stopping times. Let $\theta'_1$ be the first hitting time of $B''$ after $\tilde{\tau}$, let $\theta_1$ be the first hitting time of $B'$ after $\theta'_i$, and then let $\theta'_i$ be the first hitting time of $B''$ after $\theta_{i-1}$ and $\theta_i$ the first hitting time of $B'$ after $\theta'_i$. Also, let $\tilde{\tau}=\theta_0$, for convenience. Then we have the path decomposition
\[
p_t(x,A,y)= \sum_{i=0}^{\infty} \Prob^x\lp \tau<\theta_i< t < \theta'_{i+1} \text{ and } X_t\in dy\rp
\]
for $x\in K$ and $y$ with $\dist(y,y_0)<\eta/2$. Because the process travels a distance $2\eta$ between $\theta_i$ and $\theta'_{i+1}$, just as above, we know that, for small enough $t$, the probability that $\theta_i$ is less than $t$ is less than $(1/2)^i$. Now for $t\in(0,\infty)$ and $z\in \partial B'$, let $\mu^i(t,z)$ be the spacetime hitting measure of $(\theta_i,X_{\theta_i})$. Then we use the strong Markov property to see that
\[\begin{split}
p_t(x,A,y) &\leq \sum_{i=0}^{\infty} \int_{s,z} p^{B''}_{t-s}(z,y)  \diff \mu^i(s,z) \\
&\leq 2 \int_{s=0}^t \sup_{z\in \partial B'} p^{B''}_{t-s}(z,y) \diff \Prob^x\lp \tau<\tilde{\tau}<s\rp ,
\end{split}\]
where the final integral is understood as a Lebesgue-Stieltjes integral (with $s$ the variable of integration). Using Lemma \ref{Lem:Our7} (with $\alpha$ the empty multinomial) and the triangle inequality, we have that, for any $\delta>0$, there exists $t_0>0$ such that
\[
\sup_{z\in \partial B'} p^{B''}_{t-s}(z,y) \leq \exp\lb -\frac{\eta^2-\delta}{4(t-s)}\rb 
\]
as long as $t-s<t_0$ and $\dist(y,y_0)<\eta/2$. Since we also have, after possibly shrinking $t_0$,
\[
\Prob^x\lp \tau  \leq \tilde{\tau} \leq s\rp
\leq  \exp\lb -\frac{(\dist(x,A)+\dist(y_0,A)-\eps-(3/2)\eta)^2-\delta}{4s}\rb
\]
by Lemma \ref{Lem:TTBound}, we again can use integration by parts and standard estimates on the integral to see that, for any $\delta>0$, there exists $t_0>0$ such that
\[
p_t(x,A,y) \leq \exp\lb -\frac{(\dist(x,A)+\dist(y_0,A)-\eps-(5/2)\eta)^2-\delta}{4t}\rb
\]
for any $x\in K$, $t< t_0$, and $y$ with $\dist(y,y_0)<\eta/2$. Finally, we can use the triangle inequality to replace $\dist(y_0,A)$ with $d(y,A)-\eta/2$, which proves the lemma.
\end{proof}

From here, we can finish the proof.

\begin{proof}[Proof of Theorem \ref{THM:StrongCond}]
From Lemma \ref{Lem:LastStep}, we know that, for all small enough $\eps$, $\eta$, and $\delta$, for any $y_0\in K$, there exists $t_0>0$ such that
\[
4t \log p_t(x,A,y) \leq - \lp d(x,A)+d(y,A)-\eps-3\eta-\delta \rp^2
\]
for any $t\in (0,t_0)$, $x\in K$, and $y$ with $\dist(y,y_0)<\eta/2$. By compactness of $K$, there exist finitely many such $y_0$ such that the balls of radius $\eta/2$ around them cover $K$, and thus we can find $t_0$ so that this estimate holds for all $y\in K$. Since 
$\eps$, $\eta$, and $\delta$ can be chosen arbitrarily small, the localization estimate on $p_t(x,A,y)$ follows.

Once we have the localization condition, the exact Varadhan asymptotics, namely $4t \log p_t(x,y)\rightarrow -\dist^2(x,y)$ uniformly for $(x,y)\in \KR$, follow from including an appropriate neighborhood of $\pi_1(\KR)$ and $\pi_2(\KR)$ in a compact manifold, assuming that we have this estimate uniformly on compact manifolds. To establish this for compact manifolds, we use the same argument as in the proof of Theorem \ref{THM:LeandreCompact}. In particular, we note that the proof only uses the localization estimate from Theorem \ref{THM:StrongCond}, which we just proved. Thus, we can follow the argument exactly (and using the same notation) until \eqref{Eqn:SigmaC}, which we replace with
\[
\limsup_{t\searrow 0} 4t \log \lp  p^{\bR^{d+n}}_t \lp(x,0), \Sigma_s^c, (y,0)\rp \rp \leq - \lp d_M(x,y)+\delta \rp^2 ,
\]
uniformly for $x,y\in M$, using the localization bound just proven. Then since L\'eandre showed the Varadhan asymptotics are valid uniformly on compact subsets of $\bR^{d+n}\times \bR^{d+n}$, in place of \eqref{Eqn:Rdn}, we have
\[
\limsup_{t\searrow 0} 4t \log \lp  p_t^{\bR^{d+n}} \lp(x,0),(y,0)\rp\rp = - d_M^2(x,y)
\]
uniformly for $x,y\in M$. But then the usual decomposition implies that the same holds with $p_t^{\bR^{d+n}} \lp(x,0),(y,0)\rp$ replaced by $p_t^{\Sigma_s} \lp(x,0),(y,0)\rp$. From the product structure on $\Sigma_s$, we then see that
\[
\limsup_{t\searrow 0} 4t \log \lp  p_t^{M} \lp x,y\rp\rp + \limsup_{t\searrow 0} 4t \log \lp  p_t^{B(0,s)} \lp 0,0\rp\rp = - d_M^2(x,y)
\]
uniformly for $x,y\in M$. But, as noted in the proof of Theorem  \ref{THM:LeandreCompact}, we have
\[
\limsup_{t\searrow 0} 4t \log \lp  p_t^{B(0,s)} \lp 0,0\rp\rp =0
\]
and the result, namely the Varadhan asymptotics for compact $M$, follows.

Once we have the Varadhan asymptotics on compact manifolds and the localization estimate, the remainder of Theorem  \ref{THM:StrongCond} follows by gluing an appropriate neighborhood of $\pi_1(\KR)\cup \pi_2(\KR)$ into a compact manifold, given by a smooth doubling construction, just as in Step 2 of the proof of Theorem \ref{THM:TrueLocal}. Namely, we can find an open set $U\subset M$ with compact closure such that $\pi_1(\KR)\cup \pi_2(\KR)\subset U$ and such that (using the localization estimate), for some $\delta>0$,
\[\begin{split}
p^{M}_t(x,y) = p^U_t(x,y)+ p^{M}_t\lp x,U^c ,y\rp , \\
\text{with}\quad 
\limsup_{t\searrow 0} 4t \log \lp p^{M}_t\lp x,U^c ,y\rp  \rp & \leq -\lp d^2(x,y)+\delta\rp \quad\text{uniformly for $(x,y)\in \KR$.}
\end{split}\]
Moreover, $U$ can be chosen such that it can be included in a compact $\tilde{M}$ such that $p_t^U(x,y)$ is the same whether $U$ is understood as a subset of $M$ or of $\tilde{M}$, $d_M(x,y)=d_{\tilde{M}}(x,y)$ for all $(x,y)\in\KR$, and (again using the localization estimate)
\[\begin{split}
p^{\tilde{M}}_t(x,y) = p^U_t(x,y)+ p^{\tilde{M}}_t\lp x,U^c ,y\rp , \\
\text{with}\quad 
\limsup_{t\searrow 0} 4t \log \lp p^{\tilde{M}}_t\lp x,U^c ,y\rp  \rp & \leq -\lp d^2(x,y)+\delta\rp \quad\text{uniformly for $(x,y)\in \KR$.}
\end{split}\]
Since $\tilde{M}$ is compact, we know that
\[
\lim_{t\searrow 0} 4t \log \lp p^{\tilde{M}}_t\lp x,y\rp  \rp  = - d^2(x,y)
\]
uniformly on all of $\tilde{M}$, and in particular, for $(x,y)\in\KR$, in which case the $d^2(x,y)$ on the right-hand side is unambiguous. Then combining all of this, just as in Step 2 of the proof of Theorem \ref{THM:TrueLocal}, we conclude that
\[
\lim_{t\searrow 0} 4t \log \lp p_t\lp x,y\rp  \rp \lim_{t\searrow 0} 4t \log \lp p^{\tilde{M}}_t\lp x,y\rp  \rp  = - d^2(x,y) ,
\]
uniformly for $(x,y)\in \KR$.

\end{proof}

Finally, the same ideas can be used to prove Lemma \ref{Lem:UniformBN}.

\begin{proof}[Proof of Lemma \ref{Lem:UniformBN}]

Assuming $x$ and $y$ are as in the Lemma, we first show the existence of $\eps>0$ as claimed, under either localization condition. We already know that for small enough $\eps$, $\Gamma_{\eps}$ is compact and there is a unique, strongly normal minimizer from $x$ to $z$ and from $z$ to $y$ for all $z$ in a neighborhood of $\Gamma_{\eps}$.

Next, we consider the localizability of $\KR_1$ and $\KR_2$. In fact, in the proof, we will need larger compact sets, so let
\[
\widehat{\KR}_1=\lc (q,z): z\in \Gamma_{\eps} \text { and } d(x,q)+d(q,z)\leq d(x,z)+\eps' \rc
 \quad\text{and}\quad \widehat{\KR}_2=\lc (z,y): z\in \Gamma_{\eps}\rc .
\]
We claim that for small enough $\eps$ and $\eps'$, each of these is compact and satisfies the same localization condition as $(x,y)$.

First, suppose that $(x,y)$ satisfies the strong localization condition, which exactly means that there exists $\alpha>0$ such that 
\[
d(x,y)+\alpha < d(x,\infty)+d(y,\infty) .
\]
By the triangle inequality, we see that
\[
d(q,\infty)\geq d(x,\infty)-d(q,x) \quad\text{and} \quad  d(z,\infty)\geq d(y,\infty)-d(z,y) ,
\]
and from the definition of $\widehat{\KR}_1$, we have $d(q,z)\leq d(x,z)-d(x,q)+\eps'$. Using this, we have
\[
d(q,\infty)+d(z,\infty)-d(q,z) \geq d(x,\infty)+d(y,\infty)-d(x,z)-d(z,y)-\eps' .
\]
By the definition of $\Gamma_{\eps}$, we have $d(x,z)+d(z,y)\leq d(x,y)+\eps$. Using this plus the strong localization condition in the above gives
\[
d(q,\infty)+d(z,\infty)-d(q,z) \geq \alpha-\eps-\eps' .
\]
Choosing $\eps$ and $\eps'$ small enough so that the right-hand side is positive, and observing that this inequality holds for all $(q,z)\in\widehat{\KR}_1$, means that $\widehat{\KR}_1$ satisfies the strong localization, and then the compactness is clear.

On the other hand, suppose that $(x,y)$ satisfies the weak localization condition, so that for some $\alpha>0$
\[
U=\lc p : d(x,p)+d(p,y)< d(x,y)+\alpha\rc
\]
has compact closure. We wish to show that if $\eps$ and $\eps'$ are small enough, there will exist $\eps''>0$ such that, for any $(q,z)\in \widehat{\KR}_1$, the set
\[
V = \lc p : d(q,p)+d(p,z)< d(q,z)+\eps''\rc
\]
will be a subset of $U$. We take any such $q$ and $z$, and start by using the triangle inequality to write
\[
d(x,p)+d(p,y) \leq d(x,q)+d(q,p)+d(p,z)+d(z,y) .
\]
Using, on the right-hand side, that $p\in V$, then that $q\in \widehat{\KR}_1$, and then that $z\in \Gamma_{\eps}$, we find
\[
d(x,p)+d(p,y) \leq \eps+\eps'+\eps'' .
\]
This shows that if $\eps$ and $\eps'$ are small enough, we can find $\eps''>0$ so that $V\subset U$. But this means that $V$ has compact closure, since $U$ does. Since the sector condition is global, this shows that the resulting $\widehat{\KR}_1$ satisfies the weak localization condition, and again it is also compact.

Now $\KR_1$ is a closed subset of the compact $\widehat{\KR}_1$, so $\KR_1$ is compact and localizable. Since the distance function as well as the distance inequalities in both localization conditions are symmetric, the argument for $\hat{\KR}_2$ and $\KR_2$ is the same. Thus, in what follows, we assume $\eps$ and $\eps'$ are chosen to make $\widehat{\KR}_1$ and $\widehat{\KR}_2$ compact and localizable. 

We move on to establishing the weak convergence of $\mu^{x,z,t}$.
For $\delta>0$ and any $z\in \Gamma_{\eps}$, let $\sigma^{\delta,z,t}=\sigma$ be the first time $d(X_s,g^{x,z}_{s/t})$ hits $\delta$. For small enough $\delta$, the triangle inequality implies that $X_s$ is contained in $\hat{\KR}_1$ for $s\in[0,\sigma]$, for any $z\in\Gamma_{\eps}$, and we assume $\delta$ is sufficiently small to satisfy this condition. It follows from the finite-dimensional distributions of the bridge process and the rescaling between $t$ and $\tau$ that 
\begin{equation}\label{Eqn:BasicTube}
\mu^{x,z,t}\Big( \dist_M\lp \omega_{\tau}, g^{x,z}_{\tau}\rp<\delta \text{ for all $\tau\in[0,1]$} \Big) =
1- \frac{\Prob^x\lp X_t\in \diff z \text{ and } \sigma<t   \rp}{p_t(x,z)} .
\end{equation}
Because $(x,z)\in \widehat{\KR}_1$, we know that $p_t(x,z)= e^{-\frac{d^2(x,z)+o(1)}{4t}}$ uniformly in $z$. So the point is to estimate $\Prob^x\lp X_t\in \diff z \text{ and } \sigma<t   \rp$ to be asymptotically smaller than this, uniformly in $z$.

Consider $\frac{d^2(x,q)}{4s}+\frac{d^2(q,z)}{4(t-s)}$ as a function of $q\in M$ and $s\in[0,t]$. By Lemma \ref{Lem:Calc}, this has minimum of $\frac{d^2(x,y)}{4t}$, achieved exactly when $q=g^{x,z}_{s/t}$ for each $s$. But by the definition of $\sigma$, the points $s=\sigma$ and $q=X_{\sigma}$ avoid this minimum. This, plus smoothness and compactness and the scaling in $t$, implies that there exists $\eta>0$, depending on $\delta$ but not on $z$ or $t$, such that 
\[
\frac{d^2\lp x,X_{\sigma}\rp}{4\sigma}+\frac{d^2\lp X_{\sigma},z\rp}{4(t-\sigma)}  > \frac{d^2(x,z)+4\eta}{4t} .
\]
Below, we will need to discretize $d(x,X_{\sigma})$. To this end, for some $\rho>0$ (and smaller than $\delta$), let $k=k(d(x,X_{\sigma}))$  be the largest integer such that $k\rho \leq d(x,X_{\sigma}) < (k+1)\rho$. We see that
\[
\rho k\lp d\lp x,X_{\sigma}\rp\rp+ d\lp X_{\sigma},z\rp > d\lp x,X_{\sigma}\rp+ d\lp X_{\sigma},z\rp -\rho 
\]
Thus by continuity (and Lemma \ref{Lem:Calc} and the scaling in $t$, again), we can choose $\rho$ small enough relative to $\delta$ and $\eta$, so that we have the discretized version of the above, namely,
\[
\frac{\lb \rho k\lp d\lp x,X_{\sigma}\rp\rp\rb^2}{4\sigma}+\frac{d^2\lp X_{\sigma},z\rp}{4(t-\sigma)}  > \frac{d^2(x,z)+3\eta}{4t} .
\]
Moreover, because $\lp X_{\sigma},z\rp\in \widehat{\KR}_1$, we see that, for small enough $t$,
\begin{equation}\label{Eqn:ptFromSigma}
p_{t-\sigma}\lp X_{\sigma},z\rp  < \exp\lb -\lp  \frac{d^2(x,z)+3\eta}{4t} - \frac{\lb \rho k\lp d\lp x,X_{\sigma}\rp\rp\rb^2}{4\sigma} \rp \rb
\end{equation}
for all $z\in \Gamma_{\eps}$.

As usual, we decompose $X_t$ according to $\sigma$, so that
\[
\Prob^x\lp X_t\in \diff z \text{ and } \sigma<t  \rp = \int\limits_{\substack{s\in [0,t] \\ q\in M}} p_{t-s}(q,z) \diff \mu^{\sigma^{\delta,z,t}}(s,q)
\]
where $\mu^{\sigma^{\delta,z,t}}(s,q)$ is the joint distribution of $s=\sigma^{\delta,z,t}$ and $q=X_{\sigma^{\delta,z,t}}$ under $\Prob^x$ (of course, this is a sub-probability distribution). We now partition the integral according to $k(d(x,X_{\sigma}))$. Since $\rho$ is fixed (given $\delta$ and $\eta$) and $\widehat{\KR}_1$ has finite diameter, we have an a priori bound on $k$, say $N$. Now let $F^{\sigma, k}$ be the (defective) cdf of $\sigma$ on the event $\{k\rho \leq d(x,X_{\sigma}) < (k+1)\rho\}$. Then, using \eqref{Eqn:ptFromSigma}, we have, for small enough $t$,
\[
\Prob^x\lp X_t\in \diff z \text{ and } \sigma<t  \rp < \sum_{k=0}^N
\int_{s=0}^t  \exp\lb -\lp  \frac{d^2(x,z)+3\eta}{4t} - \frac{\lp \rho k\rp^2}{4s} \rp \rb \diff F^{\sigma,k}(s)
\]
for all $z\in \Gamma_{\eps}$.

Next, we need a uniform estimate on $F^{\sigma,k}$. The point is that $F^{\sigma,k}(s)$ is less than or equal to the probability that $X_t$ travels a distance at least $k\rho$ from its starting point, in time less than or equal to $s$, all while staying inside $\widehat{\KR}_1$ (recall that $\delta$ is small enough that $X_t$ is contained in $\hat{\KR}_1$ for $t\in[0,\sigma]$ for all $z\in\Gamma_{\eps}$). And because $\hat{\KR}_1$ is compact, by Lemma \ref{App:ExitTime}, the probability of leaving small balls in small time is uniformly bounded.  Then, just as in the proof of Lemma \ref{Lem:SigmaBound}, it follows that, for small enough $t$,
\[
F^{\sigma,k}(s) <  \exp\lb -\frac{(k\rho)^2-\eta}{4s}  \rb 
\]
for all $s\in(0,t]$ and $z\in\Gamma_{\eps}$. Then we can again use integration by parts (and absorb all of the sub-exponential factors at the cost of losing ``one more $\eta$'') to see that, for small enough $t$,
\[\begin{split}
\Prob^x\lp X_t\in \diff z \text{ and } \sigma<t  \rp &< \sum_{k=0}^N \exp\lb -  \frac{d^2(x,z)+\eta}{4t} \rb \\
&= (N+1)  \exp\lb -  \frac{d^2(x,z)+\eta}{4t} \rb
\end{split}\]
for all $z\in \Gamma_{\eps}$.

Combining this with \eqref{Eqn:BasicTube} and the uniform asymptotics of $p_t(x,z)$, we see that for all small enough $t$, we can make
\[
\mu^{x,z,t}\Big( \dist_M\lp \omega_{\tau}, g^{x,z}_{\tau}\rp<\delta \text{ for all $\tau\in[0,1]$} \Big)
\]
as close to 1 as desired, for all $z\in \Gamma_{\eps}$. In particular, it can be made greater than $1-\delta$. Since $\delta$ was arbitrarily small, in light of the characterization of weak convergence to the point mass at $g^{x,z}$ just after the statement of Lemma \ref{Lem:UniformBN}, this proves the desired convergence of $\mu^{x,z,t}$. The argument for $\mu^{z,y,t}$ is completely analogous, completing the proof of the lemma.

\end{proof}

%%%%%%%%%%%%%%%%%%%%%%%%%%%%%%%%%%%%%%%%%%%%%%%%%%%%%%%%%%%

\providecommand{\bysame}{\leavevmode\hbox to3em{\hrulefill}\thinspace}
\providecommand{\MR}{\relax\ifhmode\unskip\space\fi MR }
\providecommand{\MRhref}[2]{%
  \href{http://www.ams.org/mathscinet-getitem?mr=#1}{#2}
}
\providecommand{\href}[2]{#2}


\begin{thebibliography}{10}

\bibitem{agrachev_1996_exponential}
A.~A. Agrachev, \emph{Exponential mappings for contact sub-{R}iemannian
  structures}, J. Dynam. Control Systems \textbf{2} (1996), no.~3, 321--358.
  \MR{1403262}

\bibitem{Agrachev-RA}
A.~A. Agrach\"ev, \emph{Any sub-{R}iemannian metric has points of smoothness},
  Dokl. Akad. Nauk \textbf{424} (2009), no.~3, 295--298. \MR{2513150}

\bibitem{ABB_2018}
Andrei Agrachev, Davide Barilari, and Ugo Boscain, \emph{A comprehensive
  introduction to sub-{R}iemannian geometry}, Cambridge University Press, 2018.

\bibitem{Arnold_singularities_diff_maps1}
V.~I. Arnold, S.~M. Guse\u\i~n Zade, and A.~N. Varchenko, \emph{Singularities
  of differentiable maps. {V}ol. {I}}, Monographs in Mathematics, vol.~82,
  Birkh\"auser Boston, Inc., Boston, MA, 1985, The classification of critical
  points, caustics and wave fronts, Translated from the Russian by Ian Porteous
  and Mark Reynolds. \MR{777682}

\bibitem{Arnold_singularities_diff_maps2}
\bysame, \emph{Singularities of differentiable maps. {V}ol. {II}}, Monographs
  in Mathematics, vol.~82, Birkh\"auser Boston, Inc., Boston, MA, 1985, The
  classification of critical points, caustics and wave fronts, Translated from
  the Russian by Ian Porteous and Mark Reynolds. \MR{777682}

\bibitem{MashaMalva}
Malva Asaad and Maria Gordina, \emph{Hypoelliptic heat kernels on nilpotent
  {L}ie groups}, Potential Anal. \textbf{45} (2016), no.~2, 355--386.
  \MR{3518678}

\bibitem{Azencott}
Robert Azencott, \emph{Un probl\`eme pos\'e par le passage des estim\'ees
  locales aux estim\'ees globales pour la densit\'e d'une diffusion},
  Asterisque (1981), no.~84, 131--150.

\bibitem{Bailleul-Bridge}
Isma\"{e}l Bailleul, \emph{Large deviation principle for bridges of
  sub-{R}iemannian diffusion processes}, S\'{e}minaire de {P}robabilit\'{e}s
  {XLVIII}, Lecture Notes in Math., vol. 2168, Springer, Cham, 2016,
  pp.~189--198. \MR{3618130}

\bibitem{BMN}
Isma\"el Bailleul, Laurent Mesnager, and James Norris, \emph{Small-time
  fluctuations for the bridge of a sub-{R}iemannian diffusion}, Ann. Sci. \'Ec.
  Norm. Sup\'er. (4) \textbf{54} (2021), no.~3, 549--586. \MR{4311094}

\bibitem{NorrisB}
Ismael Bailleul and James Norris, \emph{Diffusion in small time in incomplete
  sub-{R}iemannian manifolds}, Anal. PDE \textbf{15} (2022), no.~1, 63--84.
  \MR{4395153}

\bibitem{BanyagaHurtubise}
Augustin Banyaga and David~E. Hurtubise, \emph{A proof of the {M}orse-{B}ott
  lemma}, Expo. Math. \textbf{22} (2004), no.~4, 365--373. \MR{2075744}

\bibitem{DavideTrace}
D.~Barilari, \emph{Trace heat kernel asymptotics in 3{D} contact
  sub-{R}iemannian geometry}, J. Math. Sci. (N.Y.) \textbf{195} (2013), no.~3,
  391--411, Translation of Sovrem. Mat. Prilozh. No. 82 (2012). \MR{3207127}

\bibitem{BBCN-IMRN}
Davide Barilari, Ugo Boscain, Gr{\'e}goire Charlot, and Robert~W Neel, \emph{On
  the heat diffusion for generic {R}iemannian and sub-{R}iemannian structures},
  International Mathematics Research Notices \textbf{2017} (2016), no.~15,
  4639--4672.

\bibitem{BBN-JDG}
Davide Barilari, Ugo Boscain, and Robert~W. Neel, \emph{Small-time heat kernel
  asymptotics at the sub-{R}iemannian cut locus}, J. Differential Geom.
  \textbf{92} (2012), no.~3, 373--416. \MR{3005058}

\bibitem{BBN-BiHeis}
\bysame, \emph{Heat kernel asymptotics on sub-{R}iemannian manifolds with
  symmetries and applications to the bi-{H}eisenberg group}, Ann. Fac. Sci.
  Toulouse Math. (6) \textbf{28} (2019), no.~4, 707--732. \MR{4045424}

\bibitem{LucaDavide}
Davide Barilari and Luca Rizzi, \emph{Sub-{R}iemannian interpolation
  inequalities}, Invent. Math. \textbf{215} (2019), no.~3, 977--1038.
  \MR{3935035}

\bibitem{FabriceMichel}
Fabrice Baudoin and Michel Bonnefont, \emph{The subelliptic heat kernel on
  {${\rm SU}(2)$}: representations, asymptotics and gradient bounds}, Math. Z.
  \textbf{263} (2009), no.~3, 647--672. \MR{2545862}

\bibitem{FabriceJing}
Fabrice Baudoin and Jing Wang, \emph{The subelliptic heat kernel on the {CR}
  sphere}, Math. Z. \textbf{275} (2013), no.~1-2, 135--150. \MR{3101801}

\bibitem{ABell}
Andre Bellaiche, \emph{Propri\'et\'es extr\'emales des g\'eod\'esiques},
  Asterisque (1981), no.~84, 83--130.

\bibitem{CBell}
Catherine Bellaiche, \emph{Comportement asymptotique de $p(t,x,y)$ quand $t
  \rightarrow0$ (points \'eloign\'es)}, Asterisque (1981), no.~84, 151--188.

\bibitem{BenArous}
G.~Ben~Arous, \emph{D\'eveloppement asymptotique du noyau de la chaleur
  hypoelliptique hors du cut-locus}, Ann. Sci. \'Ecole Norm. Sup. (4)
  \textbf{21} (1988), no.~3, 307--331. \MR{MR974408 (89k:60087)}

\bibitem{BN-Grushin}
Ugo Boscain and Robert~W. Neel, \emph{Extensions of {B}rownian motion to a
  family of {G}rushin-type singularities}, Electron. Commun. Probab.
  \textbf{25} (2020), Paper No. 29, 12. \MR{4089736}

\bibitem{UgoDario}
Ugo Boscain and Dario Prandi, \emph{Self-adjoint extensions and stochastic
  completeness of the {L}aplace-{B}eltrami operator on conic and anticonic
  surfaces}, J. Differential Equations \textbf{260} (2016), no.~4, 3234--3269.
  \MR{3434398}

\bibitem{XM-GenComplete}
Xin Chen, Xue-Mei Li, and Bo~Wu, \emph{Logarithmic heat kernel estimates
  without curvature restrictions}, Ann. Probab. \textbf{51} (2023), no.~2,
  442--477. \MR{4546623}

\bibitem{Gauthier_1996_small_SR_balls}
El-H.~Ch. El-Alaoui, J.P.A. Gauthier, and I.~Kupka, \emph{Small
  sub-{R}iemannian balls on {$\mathbf{R}^3$}}, J. Dynam. Control Systems
  \textbf{2} (1996), no.~3, 359--421. \MR{1403263}

\bibitem{Nate}
Nathaniel Eldredge, \emph{Precise estimates for the subelliptic heat kernel on
  {$H$}-type groups}, J. Math. Pures Appl. (9) \textbf{92} (2009), no.~1,
  52--85. \MR{2541147}

\bibitem{EAndK}
Ricardo Estrada and Ram~P. Kanwal, \emph{A distributional approach to
  asymptotics}, second ed., Birkh\"auser Advanced Texts: Basler Lehrb\"ucher.
  [Birkh\"auser Advanced Texts: Basel Textbooks], Birkh\"auser Boston Inc.,
  Boston, MA, 2002, Theory and applications. \MR{2002k:46096}

\bibitem{Gallone2}
Matteo Gallone and Alessandro Michelangeli, \emph{Quantum particle across
  {G}rushin singularity}, J. Phys. A \textbf{54} (2021), no.~21, Paper No.
  215201, 42. \MR{4271283}

\bibitem{Gallone1}
Matteo Gallone, Alessandro Michelangeli, and Eugenio Pozzoli, \emph{Quantum
  geometric confinement and dynamical transmission in {G}rushin cylinder}, Rev.
  Math. Phys. \textbf{34} (2022), no.~7, Paper No. 2250018, 91. \MR{4471194}

\bibitem{Habermann}
Karen Habermann, \emph{Small-time fluctuations for sub-{R}iemannian diffusion
  loops}, Probab. Theory Related Fields \textbf{171} (2018), no.~3-4, 617--652.
  \MR{3827218}

\bibitem{HsuLocal}
Elton~P. Hsu, \emph{On the principle of not feeling the boundary for diffusion
  processes}, J. London Math. Soc. (2) \textbf{51} (1995), no.~2, 373--382.
  \MR{1325580}

\bibitem{HsuLogDer}
\bysame, \emph{Estimates of derivatives of the heat kernel on a compact
  {R}iemannian manifold}, Proc. Amer. Math. Soc. \textbf{127} (1999), no.~12,
  3739--3744. \MR{1618694}

\bibitem{Hsu}
\bysame, \emph{Stochastic analysis on manifolds}, Graduate Studies in
  Mathematics, vol.~38, American Mathematical Society, Providence, RI, 2002.
  \MR{2003c:58026}

\bibitem{HsuBridge}
Pei Hsu, \emph{Brownian bridges on {R}iemannian manifolds}, Probab. Theory
  Related Fields \textbf{84} (1990), no.~1, 103--118. \MR{1027823}

\bibitem{HsuIncomplete}
\bysame, \emph{Heat kernel on noncomplete manifolds}, Indiana Univ. Math. J.
  \textbf{39} (1990), no.~2, 431--442. \MR{1089046}

\bibitem{Inahama-LargeDeviations}
Yuzuru Inahama, \emph{Large deviations for rough path lifts of {W}atanabe's
  pullbacks of delta functions}, Int. Math. Res. Not. IMRN (2016), no.~20,
  6378--6414. \MR{3579967}

\bibitem{Inahama}
Yuzuru Inahama and Setsuo Taniguchi, \emph{Short time full asymptotic expansion
  of hypoelliptic heat kernel at the cut locus}, Forum Math. Sigma \textbf{5}
  (2017), Paper No. e16, 74. \MR{3669328}

\bibitem{StroockKusuoka}
Shigeo Kusuoka and Daniel~W. Stroock, \emph{Asymptotics of certain {W}iener
  functionals with degenerate extrema}, Comm. Pure Appl. Math. \textbf{47}
  (1994), no.~4, 477--501. \MR{1272385}

\bibitem{leandremaj}
R{\'e}mi L{\'e}andre, \emph{Majoration en temps petit de la densit\'e d'une
  diffusion d\'eg\'en\'er\'ee}, Probab. Theory Related Fields \textbf{74}
  (1987), no.~2, 289--294. \MR{871256 (88c:60144)}

\bibitem{leandreminoration}
\bysame, \emph{Minoration en temps petit de la densit\'e d'une diffusion
  d\'eg\'en\'er\'ee}, J. Funct. Anal. \textbf{74} (1987), no.~2, 399--414.
  \MR{904825 (88k:60147)}

\bibitem{LeeBook}
John~M. Lee, \emph{Introduction to smooth manifolds}, second ed., Graduate
  Texts in Mathematics, vol. 218, Springer, New York, 2013. \MR{2954043}

\bibitem{HQLi}
Hong-Quan Li, \emph{Estimations asymptotiques du noyau de la chaleur sur les
  groupes de {H}eisenberg}, C. R. Math. Acad. Sci. Paris \textbf{344} (2007),
  no.~8, 497--502. \MR{2324485}

\bibitem{Ludewig2}
Matthias Ludewig, \emph{Heat kernel asymptotics, path integrals and
  infinite-dimensional determinants}, J. Geom. Phys. \textbf{131} (2018),
  66--88. \MR{3815228}

\bibitem{Ludewig1}
\bysame, \emph{Strong short-time asymptotics and convolution approximation of
  the heat kernel}, Ann. Global Anal. Geom. \textbf{55} (2019), no.~2,
  371--394. \MR{3923544}

\bibitem{Molchanov}
S.~A. Mol{\v{c}}anov, \emph{Diffusion processes, and {R}iemannian geometry},
  Uspehi Mat. Nauk \textbf{30} (1975), no.~1(181), 3--59. \MR{MR0413289 (54
  \#1404)}

\bibitem{OurAIHP}
R.~Neel and L.~Sacchelli, \emph{Localized bounds on log-derivatives of the heat
  kernel on incomplete {R}iemannian manifolds}, arXiv:2212.09559, to appear in
  Ann.\ Inst.\ Henri Poincar\'e Probab.\ Stat. (2024).

\bibitem{Neel}
Robert Neel, \emph{The small-time asymptotics of the heat kernel at the cut
  locus}, Comm. Anal. Geom. \textbf{15} (2007), no.~4, 845--890. \MR{MR2395259}

\bibitem{Rioul}
O.~{Rioul}, \emph{Information theoretic proofs of entropy power inequalities},
  IEEE Transactions on Information Theory \textbf{57} (2011), no.~1, 33--55.

\bibitem{StroockTuretsky}
Daniel~W. Stroock and James Turetsky, \emph{Upper bounds on derivatives of the
  logarithm of the heat kernel}, Comm. Anal. Geom. \textbf{6} (1998), no.~4,
  669--685. \MR{1664888}

\end{thebibliography}
\end{document}